\documentclass{article}

\usepackage{setspace}
\usepackage{enumerate}
\usepackage{amsmath}
\usepackage{amsthm}
\usepackage{amssymb}
\usepackage{amsfonts}
\usepackage[dvips]{graphicx, color}

\theoremstyle{plain}
\newtheorem{theorem}{Theorem}[section]
\theoremstyle{plain}
\newtheorem{prop}{Proposition}[section]
\theoremstyle{plain}
\newtheorem{lemma}{Lemma}[section]
\theoremstyle{plain}
\newtheorem{cor}{Corollary}[section]
\theoremstyle{plain}
\newtheorem{rem}{Remark}[section]
\theoremstyle{plain}
\newtheorem{assumption}{}

\begin{document}

\title{Inverse scattering theory for discrete Schr\"{o}dinger operators on the
hexagonal lattice}

\author{Kazunori Ando\\
  \texttt{\it{Graduate school of pure and applied science, University of Tsukuba}}\\
  \texttt{\it{Tennoudai, Tsukuba, Ibaraki 305-8571, Japan}}\\
  \texttt{\it{email: ando@math.tsukuba.ac.jp}}}

\maketitle

\begin{abstract}
We consider the spectral theory for discrete Schr\"{o}dinger operators on the
hexagonal lattice and their inverse scattering problem. We give a procedure
for reconstructing the compactly supported potential from the scattering
matrix for all energies. The same procedure is applicable for the inverse
scattering problem on the triangle lattice.
\end{abstract}

\medskip{\it{$2000$ Mathematics Subject Classification{\rm :}}}
35P25, 39A12, 47B39, 81U40

%
%
\section{Introduction}

Let us begin with recalling the study of inverse scattering problems for
Schr\"{o}dinger operators on $\mathbb{R}^n$. For $n=1$, it was solved by
Gel'fand and Levitan \cite{MR0045281}, Mar{\v{c}}henko \cite{MR0075402} in
1950's. For $n=3$, Faddeev \cite{MR0080533} established the uniqueness of the
potential with given scattering matrix by using high-energy Born approximation
soon after that. In the 1970's, Faddeev \cite{springerlink:10.1007/BF01083780}
and Newton \cite{new:1974} investigated an analogue of Gel'fand-Levitan theory
for the multidimensional case by making use of Faddeev's Green operator. In
the late 1980's, there was a breakthrough, which is called
$\overline{\partial}$-approach, in this direction by Beals and Coifman
\cite{MR812283}, Nachman and Ablowitz \cite{MR769078}, Khenkin and Novikov
\cite{MR896879}, Weder \cite{MR1108287}. See e.g. a survey article of Isozaki
\cite{MR2117108}.

In discrete settings, it was pointed out early that the theory of
Gel'fand-Levitan-Mar{\v{c}}henko is applicable for the discrete
Schr\"{o}dinger operator on $\mathbb{Z}$ (e.g. Case and Kac
\cite{MR0332065}). On $\mathbb{Z}^n$, $n\ge 2$, Isozaki and Korotyaev
\cite{2011arXiv1103.2193I} studied the inverse scattering problem for the
discrete Schr\"{o}dinger operator with a compactly supported potential,
recently. They used a kind of complex Born approximation, and derived a
procedure for reconstructing the potential. Furthermore, they got a partial
result about the inverse scattering problem from a fixed energy by using a
discrete analogue of Faddeev's Green operator.

In this paper, we consider the inverse scattering problem for the discrete
Schr\"{o}dinger operator on the hexagonal lattice, and derive a reconstruction
procedure for the potential from the scattering matrix of all energies. The
hexagonal lattice is a sort of two dimensional lattice, which covers the
plane by equilateral hexagons with honeycomb structure. We regard the
hexagonal lattice as graph. Before stating our main theorem, let us briefly
recall the graph theory.

We denote by $G=(V(G),E(G))$ the graph that consists of a vertex set $V(G)$,
whose cardinality is at most countable, and an edge set $E(G)$, each element
of which connects a pair of vertices. Let $v$, $u\in{}V(G)$ and
$e\in{}E(G)$. We denote by $v\stackrel{e}\sim{}u$, or simply $v\sim{}u$, when
$v$ is adjacent to $u$ by $e$; by $N_v=\{u\in{}V(G);v\sim{}u\}$ the set of
vertices which are adjacent to $v$. We also denote by $\deg(v)=\sharp{N_v}$
the degree of $v$. We assume that the graph $G$ is connected, which implies
that $\deg(v)>0$ for any $v\in{}V(G)$, locally finite, that is,
$\deg(v)<\infty$ for any $v\in{}V(G)$, and simple, that is, there are neither
self-loops nor multiple edges. Here, a self-loop is an edge which joins a
vertex to itself, and multiple edges are two or more edges which join the same
two vertices. The discrete Laplacian $\Delta_d$ on $G$ is defined as
\begin{equation*}
    (\Delta_d\hat{f})(v)=\frac{1}{\deg(v)}\sum_{u{\in}N_v}\hat{f}(u)-\hat{f}(v)
\end{equation*}
for the function $\hat{f}$ on $V(G)$. It is well-known that $-\Delta_d$ is a
bounded, self-adjoint operator on
\begin{equation*}
    l^2(G)=\{\hat{f}:V(G)\rightarrow\mathbb{C};\Vert\hat{f}\Vert_{l^2(G)}^2=\sum_{v{\in}V(G)}\vert\hat{f}(v)\vert^{2}\deg(v)<\infty\},
\end{equation*}
and $\sigma(-\Delta_d)\subset[0,2]$. See e.g. Chung \cite{MR1421568}.

Let $G$ be the hexagonal lattice. Figure \ref{fig_hexagonal_lattice_0}
illustrates the hexagonal lattice, where the vertices are represented as
big black and white dots and the edges as segments between two of them.
\begin{figure}[htbp]
    \begin{center}
        \includegraphics[height=3.5cm]{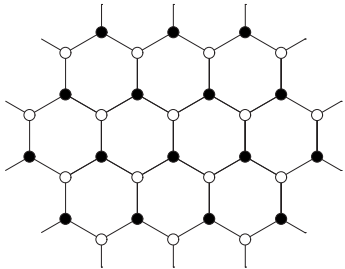}
    \end{center}
    \caption{the hexagonal lattice as graph \label{fig_hexagonal_lattice_0}}
\end{figure}
For simplicity's sake, we adopt as free Hamiltonian the operator which is
unitarily equivalent to $3(\Delta_d+1)$ on the hexagonal lattice. See Section
2 for the details. The discrete Schr\"{o}dinger operator is defined as the
free Hamiltonian plus the potential $\hat{q}$, which is a multiplication
operator by a real-valued function:
$(\hat{q}\hat{f})(v)=\hat{q}(v)\hat{f}(v)$. Our main result is
\begin{theorem}
    Assume that $\hat{q}(v)=0$ except for a finite number of
    $v\in{}V(G)$. Then from the scattering amplitude $A(\lambda)$ for all
    $\lambda\in\sigma_{ess}(H)$, one can uniquely compute
    $(\hat{q}(v))_{v\in{}V(G)}$.
\end{theorem}

The hexagonal lattice can be viewed  as a discrete model of graphene, which is
a two dimensional, single-layered carbon sheet with honeycomb
structure. Therefore, discrete Schr\"{o}dinger operators on the hexagonal
lattice are regarded as discrete Hamiltonians on graphene. In physics,
graphene is one of the most interesting subjects due to the peculiar behavior
of electrons, and very actively studied, recently. Discrete Schr\"{o}dinger
operators, or tight binding models of Hamiltonians, are widely used to
investigate graphene (Neto, Guinea, Peres, Novoselov and Geim
\cite{RevModPhys.81.109}). Another approach is to study quantum graphs
(Kuchment and Post \cite{MR2336365}).

If we introduce a small parameter and take the limit to zero around the Dirac
points, it is well-known that the discrete Schr\"{o}dinger operator on the
hexagonal lattice tends to the Dirac operator on $\mathbb{R}^2$ (Semenoff
\cite{PhysRevLett.53.2449}, Gonz{\'a}lez, Guinea and Vozmediano
\cite{MR1239630}). Therefore, in a sense, we are considering a discrete
version of the inverse scattering problem for Dirac operators in two
dimensions. It is not known so much about the inverse scattering theory for
Dirac operators, relative to that for Schr\"{o}dinger operators. We refer to
Isozaki \cite{MR1449193}.

The rest of this paper is organized as follows: In Section 2, we review
spectral properties of the free Hamiltonian on the hexagonal lattice, and
define a conjugate operator to derive Mourre estimates. As a consequence, we
show the absolute continuity of the discrete Schr\"{o}dinger operator on the
hexagonal lattice. In Section 3, we construct a spectral representation, and
then obtain a representation of the S-matrix. Our main result, the
reconstruction procedure for the potential, is shown in Section 6, and the key
lemmas for it, analytic continuation and estimates of the resolvent, are
proved in Sections 4 and 5, respectively. Section 7 remarks that our
reconstruction procedure also works on the triangle lattice in almost the same
way as on the hexagonal lattice and the two-dimensional square lattice. Some
technical lemmas are proved in Appendices.


%
%
\section{Discrete Schr\"{o}dinger operators on the hexagonal lattice}

\subsection{Preliminaries}

As far as the discrete Laplacian is concerned, we only have to take care of
the adjacent relations between the vertices, which enables us to illustrate
the hexagonal lattice $G$ as Figure \ref{fig_hexagonal_lattice_1},
\begin{figure}[htbp]
    \begin{center}
        \includegraphics[height=4cm]{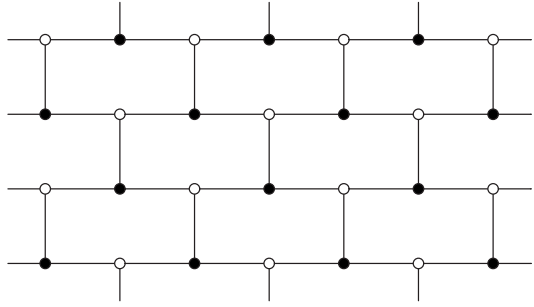}
    \end{center}
    \caption{The hexagonal lattice can be deformed as above, where the black
      and white dots are placed on $\mathbb{Z}^2$, and correspond to those in
      Figure \ref{fig_hexagonal_lattice_0},
      respectively. \label{fig_hexagonal_lattice_1}}
\end{figure}
where the vertices are placed on $\mathbb{Z}^2$. Therefore, we can regard that
\begin{equation*}
    V(G)=\mathbb{Z}^2,
\end{equation*}
\begin{align*}
    E(G)=&\{[m,n];m,n\in\mathbb{Z}^2,n_1=m_1+1,m_2=n_2\}\\
    &\cup\{[m,n];m,n\in\mathbb{Z}^2,m_1=n_1,m_2\in2\mathbb{Z}+1,n_2=m_2+1\}\\
    &\cup\{[m,n];m,n\in\mathbb{Z}^2,m_1=n_1,m_2\in2\mathbb{Z},n_2=m_2+1\},
\end{align*}
where $m=(m_1,m_2)$ and $n=(n_1,n_2)$, and the segment between $m$ and $n$ is
denoted by $[m,n]$. Also, noticing that the degree of the vertices on the
hexagonal lattice is always three, we can regard $l^2(G)$ as just the
$l^2$-space on $\mathbb{Z}^2$ equipped with the norm
$\Vert\hat{g}\Vert_{l^2(G)}^2=3\sum_{n\in\mathbb{Z}^2}\vert\hat{g}(n)\vert^2$.

We next introduce the hexagonal lattice structure on $\mathbb{Z}^2$, which is
different from the standard lattice one, in the following manner. We split
$\mathbb{Z}^2$ into two parts:
$\mathbb{Z}^2=\mathbb{Z}_e^2\cup\mathbb{Z}_o^2$, where
\begin{equation*}
    \mathbb{Z}_e^2=\{(n_1,n_2); n_1+n_2\in2\mathbb{Z}\},
\end{equation*}
\begin{equation*}
    \mathbb{Z}_o^2=\{(n_1,n_2); n_1+n_2-1\in2\mathbb{Z}\}.
\end{equation*}
In Figure \ref{fig_hexagonal_lattice_1}, $\mathbb{Z}_e^2$ and $\mathbb{Z}_o^2$
are represented as the black dots and the white ones, respectively, each of
which has a lattice structure with basis $\vec{e}_1=(1,1)$ and 
$\vec{e}_2=(-1,1)$. Therefore, there exists a canonical isomorphisms
\begin{equation*}
    \mathbb{Z}_e^2\ni(n_1,n_2)\longmapsto(m_1,m_2)\in\mathbb{Z}^2,
\end{equation*}
\begin{equation}
    \label{n_1+n_2:even}
    n_1=m_1-m_2,\ n_2=m_1+m_2,
\end{equation}
and
\begin{equation*}
    \mathbb{Z}_o^2\ni(n_1,n_2)\longmapsto(m_1,m_2)\in\mathbb{Z}^2,
\end{equation*}
\begin{equation}
    \label{n_1+n_2:odd}
    n_1=m_1-m_2,\ n_2=1+m_1+m_2.
\end{equation}
Noticing that $\mathbb{Z}_e^2$ and $\mathbb{Z}_o^2$ are the minimal periodic
lattice structures in the hexagonal lattice, we need to use a system of
difference operators to analyze the discrete Laplacian on the hexagonal
lattice.

\subsection{Discrete Laplacian}
There is a natural unitary mapping
${\cal{J}}:(l^2(\mathbb{Z}^2))^2\rightarrow{}l^2(G)$, more precisely,
\begin{equation*}
    {\cal{J}}:l^2(\mathbb{Z}^2)\oplus{}l^2(\mathbb{Z}^2)\ni\hat{f}=(\hat{f}_1,\hat{f}_2){}\longmapsto\hat{g}\in{}l^2(G),
\end{equation*}
\begin{equation*}
    \hat{g}(n_1,n_2)=({\cal{J}}\hat{f})(n_1,n_2)=
    \begin{cases}
        \displaystyle\frac{1}{\sqrt{3}}\hat{f}_1(m_1,m_2),&\text{if $(n_1,n_2)\in\mathbb{Z}_e^2$},\\
        \displaystyle\frac{1}{\sqrt{3}}\hat{f}_2(m_1,m_2),&\text{if $(n_1,n_2)\in\mathbb{Z}_o^2$},
    \end{cases}
\end{equation*}
where $(m_1,m_2)$ is defined by \eqref{n_1+n_2:even} and \eqref{n_1+n_2:odd},
respectively. Note that $\sqrt{3}$ comes from the degree of the vertices,
which contributes to the definition of the norm of $l^2(G)$.

The adjacent vertices of $(n_1,n_2)\in\mathbb{Z}_e^2$ are $(n_1\pm1,n_2)$,
$(n_1,n_2+1)$, and those of $\mathbb{Z}_o^2$ are $(n_1\pm1,n_2)$,
$(n_1,n_2-1)$. Therefore, $3(\Delta_d+1)$ on $G$ is written as
\begin{align*}
    &(3(\Delta_d+1)\hat{g})(n_1,n_2)\\
    =&
    \begin{cases}
        \hat{g}(n_1,n_2+1)+\hat{g}(n_1-1,n_2)+\hat{g}(n_1+1,n_2),&\text{if
          $(n_1,n_2)\in\mathbb{Z}_e^2$},\\
        \hat{g}(n_1,n_2-1)+\hat{g}(n_1-1,n_2)+\hat{g}(n_1+1,n_2),&\text{if
          $(n_1,n_2)\in\mathbb{Z}_o^2$},
    \end{cases}
\end{align*}
for $\hat{g}\in{}l^2(G)$. This implies that
\begin{equation*}
    ({\cal{J}}^*3(\Delta_d+1){\cal{J}}\hat{f})_1(m_1,m_2)=\hat{f}_2(m_1,m_2)+\hat{f}_2(m_1-1,m_2)+\hat{f}_2(m_1,m_2-1),
\end{equation*}
\begin{equation*}
    ({\cal{J}}^*3(\Delta_d+1){\cal{J}}\hat{f})_2(m_1,m_2)=\hat{f}_1(m_1,m_2)+\hat{f}_1(m_1+1,m_2)+\hat{f}_1(m_1,m_2+1),
\end{equation*}
for
$\hat{f}=(\hat{f}_1,\hat{f}_2)\in{}l^2(\mathbb{Z}^2)\oplus{}l^2(\mathbb{Z}^2)$.

Let us define
\begin{equation*}
    l^2(\mathbb{Z}^2;\mathbb{C}^2)=\{\hat{f}=(
    \begin{pmatrix}
        \hat{f}_1(m)\\
        \hat{f}_2(m)
    \end{pmatrix}
    )_{m\in\mathbb{Z}^2};\Vert\hat{f}\Vert_{l^2(\mathbb{Z}^2;\mathbb{C}^2)}^2=\sum_{m\in\mathbb{Z}^2}(\vert\hat{f}_1(m)\vert^2+\vert\hat{f}_2(m)\vert^2)<\infty\}.
\end{equation*}
There is a unitary mapping
${\cal{I}}:l^2(\mathbb{Z}^2;\mathbb{C}^2)\rightarrow(l^2(\mathbb{Z}^2))^2$,
under which we can naturally identify $l^2(\mathbb{Z}^2;\mathbb{C}^2)$ with
$(l^2(\mathbb{Z}^2))^2$, more precisely,
\begin{equation*}
    {\cal{I}}:l^2(\mathbb{Z}^2;\mathbb{C}^2)\ni(
    \begin{pmatrix}
        \hat{f}_1(m)\\
        \hat{f}_2(m)
    \end{pmatrix}
    )_{m\in\mathbb{Z}^2}\longmapsto((\hat{f}_1(m))_{m\in\mathbb{Z}^2},(\hat{f}_2(m))_{m\in\mathbb{Z}^2})\in{}l^2(\mathbb{Z}^2)\oplus{}l^2(\mathbb{Z}^2).
\end{equation*}
Let $\hat{H}_0={\cal{I}}^*{\cal{J}}^*3(\Delta_d+1){\cal{J}}{\cal{I}}$. Then we
can write 
\begin{equation*}
    (\hat{H}_0\hat{f})(m)=
    \begin{pmatrix}
        \hat{f_2}(m_1,m_2)+\hat{f_2}(m_1-1,m_2)+\hat{f_2}(m_1,m_2-1)\\
        \hat{f_1}(m_1,m_2)+\hat{f_1}(m_1+1,m_2)+\hat{f_1}(m_1,m_2+1)
    \end{pmatrix}
\end{equation*}
for
\begin{equation*}
    \hat{f}=(
    \begin{pmatrix}
        \hat{f}_1(m)\\
        \hat{f}_2(m)
    \end{pmatrix}
    )_{m\in\mathbb{Z}^2}\in{}l^2(\mathbb{Z}^2;\mathbb{C}^2).
\end{equation*}

We put $\mathbb{T}^2=\mathbb{R}^2/(2\pi\mathbb{Z})^2$ and
\begin{equation*}
    L^2(\mathbb{T}^2;\mathbb{C}^2)=\{f=(f_1(\xi),f_2(\xi));\Vert{f}\Vert_{L^2(\mathbb{T}^2;\mathbb{C}^2)}^2=\int_{\mathbb{T}^2}(\vert{f_1(\xi)}\vert^2+\vert{f_2(\xi)}\vert^2)d\xi<\infty\}.
\end{equation*}
We define a unitary operator
${\cal{F}}:l^2(\mathbb{Z}^2)\rightarrow{}L^2(\mathbb{T}^2)$ and its adjoint
${\cal{F}}^*$ by
\begin{equation*}
    ({\cal{F}}\hat{f})(\xi)=\frac{1}{2\pi}\sum_{n\in\mathbb{Z}^2}\hat{f}(n)e^{in\xi},
\end{equation*}
\begin{equation*}
    ({\cal{F}}^*f)(n)=\frac{1}{2\pi}\int_{\mathbb{T}^2}f(\xi)e^{-in\xi}d\xi,
\end{equation*}
and extend them naturally on $l^2(\mathbb{Z}^2;\mathbb{C}^2)$ and
$L^2(\mathbb{T}^2;\mathbb{C}^2)$. They are the usual Fourier series and
Fourier coefficients for each component. We put
\begin{equation*}
    H_0={\cal{F}}\hat{H}_0{\cal{F}}^*.
\end{equation*}
Then $H_0$ is a multiplication operator by a symmetric matrix on
$L^2(\mathbb{T}^2;\mathbb{C}^2)$:
\begin{equation*}
    (H_0f)(\xi)=H_0(\xi)f(\xi),
\end{equation*}
where
\begin{equation*}
    H_0(\xi)=
    \begin{pmatrix}
        0&\alpha(\xi)\\
        \overline{\alpha}(\xi)&0
    \end{pmatrix},
\end{equation*}
\begin{equation}
    \label{alpha}
    \alpha(\xi)=1+e^{i\xi_1}+e^{i\xi_2},\ \xi=(\xi_1,\xi_2)\in\mathbb{T}^2,
\end{equation}
\begin{equation}
    \label{overline_alpha}
    \overline{\alpha}(\xi)=\overline{\alpha(\xi)}.
\end{equation}

Next, let us define the unitary operator ${\cal{U}}$ as
\begin{equation*}
    ({\cal U}f)(\xi)=U(\xi)f(\xi)
\end{equation*}
on $L^2(\mathbb{T}^2;\mathbb{C}^2)$, where
\begin{equation*}
    U(\xi)=\frac{1}{\sqrt{2}}
    \begin{pmatrix}
        1&\alpha(\xi)/\vert\alpha(\xi)\vert\\
        1&-\alpha(\xi)/\vert\alpha(\xi)\vert
    \end{pmatrix}
\end{equation*}
is a unitary matrix for each $\xi\in\mathbb{T}^2$. Its adjoint ${\cal{U}}^*$
is a multiplication operator by the adjoint matrix $U^*(\xi)=U(\xi)^*$. We put
\begin{equation*}
    \tilde{H}_0={\cal{U}}H_0{\cal{U}}^*.
\end{equation*}
The conjugation by $U(\xi)$ diagonalizes $H_0(\xi)$, which means that
$\tilde{H}_0$ is a multiplication operator by a diagonal matrix on
$L^2(\mathbb{T}^2;\mathbb{C}^2)$:
\begin{equation*}
    (\tilde{H}_0f)(\xi)=\tilde{H}_0(\xi)f(\xi),
\end{equation*}
where
\begin{equation*}
    \tilde{H}_0(\xi)=U(\xi)H_0(\xi)U^*(\xi)=
    \begin{pmatrix}
        p(\xi)&0\\
        0&-p(\xi)
    \end{pmatrix},
\end{equation*}
\begin{equation}
    \label{p}
    p(\xi)=\vert\alpha(\xi)\vert=\sqrt{3+2\cos{\xi_1}+2\cos{\xi_2}+2\cos{(\xi_1-\xi_2)}}.
\end{equation}

One computes that
\begin{align}
    \label{p(xi)=0}
    \begin{split}
        &3+2\cos{\xi_1}+2\cos{\xi_2}+2\cos{(\xi_1-\xi_2)}\\
        =&4(\cos{\frac{\xi_1-\xi_2}{2}}+\frac{1}{2}\cos{\frac{\xi_1+\xi_2}{2}})^2-\cos^2{\frac{\xi_1+\xi_2}{2}}+1.
    \end{split}
\end{align}
We put
\begin{equation}
    \label{xi_0_1,xi_0_2}
    \xi_{0,1}=(\frac{2\pi}{3},-\frac{2\pi}{3}),\ \xi_{0,2}=(-\frac{2\pi}{3},\frac{2\pi}{3})\in\mathbb{T}^2.
\end{equation}
Then we have
\begin{prop}
    \label{properties_of_p(xi)}
    $p(\xi)$ is continuous and
    \begin{equation*}
        0\le{p(\xi)}\le{3}
    \end{equation*}
    on $\mathbb{T}^2$. Moreover,
    \begin{enumerate}
    \item $p(\xi)=0$ if and only if $\xi=\xi_{0,1}$ or
        $\xi_{0,2}$,\label{more_properties_of_p(xi)_1}
    \item $p(\xi)=3$ if and only if
        $\xi=(0,0)$,\label{more_properties_of_p(xi)_2}
    \item $p(\xi)$ is real-analytic on
        $\mathbb{T}^2\setminus\{\xi_{0,1},\xi_{0,2}\}$.\label{more_properties_of_p(xi)_3}
    \end{enumerate}
\end{prop}
\begin{proof}
    Clearly, $p(\xi)^2$ is real-analytic and
    \begin{equation*}
        0\le{}p(\xi)^2=\vert\alpha(\xi)\vert^2=3+2\cos{\xi_1}+2\cos{\xi_2}+2\cos{(\xi_1-\xi_2)}\le9
    \end{equation*}
    on $\mathbb{T}^2$, which proves the first statement.

    By using \eqref{p(xi)=0}, we have $p(\xi)=0$ if and only if
    \begin{equation*}
        \cos{\frac{\xi_1-\xi_2}{2}}+\frac{1}{2}\cos{\frac{\xi_1+\xi_2}{2}}=0,
    \end{equation*}
    \begin{equation*}
        \cos{\frac{\xi_1+\xi_2}{2}}=\pm1,
    \end{equation*}
    which means that the zeros of $p(\xi)$ are exactly $\xi_{0,j}$,
    $j\in\{1,2\}$.

    The next statement is obvious.

    Noticing that $\xi_{0,1}$ and $\xi_{0,2}$ are all the zeros of $p(\xi)^2$,
    we have the last one.
\end{proof}

We denote by $\sigma(T)$ the spectrum of a self-adjoint operator $T$; by
$\sigma_{ac}(T)$, $\sigma_{sc}(T)$, and $\sigma_{pp}(T)$ its absolutely
continuous one, its singularly continuous one, and its pure point one,
respectively. By using the above proposition, we have
\begin{prop}
    $\sigma(\tilde{H}_0)=\sigma_{ac}(\tilde{H}_0)=[-3,3]$ and
    $\sigma_{pp}(\tilde{H}_0)=\sigma_{sc}(\tilde{H}_0)=\emptyset$.
\end{prop}

\begin{rem}
    In the graph theory, the hexagonal lattice is, in general, called an
    abelian covering graph of a finite graph. For more general treatments for
    the spectrum of the discrete Laplacian on such a graph, we refer to
    Kotani, Shirai and Sunada \cite{MR1658100}.
\end{rem}

\subsection{Discrete Schr\"{o}dinger operators}
The potential on $l^2(\mathbb{Z}^2;\mathbb{C}^2)$ is denoted by $\hat{q}$,
which is a multiplication operator by real-valued, diagonal, $2\times2$
matrices:
\begin{equation*}
    \hat{q}=\sum_{n\in\mathbb{Z}^2}\hat{q}(n)\hat{P}(n),
\end{equation*}
where
\begin{equation}
    \label{q_n}
    \hat{q}(n)=
    \begin{pmatrix}
        \hat{q}_1(n)&0\\
        0&\hat{q}_2(n)
    \end{pmatrix},
\end{equation}
\begin{equation*}
    \hat{q}_j(n)\in\mathbb{R},\ j\in\{1,2\}.
\end{equation*}
Here we denote by $\hat{P}(n)$ the projection onto the site $n\in\mathbb{Z}^2$
on $l^2(\mathbb{Z}^2;\mathbb{C}^2)$:
\begin{equation*}
    (\hat{P}(n)\hat{f})(m)=
    \begin{pmatrix}
        \delta_{nm}\hat{f}_1(m)\\
        \delta_{nm}\hat{f}_2(m)
    \end{pmatrix},
\end{equation*}
where $\delta_{nm}$ is Kronecker's delta. Throughout the paper, we shall
assume that
\begin{assumption}
    \label{assumption_potential}
    $\hat{q}$ is compactly supported, that is, $\hat{q}(n)=0$ except for a
    finite number of $n\in\mathbb{Z}^2$.
\end{assumption}
Note that some parts of our arguments are extended to more general decaying
potentials.

We put
\begin{equation*}
    q={\cal{F}}\hat{q}{\cal{F}}^*,\ \tilde{q}={\cal{U}}q{\cal{U}}^*.
\end{equation*}
Then $q$ and $\tilde{q}$ are written as
\begin{equation*}
    (qf)(\xi)=\frac{1}{2\pi}
    \begin{pmatrix}
        \displaystyle\int_{\mathbb{T}^2}q_1(\xi-\zeta)f_1(\zeta)d\zeta\\
        \displaystyle\int_{\mathbb{T}^2}q_2(\xi-\zeta)f_2(\zeta)d\zeta
    \end{pmatrix},
\end{equation*}
\begin{equation*}
    (\tilde{q}f)(\xi)=
    \left(
        \begin{array}{l}
            (\tilde{q}f)_1(\xi)\\
            (\tilde{q}f)_2(\xi)
        \end{array}
    \right),
\end{equation*}
for $f=(f_1,f_2)\in{}L^2(\mathbb{T}^2;\mathbb{C}^2)$, where
\begin{equation*}
    q_j(\xi)=\frac{1}{2\pi}\sum_{n\in\mathbb{Z}^2}\hat{q}_j(n)e^{in\xi},\ j\in\{1,2\},
\end{equation*}
and
\begin{align*}
    (\tilde{q}f)_1(\xi)=&\displaystyle\frac{1}{2(2\pi)}\int_{\mathbb{T}^2}(q_1(\xi-\zeta)+\frac{\alpha(\xi)}{p(\xi)}q_2(\xi-\zeta)\frac{\overline{\alpha}(\zeta)}{p(\zeta)})f_1(\zeta)d\zeta\\
    &\displaystyle+\frac{1}{2(2\pi)}\int_{\mathbb{T}^2}(q_1(\xi-\zeta)-\frac{\alpha(\xi)}{p(\xi)}q_2(\xi-\zeta)\frac{\overline{\alpha}(\zeta)}{p(\zeta)})f_2(\zeta)d\zeta,
\end{align*}
\begin{align*}
    (\tilde{q}f)_2(\xi)=&\displaystyle\frac{1}{2(2\pi)}\int_{\mathbb{T}^2}(q_1(\xi-\zeta)-\frac{\alpha(\xi)}{p(\xi)}q_2(\xi-\zeta)\frac{\overline{\alpha}(\zeta)}{p(\zeta)})f_1(\zeta)d\zeta\\
    &\displaystyle+\frac{1}{2(2\pi)}\int_{\mathbb{T}^2}(q_1(\xi-\zeta)+\frac{\alpha(\xi)}{p(\xi)}q_2(\xi-\zeta)\frac{\overline{\alpha}(\zeta)}{p(\zeta)})f_2(\zeta)d\zeta.
\end{align*}

The discrete Schr\"{o}dinger operator is denoted by
\begin{equation*}
    \hat{H}=\hat{H}_0+\hat{q}.
\end{equation*}
We also put
\begin{equation*}
    H={\cal{F}}\hat{H}{\cal{F}}^*=H_0+q,
\end{equation*}
\begin{equation*}
    \tilde{H}={\cal{U}}H{\cal{U}}^*=\tilde{H}_0+\tilde{q}.
\end{equation*}
We denote by $\sigma_{ess}(T)$ the essential spectrum of a self-adjoint
operator $T$. Then, by noticing that the potential is a compact perturbation,
we have
\begin{prop}
    $\sigma_{ess}(\tilde{H})=[-3, 3]$.
\end{prop}

\subsection{Sobolev spaces and Mourre estimates}

Let $L_0$ be the self-adjoint extension of the operator
\begin{equation*}
    \begin{pmatrix}
        -\Delta&0\\
        0&-\Delta
    \end{pmatrix}
\end{equation*}
on $C^{\infty}(\mathbb{T}^2;\mathbb{C}^2)=\{\mathbb{C}^2\text{-valued smooth
function on }\mathbb{T}^2\}$, where
$\Delta=(\frac{\partial^2}{{\partial\xi_1}^2}+\frac{\partial^2}{{\partial\xi_2}^2})$. We
denote by ${\cal{H}}^s$, $s\in\mathbb{R}$, the domain of the operator
$L_0^{s/2}$:
\begin{equation*}
    {\cal{H}}^s=\{f\in{\cal{D}}^{'}(\mathbb{T}^2;\mathbb{C}^2);\Vert{f}\Vert_s=\Vert{(1+L_0)^{s/2}f}\Vert<\infty\},
\end{equation*}
where ${\cal{D}}^{'}(\mathbb{T}^2;\mathbb{C}^2)=\{\mathbb{C}^2\text{-valued
  distribution on }\mathbb{T}^2\}$. Also, we denote
\begin{equation*}
    \hat{{\cal{H}}}^s={\cal{F}}^*{\cal{H}}^s.
\end{equation*}
We put ${\cal{H}}={\cal{H}}^0=L^2(\mathbb{T}^2;\mathbb{C}^2)$ and
$\hat{\cal{H}}=\hat{\cal{H}}^0=l^2(\mathbb{Z}^2;\mathbb{C}^2)$. The next
proposition is obvious by Parseval's equation.
\begin{prop}
    $f\in{\cal{H}}^s\Longleftrightarrow\sum_{n\in\mathbb{Z}^2}(1+\vert{n}\vert^2)^s\Vert{\hat{f}(n)}\Vert_{\mathbb{C}^2}^2=\sum_{n\in\mathbb{Z}^2}(1+\vert{n}\vert^2)^s(\vert{\hat{f}_1(n)}\vert^2+\vert{\hat{f}_2(n)}\vert^2)<\infty$,
    where $f={\cal{F}}\hat{f}$.
\end{prop}

We derive Mourre estimates. At first, we perform formal calculations, and then
modify the conjugate operator by introducing cut-off functions to prove its
self-adjointness. Let $A$ be a first-order differential operator defined as
\begin{align*}
    A&=i[\tilde{H}_0,L_0]\\
    &=i
    \begin{pmatrix}
        \nabla{p}\cdot\nabla+\nabla\cdot\nabla{p}&0\\
        0&-(\nabla{p}\cdot\nabla+\nabla\cdot\nabla{p})
    \end{pmatrix}.
\end{align*}
By a straightforward calculation, the commutator of $\tilde{H}_0$ and $A$ is
\begin{equation*}
    i[\tilde{H}_0,A]=2\vert\nabla{p}(\xi)\vert^2I_2,
\end{equation*}
where $I_2$ is the $2\times2$ identity matrix.

A simple calculation shows that
\begin{equation}
    \label{nabla_p}
    \nabla{p}(\xi)=(\frac{-\sin{\xi_1}-\sin{(\xi_1-\xi_2)}}{p(\xi)},\frac{-\sin{\xi_2}+\sin{(\xi_1-\xi_2)}}{p(\xi)}).
\end{equation}
If $\sin{\xi_1}+\sin{(\xi_1-\xi_2)}=0$, then
\begin{equation*}
    \xi_1-\xi_2=\xi_1+\pi\text{ or }-\xi_1\pmod{2\pi}.
\end{equation*}
Also, if $\sin{\xi_2}-\sin{(\xi_1-\xi_2)}=0$, then
\begin{equation*}
    \xi_1-\xi_2=\xi_2\text{ or }-\xi_2+\pi\pmod{2\pi}.
\end{equation*}
Note that we can not define $\nabla{p}(\xi)$ at $\xi_{0,1}$ and
$\xi_{0,2}$. Therefore, we have
\begin{equation*}
    \{\xi\in\mathbb{T}^2;\nabla{p}(\xi)=0\}=\{(0,0),(0,-\pi),(-\pi,0),(-\pi,-\pi)\}.
\end{equation*}

From \eqref{p(xi)=0}, $p(\xi)=1$ if and only if $\cos{(\xi_1-\xi_2)/2}=0$ or
$\cos{(\xi_1-\xi_2)/2}=-\cos{(\xi_1+\xi_2)/2}$. If $\cos{(\xi_1-\xi_2)/2}=0$,
then
\begin{equation*}
    \frac{\xi_1-\xi_2}{2}=\frac{\pi}{2}\pmod{\pi}.
\end{equation*}
Also, if $\cos{(\xi_1-\xi_2)/2}=-\cos{(\xi_1+\xi_2)/2}$, then
\begin{equation*}
    \frac{\xi_1-\xi_2}{2}=\pm\frac{\xi_1+\xi_2}{2}+\pi\pmod{2\pi}.
\end{equation*}
Therefore, we have
\begin{equation*}
    \{\xi\in\mathbb{T}^2;p(\xi)=1\}=\{(\xi_1,\xi_2)\in[-\pi,\pi)^2;\text{$\xi_1=-\pi$
      or $\xi_2=-\pi$ or $\xi_2=\xi_1\pm\pi$}\},
\end{equation*}
which includes the set $\{\xi\in\mathbb{T}^2;\nabla{p}(\xi)=0\}$ except the
origin.

Let us define
\begin{equation}
    \label{M_E}
    {\cal{M}}_E=\{\xi\in\mathbb{T}^2;p(\xi)=E\}.
\end{equation}
Then we have
\begin{prop}
    \label{prop_analytic_mfd}
    For $E\in(0,3)\setminus\{1\}$, ${\cal{M}}_E$ is a real-analytic compact
    manifold.
\end{prop}
\begin{proof}
    We have $\nabla{p}(\xi)\neq{0}$ on ${\cal{M}}_E$ for
    $E\in(0,3)\setminus\{1\}$. Noticing Proposition \ref{properties_of_p(xi)},
    we have the claim.
\end{proof}

Put $C_0(E)=2\inf_{\xi\in{\cal{M}}_E}\vert{\nabla{p}(\xi)}\vert^2$. Note that
$C_0(E)>0$ for $E\in(0,3)\setminus\{1\}$. Then, for small $\varepsilon>0$,
there is $\delta>0$ such that
\begin{equation*}
    \vert{\nabla{p}(\xi)}\vert^2\ge{}C_0(E)-\varepsilon>0
\end{equation*}
on $p^{-1}([E-\delta,E+\delta])$, which shall imply that
\begin{equation*}
    i[\tilde{H}_0,A]\ge{C_0(E)-\varepsilon}.
\end{equation*}

We have already shown that $p(\xi_{0,j})=0$, $j\in\{1,2\}$, which implies that
$\nabla{p}(\xi)$ has singularities at $\xi_{0,1}$ and $\xi_{0,2}$. To avoid
them, we introduce a smooth cut-off function $\chi(\xi)$: for small $\mu>0$,
\begin{equation}
    \label{cut-off_function}
    \chi(\xi)=\chi_{\mu}(\xi)=
    \begin{cases}
        0,&\text{if $\xi\in{}B_{\mu}(\xi_{0,1})\cup{}B_{\mu}(\xi_{0,2})$,}\\
        1,&\text{if $\xi\notin{}B_{2\mu}(\xi_{0,1})\cup{}B_{2\mu}(\xi_{0,2})$,}
    \end{cases}
\end{equation}
where $B_{\mu}(\xi_0)=\{\xi\in\mathbb{T}^2;\vert{\xi-\xi_0}\vert<\mu\}$. Let
us define
\begin{equation*}
    A_{\chi}=\chi{A}\chi.
\end{equation*}
Then, by Nelson's commutator theorem (Theorem X.37, Reed-Simon
\cite{MR529429}, where we take $N=L_0+1$), $A_{\chi}$ is essentially
self-adjoint on $C^{\infty}(\mathbb{T}^2;\mathbb{C}^2)$.

Let us choose sufficiently small $\mu>0$ depending on $E$ and $\delta>0$. Then
we have the Mourre estimate for $\tilde{H}_0$:
\begin{equation}
    \label{Mourre_estimate_for_H_0}
    f(\tilde{H}_0)i[\tilde{H}_0,A_{\chi}]f(\tilde{H}_0)\ge{(C_0(E)-\varepsilon)f(\tilde{H}_0)^2}
\end{equation}
for any real-valued $f\in{}C_0^{\infty}((E-\delta,E+\delta))$. Figure
\ref{fig_surface_and_contour_of_p(xi)}
\begin{figure}[htbp]    
    \begin{minipage}{0.5\hsize}
        \begin{center}
            \includegraphics[height=6cm]{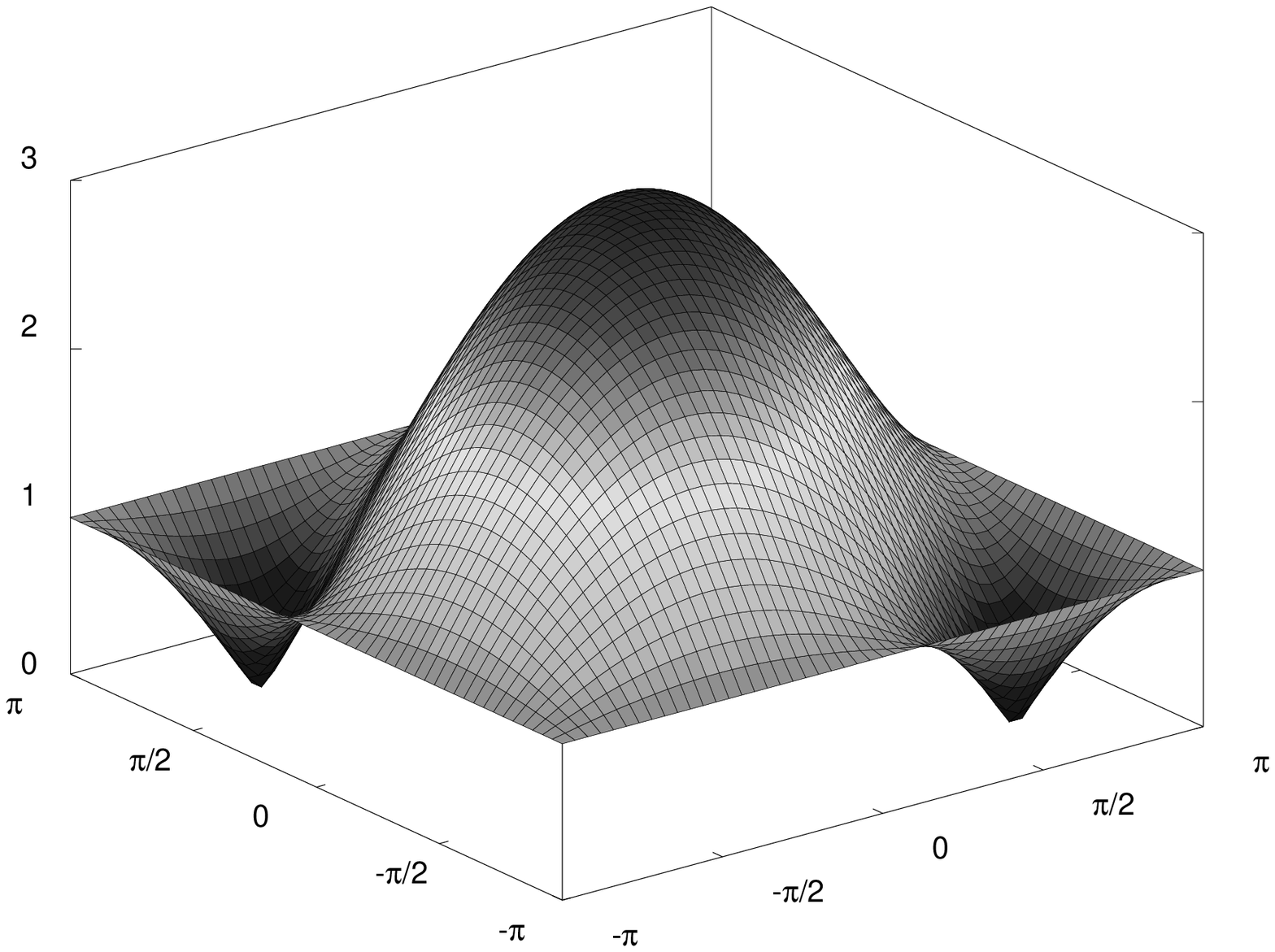}
        \end{center}
    \end{minipage}
    \quad\quad
    \begin{minipage}{0.5\hsize}
        \begin{center}
            \includegraphics[height=4.5cm]{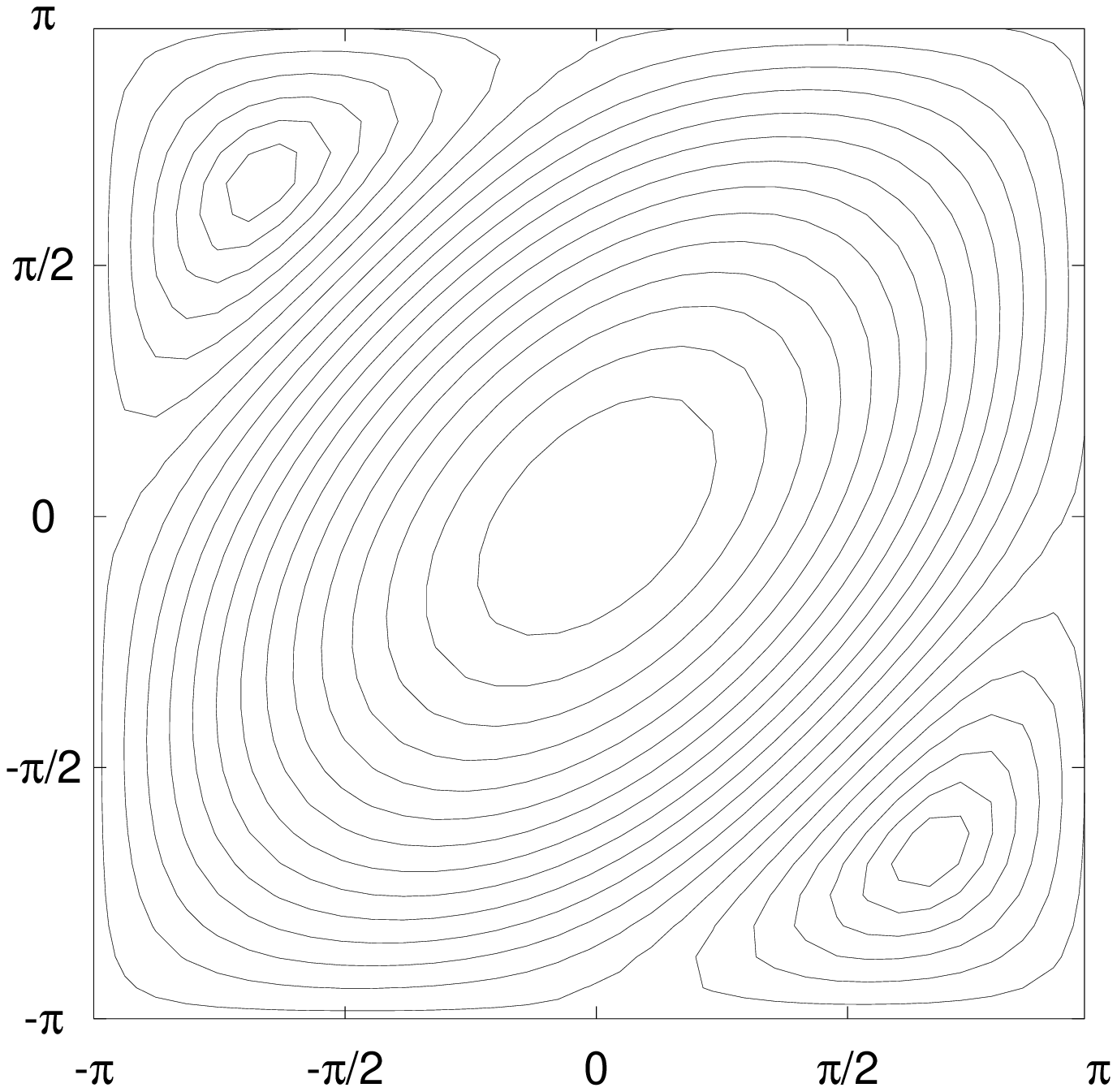}
        \end{center}
    \end{minipage}
    \caption{the surface graph and the contour lines of $p(\xi)$ on
      $[-\pi,\pi)^2$ \label{fig_surface_and_contour_of_p(xi)}}
\end{figure}
helps us to understand what we have done so far in this section.

We can write $(i[\tilde{q},A_{\chi}]f)(\xi)$ as a sum of the following terms:
for some $a(\xi)$, $b(\xi)$, $c(\xi)$, $d(\xi)\in{}L^{\infty}(\mathbb{T}^2)$,
and $j$, $k$, $l\in\{1,2\}$,
\begin{equation*}
    a(\xi)\int_{\mathbb{T}^2}q_j(\xi-\zeta)b(\zeta)f_k(\zeta)d\zeta,
\end{equation*}
\begin{equation*}
    c(\xi)\int_{\mathbb{T}^2}\frac{\partial q_j}{\partial \xi_l}(\xi-\zeta)d(\zeta)f_k(\zeta)d\zeta.
\end{equation*}
Our assumption \ref{assumption_potential} makes those operators compact. As a
result, we have the Mourre estimate for $\tilde{H}$, that is, there is a
compact operator $K=K_I$ depending on the interval $I=(E-\delta/2,E+\delta/2)$
such that
\begin{equation}
    \label{Mourre_estimate_for_H}
    E_{\tilde{H}}(I)[\tilde{H},iA_{\chi}]E_{\tilde{H}}(I)\ge{}(C_0(E)-\varepsilon)E_{\tilde{H}}(I)+K,
\end{equation}
where $E_{\tilde{H}}(\cdot)$ is the spectral projection for $\tilde{H}$.

Let $\tilde{R}(z)=(\tilde{H}-z)^{-1}$. We denote by ${\bold{B}}(X,Y)$, or
simply ${\bold{B}}(X)$ when $X=Y$, the set of all bounded operators from $X$
to $Y$, where $X$ and $Y$ are Banach spaces. We also denote
$D((1+\vert{A_{\chi}}\vert)^s)$ equipped with graph norm by
${\cal{H}}_{A_{\chi}}$. Then, with the aid of \eqref{Mourre_estimate_for_H_0}
and \eqref{Mourre_estimate_for_H}, by the well-known Mourre theory (Mourre
\cite{MR603501}), we have the following theorem.
\begin{theorem}
    \label{mourre_estimate}
    Let $I_0=(-3,3)\setminus\{\pm{1},0\}$. Then
    \begin{enumerate}
    \item The eigenvalues of $\tilde{H}$ are of finite multiplicities with
        possible accumulation points $0$, $\pm1$, $\pm3$.
    \item There is no singularly continuous spectrum:
        \begin{equation*}
            {\cal{H}}={\cal{H}}_{pp}(\tilde{H})\oplus{\cal{H}}_{ac}(\tilde{H}).
        \end{equation*}
    \item Let $s>1/2$ and $\lambda\in{}I_0$. Then there is a norm limit
        $\tilde{R}(\lambda\pm{}i0):=\lim_{\varepsilon\searrow0}\tilde{R}(\lambda\pm{i\varepsilon})$
        in ${\bold{B}}({\cal{H}}_{A_{\chi}}^s,{\cal{H}}_{A_{\chi}}^{-s})$, and 
        $I_0\ni\lambda\rightarrow\tilde{R}(\lambda\pm{}i0)\in{\bold{B}}({\cal{H}}_{A_{\chi}}^s,{\cal{H}}_{A_{\chi}}^{-s})$
        is norm continuous. Furthermore, we have
        \begin{equation*}
            \sup_{\lambda\in{}J}\Vert\tilde{R}(\lambda\pm{}i0)\Vert_{{\bold{B}}({\cal{H}}_{A_{\chi}}^s,{\cal{H}}_{A_{\chi}}^{-s})}<\infty
        \end{equation*}
        for any compact interval $J$ in $I_0\setminus\sigma_{pp}(\tilde{H})$.
    \end{enumerate}
\end{theorem}
\begin{cor}
    \label{cor_limitng_absorption_principle}
    If $s>1/2$,
    \begin{equation}
        \sup_{\lambda\in{}J}\Vert\tilde{R}(\lambda\pm{}i0)\Vert_{{\bold{B}}({\cal{H}}^s,{\cal{H}}^{-s})}<\infty.
    \end{equation}
\end{cor}
\begin{proof}[Proof of the Corollary]
    We only have to note that the inclusions
    ${\cal{H}}^s\subset{\cal{H}}_{A_{\chi}}^s$ and
    $({\cal{H}}_{A_{\chi}}^s)^*\simeq{\cal{H}}_{A_{\chi}}^{-s}\subset{\cal{H}}^{-s}\simeq({\cal{H}}^s)^*$
    are continuous for $s>0$.
\end{proof}

Conjugating by ${\cal{U}}$, we have the same statements for $H$; moreover,
by ${\cal{F}}$, we also have those for $\hat{H}$, where we use
$\hat{{\cal{H}}}^s$ instead of ${\cal{H}}^s$.


%
%
\section{Eigenoperators and scattering matrix}

\subsection{Trace operators}

Let $\tilde{R}_0(z)=(\tilde{H}_0-z)^{-1}$. Then we have
\begin{equation*}
    (\tilde{R}_0(z)f,g)_{L^2(\mathbb{Z}^2;\mathbb{C}^2)}=\int_{\mathbb{T}^2}\frac{f_1(\xi)\overline{g_1(\xi)}}{p(\xi)-z}d\xi+\int_{\mathbb{T}^2}\frac{f_2(\xi)\overline{g_2(\xi)}}{-p(\xi)-z}d\xi
\end{equation*}
for $f=(f_1,f_2),g=(g_1,g_2)\in{}C^{\infty}(\mathbb{T}^2;\mathbb{C}^2)$.

It has been already shown that ${\cal{M}}_{\lambda}$ is a real-analytic
manifold for $\lambda\in(0,3)\setminus\{1\}$. Thus we can introduce local
coordinates $\omega$ on ${\cal{M}}_{\lambda}$, which induce the measure
\begin{equation}
    \label{torus_measure}
    d\xi_1d\xi_2=d{\cal M}_{\lambda}(\omega)d\lambda=J(\lambda,\omega)d\omega{d}\lambda,
\end{equation}
where $J(\lambda,\omega)$ is real-analytic with respect to $\lambda$ and
$\omega$.

We note that the set $\{\xi\in\mathbb{T}^2;p(\xi)\in\{0,1,3\}\}$, which
includes all the extreme points and the critical points of $p(\xi)$, has null
Lebesgue measure. It enables us to write
\begin{align*}
    (\tilde{R}_0(z)f,g)_{L^2(\mathbb{Z}^2;\mathbb{C}^2)}&=\int_0^3\frac{h_1(\rho)}{\rho-z}d\rho+\int_{-3}^0\frac{h_2(-\rho)}{\rho-z}d\rho\\
    &=\int_{-3}^3\frac{h(\rho)}{\rho-z}d\rho,
\end{align*}
where
\begin{equation*}
    h(\rho)=
    \begin{cases}
        h_2(-\rho),&\text{if $\rho\in(-3,0)\setminus\{-1\}$},\\
        h_1(\rho),&\text{if $\rho\in(0,3)\setminus\{1\}$},
    \end{cases}
\end{equation*}
\begin{equation*}
    h_j(\rho)=\int_{{\cal{M}}_{\vert\rho\vert}}f_j(\xi(\rho,\omega))\overline{g_j(\xi(\rho,\omega))}d{\cal{M}}_{\vert\rho\vert}(\omega),\ j\in\{1,2\}.
\end{equation*}
Then we have
\begin{equation*}
    (\tilde{R}_0(\lambda\pm{}i0)f,g)_{L^2(\mathbb{T}^2;\mathbb{Z}^2)}=\pm{i}\pi{h}(\lambda)+\text{p.v.}\int_{\lambda-\delta}^{\lambda-\delta}\frac{h(\rho)-h(\lambda)}{\rho-\lambda}d\rho+\int_{\vert\lambda-\rho\vert>\delta}\frac{h(\rho)}{\rho-\lambda}d\rho
\end{equation*}
for $\lambda\in(-3,3)\setminus\{\pm{1},0\}$, which leads us to
\begin{equation}
    \label{stone_trace}
    \frac{1}{2\pi{}i}((\tilde{R}_0(\lambda+i0)-\tilde{R}_0(\lambda-i0))f,g)_{L^2(\mathbb{T}^2;\mathbb{C}^2)}=h(\lambda).
\end{equation}

Let us define the trace operator $\tilde{\cal{F}}_0(\lambda)$ as
\begin{equation*}
    (\tilde{{\cal{F}}_0}(\lambda)f)(\omega)=
    \begin{cases}
        \begin{pmatrix}
            0\\
            f_2(\xi(-\lambda,\omega))
        \end{pmatrix},&\text{if $-3<\lambda<0$,}\\
        \begin{pmatrix}
            f_1(\xi(\lambda,\omega))\\
            0
        \end{pmatrix},&\text{if $0<\lambda<3$,}
    \end{cases}
\end{equation*}
for $f{\in}C^{\infty}(\mathbb{T}^2;\mathbb{C}^2)$. From \eqref{stone_trace},
we have
\begin{lemma}
    \label{R_0(lambda+i0)-R_0(lambda-i0)}
    For $\lambda\in(-3,3)\setminus\{\pm{1},0\}$,
    \begin{equation*}
        \frac{1}{2\pi{i}}((\tilde{R}_0(\lambda+i0)-\tilde{R}_0(\lambda-i0))f,g)_{L^2(\mathbb{T}^2;\mathbb{C}^2)}=(\tilde{{\cal{F}}}_0(\lambda)f,\tilde{{\cal{F}}}_0(\lambda)g)_{L^2({\cal{M}}_{\vert\lambda\vert};\mathbb{C}^2)}.
    \end{equation*}
\end{lemma}

By Corollary \ref{cor_limitng_absorption_principle} and this lemma, we have
\begin{cor}
    For $s>1/2$,
    $\tilde{{\cal{F}}}_0(\lambda)\in{\bold{B}}({\cal{H}}^s,L^2({\cal{M}}_{\vert\lambda\vert};\mathbb{C}^2))$.
\end{cor}

Moreover, by Stone's formula, we have
\begin{lemma}
    \label{Stone's_formula}
    For $f$, $g\in{\cal{H}}^s$, $s>1/2$,
    \begin{equation*}
        (f,g)_{L^2(\mathbb{T}^2;\mathbb{C}^2)}=\int_{-3}^{3}(\tilde{{\cal{F}}_0}(\lambda)f,\tilde{{\cal{F}}}_0(\lambda)g)_{L^2({\cal{M}}_{\vert\lambda\vert};\mathbb{C}^2)}d\lambda.
    \end{equation*}
\end{lemma}

By the definition, we also have
\begin{equation}
    \label{eigen_op_1}
    \tilde{{\cal{F}}}_0(\lambda)(\tilde{H}_0-\lambda)=0.
\end{equation}

We pass from $\tilde{H}_0$ to $H_0$. For $s\in\mathbb{R}$ and small
$\varepsilon>0$, let ${\cal{H}}_{\varepsilon}^s$ be the completion of
$C_{\varepsilon}^{\infty}(\mathbb{T}^2;\mathbb{C}^2)$ in ${\cal{H}}^s$, where
\begin{equation*}
    C_{\varepsilon}^{\infty}(\mathbb{T}^2;\mathbb{C}^2)=\{f{\in}C^{\infty}(\mathbb{T}^2;\mathbb{C}^2);\text{supp}{f}\subset\mathbb{T}^2\setminus(B_{\varepsilon}(\xi_{0,1})\cup{}B_{\varepsilon}(\xi_{0,2}))\}.
\end{equation*}
Note that
$\tilde{R}_0(z)C_{\varepsilon}^\infty(\mathbb{T}^2;\mathbb{C}^2)\subset{}C_{\varepsilon}^{\infty}(\mathbb{T}^2;\mathbb{C}^2)$. Since
${\cal{H}}_{\varepsilon}^s$ avoids the singular points of
$\alpha(\xi)/\vert\alpha(\xi)\vert$, we have
\begin{lemma}
    \label{epsilon_homeo}
    The operator ${\cal{U}}$ is defined as a homomorphism on
    ${\cal{H}}_{\varepsilon}^{s}$.
\end{lemma}
For sufficiently small $\varepsilon>0$, for example
$0<\varepsilon<\vert\lambda\vert/2$, we have
$\tilde{R}_0(\lambda\pm{i0})\in{\bold{B}}(L^2(B_{\varepsilon}(\xi_{0,1})\cup{}B_{\varepsilon}(\xi_{0,2});\mathbb{C}^2))$;
moreover, $\tilde{R}_0(\lambda+i0)f=\tilde{R}_0(\lambda-i0)f$ for
$f\in{}L^2(B_{\varepsilon}(\xi_{0,1})\cup{}B_{\varepsilon}(\xi_{0,2});\mathbb{C}^2)$. Therefore,
we can extend $\tilde{{\cal{F}}}_0(\lambda)$ uniquely as a null operator on
$L^2(B_{\varepsilon}(\xi_{0,1})\cup{}B_{\varepsilon}(\xi_{0,2});\mathbb{C}^2)$.

Let ${\cal{U}}^{\dagger}_{\lambda}$ be a multiplication operator on
$L^2({\cal{M}}_{\vert\lambda\vert};\mathbb{C}^2)$ by the unitary
matrix
\begin{equation*}
    U^{\dagger}(\lambda,\omega)=\frac{1}{\sqrt{2}}
    \begin{pmatrix}
        1&1\\
        \vert\lambda\vert^{-1}\overline{\alpha}(\xi(\vert\lambda\vert,\omega))&-\vert\lambda\vert^{-1}\overline{\alpha}(\xi(\vert\lambda\vert,\omega))
    \end{pmatrix},
\end{equation*}
which means that ${\cal{U}}^{\dagger}_{\lambda}$ is a unitary operator. Then,
from Lemma \ref{epsilon_homeo} and its subsequent descriptions, we can define
the operator
\begin{equation*}
    {\cal{F}}_0(\lambda)={\cal{U}}^{\dagger}_{\lambda}\tilde{\cal{F}}_0(\lambda){\cal{U}},
\end{equation*}
which is written as
\begin{equation*}
    ({\cal{F}}_0(\lambda)f)_1(\omega)=\frac{1}{2}f_1(\xi(\vert\lambda\vert,\omega))+\frac{1}{2}\lambda^{-1}\alpha(\xi(\vert\lambda\vert,\omega))f_2(\xi(\vert\lambda\vert,\omega)),
\end{equation*}
\begin{equation*}
    ({\cal{F}}_0(\lambda)f)_2(\omega)=\frac{1}{2}\lambda^{-1}\overline{\alpha}(\xi(\vert\lambda\vert,\omega))f_1(\xi(\vert\lambda\vert,\omega))+\frac{1}{2}f_2(\xi(\vert\lambda\vert,\omega)),
\end{equation*}
for $f=(f_1,f_2)\in{\cal{H}}^s$.
\begin{lemma}
    \label{F_0_continuity}
    For $s>1/2$,
    ${\cal{F}}_0(\lambda)\in{\bold{B}}({\cal{H}}^s,L^2({\cal{M}}_{\vert\lambda\vert};\mathbb{C}^2))$.
\end{lemma}
\begin{proof}
    Let $\chi=\chi_{\varepsilon}\in{}C^{\infty}(\mathbb{T}^2)$ be defined as
    \eqref{cut-off_function}. Then $1-\chi$ as a multiplication operator is
    continuous from ${\cal{H}}^s$ to
    $L^2(B_{\varepsilon}(\xi_{0,1})\cup{}B_{\varepsilon}(\xi_{0,2});\mathbb{C}^2)$;
    so is $\chi$ from ${\cal{H}}^s$ to ${\cal{H}}_{\varepsilon/4}^s$, which
    proves the lemma.
\end{proof}

Let $R_0(z)=(H_0-z)^{-1}$. Passing through ${\cal{U}}^{\dagger}_{\lambda}$ and
${\cal{U}}$ in Lemmas \ref{R_0(lambda+i0)-R_0(lambda-i0)} and
\ref{Stone's_formula}, we have the following two lemmas.
\begin{lemma}
    For $f$, $g\in{\cal{H}}^s$, $s>1/2$,
    \begin{equation*}
        \frac{1}{2\pi{i}}((R_0(\lambda+i0)-R_0(\lambda-i0))f,g)_{L^2(\mathbb{T}^2;\mathbb{C}^2)}=({\cal{F}}_0(\lambda)f,{\cal{F}}_0(\lambda)g)_{L^2({\cal{M}}_{\vert\lambda\vert};\mathbb{C}^2)},
    \end{equation*}
\end{lemma}
\begin{lemma}
    \label{F_0_unitary}
    For $f$, $g\in{\cal{H}}^s$, $s>1/2$,
    \begin{equation*}
        (f,g)_{L^2(\mathbb{T}^2;\mathbb{C}^2)}=\int_{-3}^{3}({\cal{F}}_0(\lambda)f,{\cal{F}}_0(\lambda)g)_{L^2({\cal{M}}_{\vert\lambda\vert};\mathbb{C}^2)}d\lambda.
    \end{equation*}
\end{lemma}
From \eqref{eigen_op_1}, we also have
\begin{equation}
    \label{eigen_op_2}
    {\cal{F}}_0(\lambda)(H_0-\lambda)=0.
\end{equation}

For $\lambda\in(-3,3)\setminus(\{\pm{1},0\}\cup\sigma_{pp}(H))$, let us define
the operator
\begin{equation*}
    {\cal{F}}^{(\pm)}(\lambda)={\cal{F}}_0(\lambda)(I-qR(\lambda\pm{i0})).
\end{equation*}
Then we have
\begin{lemma}
    For $s>1/2$ and $\lambda\in(-3,3)\setminus(\{\pm{1},0\}\cup\sigma_{pp}(H))$,
    \begin{equation*}
        {\cal{F}}^{(\pm)}(\lambda)\in{\bold{B}}({\cal{H}}^s,L^2({\cal{M}}_{\vert\lambda\vert};\mathbb{C}^2)).
    \end{equation*}
\end{lemma}
\begin{proof}
    Noticing that $q\in{\bold{B}}({\cal{H}}^{-s},{\cal{H}}^s)$, this is an
    immediate consequence of Corollary \ref{cor_limitng_absorption_principle}
    and Lemma \ref{F_0_continuity}.
\end{proof}
By the definition of ${\cal{F}}^{(\pm)}(\lambda)$ and \eqref{eigen_op_2}, we
have
\begin{equation}
    \label{eigen_op_3}
    {\cal{F}}^{(\pm)}(\lambda)(H-\lambda)=0.
\end{equation}

The next lemma is proved by calculating
$2\pi{i}{\cal{F}}^{(\pm)}(\lambda)^*{\cal{F}}^{(\pm)}(\lambda)$ in the same
way as in Lemma 3.3 of Isozaki and Korotyaev \cite{2011arXiv1103.2193I}.
\begin{lemma}
    For $f$, $g\in{\cal{H}}^s$, $s>1/2$,
    \begin{equation*}
        \frac{1}{2\pi{i}}((R(\lambda+i0)-R(\lambda-i0))f,g)_{L^2(\mathbb{T}^2;\mathbb{C}^2)}=({\cal{F}}^{(\pm)}(\lambda)f,{\cal{F}}^{(\pm)}(\lambda)g)_{L^2({\cal{M}}_{\vert\lambda\vert};\mathbb{C}^2)}.
    \end{equation*}
\end{lemma}

\subsection{Eigenoperators and spectral representations}

The adjoint operator ${\cal{F}}_0(\lambda)^*$ is defined by
\begin{equation*}
    ({\cal{F}}_0(\lambda)f,\phi)_{L^2({\cal{M}}_{\vert\lambda\vert};\mathbb{C}^2)}=(f,{\cal{F}}_0(\lambda)^*\phi)_{L^2(\mathbb{T}^2;\mathbb{C}^2)}
\end{equation*}
for $f\in{\cal{H}}^s$, $s>1/2$, and
$\phi{\in}L^2({\cal{M}}_{\vert\lambda\vert};\mathbb{C}^2)$. By the definition,
we have
${\cal{F}}_0(\lambda)^*\in{\bold{B}}(L^2({\cal{M}}_{\vert\lambda\vert};\mathbb{C}^2),{\cal{H}}^{-s})$,
which is written as
\begin{align}
    \label{(F_0^*phi)_1}
    \begin{split}
        &({\cal{F}}_0(\lambda)^*\phi)_1(\xi(\vert\rho\vert,\omega))\\
        &=\displaystyle\frac{1}{2}\delta(p(\xi(\vert\rho\vert,\omega))-\vert\lambda\vert)\phi_1(\omega)+\frac{1}{2}\frac{\alpha(\xi(\vert\rho\vert,\omega))}{\lambda}\delta(p(\xi(\vert\rho\vert,\omega))-\vert\lambda\vert)\phi_2(\omega),
    \end{split}
\end{align}
\begin{align}
    \label{(F_0^*phi)_2}
    \begin{split}
        &({\cal{F}}_0(\lambda)^*\phi)_2(\xi(\vert\rho\vert,\omega))\\
        &=\displaystyle\frac{1}{2}\frac{\overline{\alpha}(\xi(\vert\rho\vert,\omega))}{\lambda}\delta(p(\xi(\vert\rho\vert,\omega))-\vert\lambda\vert)\phi_1(\omega)+\frac{1}{2}\delta(p(\xi(\vert\rho\vert,\omega)-\vert\lambda\vert))\phi_2(\omega),
    \end{split}
\end{align}
for
$\phi=(\phi_1(\omega),\phi_2(\omega)){\in}C^{\infty}({\cal{M}}_{\vert\lambda\vert};\mathbb{C}^2)$,
where $\delta(\cdot)$ is Dirac's delta.

By taking the adjoint of \eqref{eigen_op_2}, ${\cal{F}}_0(\lambda)^*$ is the
eigenoperator of $H_0$:
\begin{equation*}
    (H_0-\lambda){\cal{F}}_0(\lambda)^*=0.
\end{equation*}

Let us define the operator ${\cal{F}}_0$ by
\begin{equation*}
    ({\cal{F}}_0f)(\lambda,\omega)=({\cal{F}}_0(\lambda)f)(\omega)
\end{equation*}
for $f\in{\cal{H}}^s$, $s>1/2$. The proof of the next theorem as well as that
for Theorem \ref{spectral_rep_H} can be done in the same way as in Isozaki and
Korotyaev \cite{2011arXiv1103.2193I}. We shall give the sketch in
\ref{appendix_spectral_representation} for the reader's convenience.
\begin{theorem}
    \label{spec_rep_1}
    \begin{enumerate}
        \item ${\cal{F}}_0$ is uniquely extended to a unitary operator from
            $L^2(\mathbb{T}^2;\mathbb{C}^2)$ to
            $L^2((-3,3);L^2({\cal{M}}_{\vert\lambda\vert};\mathbb{C}^2);d\lambda)$.
        \item ${\cal{F}}_0$ diagonalizes $H_0$:
            \begin{equation*}
                ({\cal{F}}_0H_0f)(\lambda)=\lambda({\cal{F}}_0f)(\lambda).
            \end{equation*}
        \item For any compact interval $I\subset(-3,3)\setminus\{\pm{1},0\}$,
            \begin{equation*}
                \int_I{\cal{F}}_0(\lambda)^*g(\lambda)d\lambda\in{L^2(\mathbb{T}^2;\mathbb{C}^2)}
            \end{equation*}
            for
            $g{\in}L^2((-3,3);L^2({\cal{M}}_{\vert\lambda\vert};\mathbb{C}^2);d\lambda)$.
            Moreover, the following inversion formula holds:
            \begin{equation*}
                f=\text{s-}\lim_{N\rightarrow\infty}\int_{I_N}{\cal{F}}_0(\lambda)^*({\cal{F}}_0f)(\lambda)d\lambda
            \end{equation*}
            for $f{\in}L^2(\mathbb{T}^2;\mathbb{C}^2)$, where $I_N$ is a
            finite union of compact intervals in $(-3,3)\setminus\{\pm{1},0\}$
            such that $I_N\rightarrow(-3,3)$ as $N\rightarrow\infty$.
    \end{enumerate}
\end{theorem}

Let us also define the operators ${\cal{F}}^{(\pm)}$ as
\begin{equation*}
    ({\cal{F}}^{(\pm)}f)(\lambda,\omega)=({\cal{F}}^{(\pm)}(\lambda)f)(\omega)
\end{equation*}
for $f\in{\cal H}^s$, $s>1/2$.
\begin{theorem}
    \label{spectral_rep_H}
    \begin{enumerate}
        \item ${\cal{F}}^{(\pm)}$ are uniquely extended to partial isometries
            from $L^2(\mathbb{T}^2;\mathbb{C}^2)$ to
            $L^2((-3,3);L^2({\cal{M}}_{\vert\lambda\vert};\mathbb{C});d\lambda)$
            with the initial set ${\cal{H}}_{ac}(H)$ and the final set
            $L^2((-3,3);L^2({\cal{M}}_{\vert\lambda\vert};\mathbb{C}^2);d\lambda)$. Moreover,
            ${\cal{F}}^{(\pm)}$ diagonalize $H$:
            \begin{equation*}
                ({\cal{F}}^{(\pm)}Hf)(\lambda)=\lambda({\cal{F}}^{(\pm)}f)(\lambda)
            \end{equation*}
            for $f{\in}L^2(\mathbb{T}^2;\mathbb{C}^2)$.
        \item The following inversion formula holds:
            \begin{equation*}
                f=\text{s-}\lim_{N\rightarrow\infty}\int_{I_N}{\cal{F}}^{(\pm)}(\lambda)^*({\cal{F}}^{(\pm)}f)(\lambda)d\lambda
            \end{equation*}
            for $f\in{\cal{H}}_{ac}(H)$, where $I_N$ is a finite union of
            compact intervals in
            $(-3,3)\setminus\bigl(\{\pm{1},0\}\cup\sigma_{pp}(H)\bigr)$ such
            that $I_N\rightarrow(-3,3)$ as $N\rightarrow\infty$.
        \item
            ${\cal{F}}^{(\pm)}\in{\bold{B}}(L^2({\cal{M}}_{\vert\lambda\vert};\mathbb{C}^2),{\cal{H}}^{-s})$ 
            are the eigenoperators for $H$ in the following sense:
            \begin{equation*}
                (H-\lambda){\cal{F}}^{(\pm)}(\lambda)^*\phi=0
            \end{equation*}
            for $\phi{\in}L^2({\cal{M}}_{\vert\lambda\vert};\mathbb{C}^2)$.
    \end{enumerate}
\end{theorem}

Now, we turn to the spectral representation for $\hat{H}$ on the lattice. For
$\omega\in{\cal{M}}_{\vert\lambda\vert}$, let us define the distribution
$\delta_{\omega}(\cdot)\in{\cal{D}}^{'}({\cal{M}}_{\vert\lambda\vert})$ by
\begin{equation*}
    \langle\delta_{\omega},\phi\rangle=\phi(\omega)
\end{equation*}
for $\phi\in{}C^{\infty}({\cal{M}}_{\vert\lambda\vert})$. Note that
${\cal{M}}_{\vert\lambda\vert}$ is an analytic manifold for
$\lambda\in(-3,3)\setminus\{\pm{1},0\}$. Then, from \eqref{(F_0^*phi)_1} and
\eqref{(F_0^*phi)_2}, we have the matrix-valued distribution kernel of
${\cal{F}}_0(\lambda)^*$, which is written as
\begin{align*}
    \psi^{(0)}(\xi;\lambda,\theta)=&\psi^{(0)}(\xi(\vert\rho\vert,\omega);\lambda,\theta)\\
    =&\frac{1}{2}
    \begin{pmatrix}
        1&\lambda^{-1}\alpha(\xi(\vert\rho\vert,\omega))\\
        \lambda^{-1}\overline{\alpha}(\xi(\vert\rho\vert,\omega))&1
    \end{pmatrix}
    \delta(p(\xi(\vert\rho\vert,\omega))-\vert\lambda\vert)\otimes\delta_{\omega}(\theta),
\end{align*}
where $\omega\in{\cal{M}}_{\vert\rho\vert}$ and
$\theta\in{\cal{M}}_{\vert\lambda\vert}$.

In the lattice space, we consider the Fourier coefficients of
$\psi^{(0)}(\xi;\lambda,\theta)$:
\begin{equation}
    \label{hat_psi}
    \hat{\psi}(\lambda,\theta)=(\hat{\psi}^{(0)}(n;\lambda,\theta))_{n\in\mathbb{Z}^2},
\end{equation}
where 
\begin{align}
    \label{hat_psi^(0)}
    \begin{split}
        \hat{\psi}^{(0)}(n;\lambda,\theta)&=\frac{1}{2\pi}\int_{\mathbb{T}^2}\psi^{(0)}(\xi;\lambda,\theta)e^{-in\xi}d\xi\\
        &=\frac{1}{2}\frac{1}{2\pi}J(\vert\lambda\vert,\theta)e^{in\xi(\vert\lambda\vert,\theta)}
        \begin{pmatrix}
            1&\lambda^{-1}\alpha(\xi(\vert\lambda\vert,\theta))\\
            \lambda^{-1}\overline{\alpha}(\xi(\vert\lambda\vert,\theta))&1
        \end{pmatrix},
    \end{split}
\end{align}
and $J(\vert\lambda\vert,\theta)$ is the density introduced in
\eqref{torus_measure}. We put
\begin{equation*}
    \hat{{\cal{F}}}_0={\cal{F}}_0{\cal{F}}.
\end{equation*}
Then we have
\begin{align}
    \label{cal_F_0_1}
    \begin{split}
        &({\cal{F}}_0(\lambda){\cal{F}}\hat{f})_1(\theta)\\
        =&\displaystyle\frac{1}{2}\frac{1}{2\pi}\sum_{n\in\mathbb{Z}^2}e^{in\xi(\vert\lambda\vert,\theta)}\hat{f}_1(n)+\frac{1}{2}\frac{1}{2\pi}\frac{\alpha(\xi(\vert\lambda\vert,\theta))}{\lambda}\sum_{n\in\mathbb{Z}^2}e^{in\xi(\vert\lambda\vert,\theta)}\hat{f}_2(n),
    \end{split}
\end{align}
\begin{align}
    \label{cal_F_0_2}
    \begin{split}
        &({\cal{F}}_0(\lambda){\cal{F}}\hat{f})_2(\theta)\\
        =&\displaystyle\frac{1}{2}\frac{1}{2\pi}\frac{\overline{\alpha}(\xi(\vert\lambda\vert,\theta))}{\lambda}\sum_{n\in\mathbb{Z}^2}e^{in\xi(\vert\lambda\vert,\theta)}\hat{f}_1(n)+\frac{1}{2}\frac{1}{2\pi}\sum_{n\in\mathbb{Z}^2}e^{in\xi(\vert\lambda\vert,\theta)}\hat{f}_2(n),
    \end{split}
\end{align}
for $\hat{f}=(\hat{f}_1,\hat{f}_2)\in\hat{{\cal{H}}}^s$, $s>1/2$. Also, we
have
\begin{equation*}
    \hat{{\cal{F}}}_0^*\phi(\lambda)=(\hat{{\cal{F}}}_0^*\phi(n;\lambda))_{n\in\mathbb{Z}^2},
\end{equation*}
for $\phi=(\phi_1,\phi_2){\in}L^2({\cal{M}}_{\vert\lambda\vert};\mathbb{C}^2)$, where
\begin{align*}
    \hat{{\cal{F}}}_0^*\phi(n;\lambda)&=(\hat{{\cal{F}}}_0(\lambda)^*\phi)(n)\\
    &=\int_{{\cal{M}}_{\vert\lambda\vert}}\hat{\psi}^{(0)}(n;\lambda,\theta)\phi(\theta)d{\cal{M}}_{\vert\lambda\vert}(\theta)\\
    &=
    \begin{pmatrix}
        (\hat{{\cal{F}}}_0(\lambda)^*\phi)_1(n)\\
        (\hat{{\cal{F}}}_0(\lambda)^*\phi)_2(n)
    \end{pmatrix},
\end{align*}
\begin{align}
    \label{cal_F_*_0_1}
    \begin{split}
        (\hat{{\cal{F}}}_0(\lambda)^*\phi)_1(n)&=\frac{1}{2\pi}\int_{{\cal{M}}_{\vert\lambda\vert}}\frac{1}{2}e^{-in\xi(\vert\lambda\vert,\theta)}\phi_1(\theta)d{\cal{M}}_{\vert\lambda\vert}(\theta)\\
        &\quad+\frac{1}{2\pi}\int_{{\cal{M}}_{\vert\lambda\vert}}\frac{1}{2}\frac{\alpha(\xi(\vert\lambda\vert,\theta))}{\lambda}e^{-in\xi(\vert\lambda\vert,\theta)}\phi_2(\theta)d{\cal{M}}_{\vert\lambda\vert}(\theta),
    \end{split}
\end{align}
\begin{align}
    \label{cal_F_*_0_2}
    \begin{split}
        (\hat{{\cal{F}}}_0(\lambda)^*\phi)_2(n)&=\displaystyle\frac{1}{2\pi}\int_{{\cal{M}}_{\vert\lambda\vert}}\frac{1}{2}\frac{\overline{\alpha}(\xi(\vert\lambda\vert,\theta))}{\lambda}e^{-in\xi(\vert\lambda\vert,\theta)}\phi_1(\theta)d{\cal{M}}_{\vert\lambda\vert}(\theta)\\
        &\quad+\displaystyle\frac{1}{2\pi}\int_{{\cal{M}}_{\vert\lambda\vert}}\frac{1}{2}e^{-in\xi(\vert\lambda\vert,\theta)}\phi_2(\theta)d{\cal{M}}_{\vert\lambda\vert}(\theta).
    \end{split}
\end{align}

A direct calculation shows that the resolvent $R_0(z)=(H_0-z)^{-1}$ is a
multiplication operator by the matrix
\begin{equation*}
    R_0(z,\xi)=\frac{1}{z^2-\vert\alpha(\xi)\vert}
    \begin{pmatrix}
        -z&-\alpha(\xi)\\
        -\overline{\alpha}(\xi)&-z
    \end{pmatrix}.
\end{equation*}
Let $\hat{r}_0(z)=(\hat{r}_0(z,n))_{n\in\mathbb{Z}^2}$ be the Fourier
coefficients of $R_0(z,\xi)$, which is written as
\begin{align}\label{hat_r_0}
    \begin{split}
        \hat{r}_0(z,n)&=
        \begin{pmatrix}
            \hat{r}_{0,11}(z,n)&\hat{r}_{0,12}(z,n)\\
            \hat{r}_{0,21}(z,n)&\hat{r}_{0,22}(z,n)
        \end{pmatrix}\\
        &=\frac{1}{2\pi}\int_{\mathbb{T}^2}R_0(z,\xi)e^{-in\xi}d\xi.
    \end{split}
\end{align}
Then the resolvent $\hat{R}_0(z)=(\hat{H}_0-z)^{-1}$ is a convolution
operator by $\hat{r}_0(z)$:
\begin{align}
    \label{hat_R_0_hat_f}
    \begin{split}
        (\hat{R}_0(z)\hat{f})(n)&=
        \begin{pmatrix}
            (\hat{r}_{0,11}(z)*\hat{f}_1)(n)+(\hat{r}_{0,12}(z)*\hat{f}_2)(n)\\
            (\hat{r}_{0,21}(z)*\hat{f}_1)(n)+(\hat{r}_{0,22}(z)*\hat{f}_2)(n)
        \end{pmatrix}\\
        &=
        \begin{pmatrix}
            \displaystyle\sum_{m\in\mathbb{Z}^2}\hat{r}_{0,11}(z,n-m)\hat{f}_1(m)+\sum_{m\in\mathbb{Z}^2}\hat{r}_{0,12}(z,n-m)\hat{f}_2(m)\\
            \displaystyle\sum_{m\in\mathbb{Z}^2}\hat{r}_{0,21}(z,n-m)\hat{f}_1(m)+\sum_{m\in\mathbb{Z}^2}\hat{r}_{0,22}(z,n-m)\hat{f}_2(m)
        \end{pmatrix},
    \end{split}
\end{align}
for $\hat{f}=(\hat{f}_1,\hat{f}_2){\in}l^2(\mathbb{Z}^2;\mathbb{C}^2)$.

We put
\begin{equation*}
    \hat{{\cal{F}}}^{(\pm)}={\cal{F}}^{(\pm)}{\cal{F}}.
\end{equation*}
Passing to the Fourier transform ${\cal{F}}$ in Theorem \ref{spectral_rep_H},
we have
\begin{theorem}
    \begin{enumerate}
        \item $\hat{{\cal{F}}}^{(\pm)}$ are uniquely extended to partial
            isometries from $l^2(\mathbb{Z}^2;\mathbb{C}^2)$ to
            $L^2((-3,3);L^2({\cal{M}}_{\vert\lambda\vert};\mathbb{C}^2);d\lambda)$
            with the initial set ${\cal{H}}_{ac}(\hat{H})$ and the final set
            $L^2((-3,3);L^2({\cal{M}}_{\vert\lambda\vert};\mathbb{C}^2);d\lambda)$. Moreover,
            $\hat{{\cal{F}}}^{(\pm)}$ diagonalize $\hat{H}$:
            \begin{equation*}
                (\hat{{\cal{F}}}^{(\pm)}\hat{H}\hat{f})(\lambda)=\lambda(\hat{{\cal{F}}}^{(\pm)}\hat{f})(\lambda)
            \end{equation*}
            for $f{\in}l^2(\mathbb{Z}^2;\mathbb{C}^2)$.
        \item The following inversion formula holds:
            \begin{equation*}
                \hat{f}=\text{s-}\lim_{N\rightarrow\infty}\int_{I_N}\hat{{\cal{F}}}^{(\pm)}(\lambda)^*(\hat{{\cal{F}}}^{(\pm)}\hat{f})(\lambda)d\lambda
            \end{equation*}
            for $f\in{\cal{H}}_{ac}(\hat{H})$, where $I_N$ is a finite union
            of compact intervals in
            $(-3,3)\setminus\bigl(\{\pm{1},0\}\cup\sigma_{pp}(H)\bigr)$ such
            that $I_N\rightarrow(-3,3)$ as $N\rightarrow\infty$.
        \item $\hat{{\cal{F}}}^{(\pm)}$ diagonalize $\hat{H}$ in the following
            sense:
            \begin{equation*}
                (\hat{{\cal{F}}}^{(\pm)}\hat{H}\hat{f})(\lambda)=\lambda(\hat{{\cal{F}}}^{(\pm)}\hat{f})(\lambda)
            \end{equation*}
            for $\hat{f}\in{\cal{H}}_{ac}(\hat{H})$.
    \end{enumerate}
\end{theorem}

\subsection{Scattering matrix}
\label{section_scattering_matrix}

Since $\hat{q}(n)$ is compactly supported, we can use the trace class
perturbation theory for $\hat{H}$ and $\hat{H}_0$, that is, the wave operators
\begin{equation*}
    W_{\pm}=\text{s-}\lim_{t\pm\infty}e^{it\hat{H}}e^{-it\hat{H}_0}
\end{equation*}
exist, and they are complete (Theorem XI.8, Reed-Simon \cite{MR529429}). The
scattering operator is defined by
\begin{equation*}
    S=W_{+}^*W_{-}.
\end{equation*}

We put
\begin{equation*}
    \hat{S}=\hat{{\cal{F}}}_0S\hat{{\cal{F}}}_0^*.
\end{equation*}
Since $S$ commutes with $\hat{H}_0$ and $\hat{{\cal{F}}}_0$ diagonalizes
$\hat{H_0}$, $\hat{S}$ is decomposable on
$L^2((-3,3);L^2({\cal{M}}_{\vert\lambda\vert};\mathbb{C}^2);d\lambda)=\int_{-3}^3\oplus{}L^2({\cal{M}}_{\vert\lambda\vert};\mathbb{C}^2)d\lambda$:
\begin{equation*}
    \hat{S}=\int_{-3}^3\oplus\hat{S}(\lambda)d\lambda,
\end{equation*}
where $\hat{S}(\lambda)$ is the S-matrix, which is unitary on
$L^2({\cal{M}}_{\vert\lambda\vert};\mathbb{C}^2)$.

Let us define the scattering amplitude $A(\lambda)$ by
\begin{equation*}
    \hat{S}(\lambda)=I-2{\pi}iA(\lambda).
\end{equation*}
Then, by using abstract stationary scattering theory (Kato and Kuroda
\cite{MR0385604}, Kuroda \cite{MR0326435,MR0326436}), we have
\begin{equation*}
    A(\lambda)=\hat{{\cal{F}}}_0(\lambda)\hat{q}\hat{{\cal{F}}}_0(\lambda)^*-\hat{{\cal{F}}}_0(\lambda)\hat{q}\hat{R}(\lambda+i0)\hat{q}\hat{{\cal{F}}}_0(\lambda)^*.
\end{equation*}

Let $A(\lambda,\theta,\theta^{\prime})$ be the integral kernel of
$A(\lambda)$, which is written as
\begin{equation*}
    A(\lambda,\theta,\theta^{\prime})=
    \begin{pmatrix}
        A_{11}(\lambda,\theta,\theta^{\prime})&A_{12}(\lambda,\theta,\theta^{\prime})\\
        A_{21}(\lambda,\theta,\theta^{\prime})&A_{22}(\lambda,\theta,\theta^{\prime})
    \end{pmatrix}.
\end{equation*}
The potential $\hat{q}$ is a multiplication operator by the $2{\times}2$
matrices $(\hat{q}(n))_{n\in\mathbb{Z}^2}$ defined by \eqref{q_n}, and the
action of $\hat{R}(z)$ has already been given by \eqref{hat_r_0} and
\eqref{hat_R_0_hat_f}. Then, by using \eqref{cal_F_0_1}, \eqref{cal_F_0_2},
\eqref{cal_F_*_0_1} and \eqref{cal_F_*_0_2}, we have the following
representations for $A_{jj}(\lambda,\theta,\theta^{\prime})$, $j\in\{1,2\}$:
\begin{align*}
    \begin{split}
        &A_{11}(\lambda,\theta,\theta^{\prime})\\
        =&\frac{1}{4}\frac{1}{(2\pi)^2}J(\vert\lambda\vert,\theta)J(\vert\lambda\vert,\theta^{\prime})\sum_{n\in\mathbb{Z}^2}e^{in(\xi(\vert\lambda\vert,\theta)-\xi(\vert\lambda\vert,\theta^{\prime}))}\\
        &\quad\cdot\bigl(\hat{q}_1(n)+\frac{\alpha(\xi(\vert\lambda\vert,\theta))}{\lambda}\frac{\overline{\alpha}(\xi(\vert\lambda\vert,\theta^{\prime}))}{\lambda}\hat{q}_2(n)\bigr)\\
        &+\frac{1}{2}\frac{1}{2\pi}J(\vert\lambda\vert,\theta)\Bigl(\sum_{n\in\mathbb{Z}^2}e^{in\xi(\vert\lambda\vert,\theta)}\hat{q}_1(n)\bigl(\hat{R}(\lambda+i0)\hat{q}\hat{\psi}^{(0)}(\lambda,\theta^{\prime})\bigr)_{11}(n)\\
        &\quad\quad+\frac{\alpha(\xi(\vert\lambda\vert,\theta))}{\lambda}\sum_{n\in\mathbb{Z}^2}e^{in\xi(\vert\lambda\vert,\theta)}\hat{q}_2(n)\bigl(\hat{R}(\lambda+i0)\hat{q}\hat{\psi}^{(0)}(\lambda,\theta^{\prime})\bigr)_{21}(n)\Bigr),
    \end{split}
\end{align*}
\begin{align*}
    \begin{split}
        &A_{22}(\lambda,\theta,\theta^{\prime})\\
        =&\frac{1}{4}\frac{1}{(2\pi)^2}J(\vert\lambda\vert,\theta)J(\vert\lambda\vert,\theta^{\prime})\sum_{n\in\mathbb{Z}^2}e^{in(\xi(\vert\lambda\vert,\theta)-\xi(\vert\lambda\vert,\theta^{\prime}))}\\
        &\quad\cdot\bigl(\frac{\overline{\alpha}(\xi(\vert\lambda\vert,\theta))}{\lambda}\frac{\alpha(\xi(\vert\lambda\vert,\theta^{\prime}))}{\lambda}\hat{q}_1(n)+\hat{q}_2(n)\bigr)\\
        &+\frac{1}{2}\frac{1}{2\pi}J(\vert\lambda\vert,\theta)\Bigl(\frac{\overline{\alpha}(\xi(\vert\lambda\vert,\theta))}{\lambda}\sum_{n\in\mathbb{Z}^2}e^{in\xi(\vert\lambda\vert,\theta)}\hat{q}_1(n)\bigl(\hat{R}(\lambda+i0)\hat{q}\hat{\psi}^{(0)}(\lambda,\theta^{\prime})\bigr)_{12}(n)\\
        &\quad\quad+\sum_{n\in\mathbb{Z}^2}e^{in\xi(\vert\lambda\vert,\theta)}\hat{q}_2(n)\bigl(\hat{R}(\lambda+i0)\hat{q}\hat{\psi}^{(0)}(\lambda,\theta^{\prime})\bigr)_{22}(n)\Bigr).
    \end{split}
\end{align*}
We omit the formulas for $A_{ij}(\lambda,\theta,\theta^{\prime})$,
$(i,j)=(1,2)$ or $(2,1)$, since we do not use them later.


%
%
\section{Analytic continuation}
\label{analytic_continuation}

A simple calculation shows that if $\xi\in{\cal{M}}_{\lambda}$
\begin{equation*}
    \cos^2{\frac{\xi_1}{2}}\cos^2{\frac{\xi_2}{2}}+\sin{\frac{\xi_1}{2}}\sin{\frac{\xi_2}{2}}\cos{\frac{\xi_1}{2}}\cos{\frac{\xi_2}{2}}=\frac{1}{8}(\lambda^2-1),
\end{equation*}
which is rewritten as
\begin{equation}
    \label{cosh^2-sinh^2=1}
    (\cos{\frac{\xi_1}{2}}\cos{\frac{\xi_2}{2}}+\frac{1}{2}\sin{\frac{\xi_1}{2}}\sin{\frac{\xi_2}{2}})^2-(\frac{1}{2}\sin{\frac{\xi_1}{2}}\sin{\frac{\xi_2}{2}})^2=\frac{1}{8}(\lambda^2-1).
\end{equation}
We let $\lambda=\sqrt{8k^2+1}$ so that $k$ varies over $(-1,0)\cup(0,1)$. Then
we have
\begin{equation}
    \label{k_cosh_theta}
    \cos{\frac{\xi_1}{2}}\cos{\frac{\xi_2}{2}}+\frac{1}{2}\sin{\frac{\xi_1}{2}}\sin{\frac{\xi_2}{2}}=\pm{k}\cosh{\theta},
\end{equation}
\begin{equation}
    \label{k_sinh_theta}
    \frac{1}{2}\sin{\frac{\xi_1}{2}}\sin{\frac{\xi_2}{2}}=\pm{k}\sinh{\theta},
\end{equation}
where $\theta\in\mathbb{R}$. We parametrize $\xi\in{\cal{M}}_{\lambda}$ by
using \eqref{k_cosh_theta} and \eqref{k_sinh_theta}, for which we need to
solve \eqref{k_cosh_theta} and \eqref{k_sinh_theta} with respect to $\xi_1$
and $\xi_2$.

From \eqref{k_cosh_theta} and \eqref{k_sinh_theta}, we have
\begin{equation}
    \label{cos^2_xi_1/2cos^2_xi_2/2}
    \cos^2{\frac{\xi_1}{2}}\cos^2{\frac{\xi_2}{2}}=k^2e^{-2\theta},
\end{equation}
\begin{equation}
    \label{sin^2_xi_1/2sin^2_xi_2/2}
    \sin^2{\frac{\xi_1}{2}}\sin^2{\frac{\xi_2}{2}}=4k^2\sinh^2{\theta}.
\end{equation}
By solving \eqref{cos^2_xi_1/2cos^2_xi_2/2} and
\eqref{sin^2_xi_1/2sin^2_xi_2/2} with respect to $\sin^2{\frac{\xi_j}{2}}$ and
$\cos^2{\frac{\xi_j}{2}}$, $j\in\{1,2\}$, we put
\begin{equation}
    \label{sin^2_xi_1/2}
    \sin^2{\frac{\xi_1}{2}}=\frac{1}{2}\left\{1-\left(2-e^{2\theta}\right)k^2+\sqrt{\left(2-e^{2\theta}\right)^2k^4-2\left(2e^{-2\theta}-\left(2-e^{2\theta}\right)\right)k^2+1}\right\},
\end{equation}
\begin{equation}
    \label{sin^2_xi_2/2}
    \sin^2{\frac{\xi_2}{2}}=\frac{1}{2}\left\{1-\left(2-e^{2\theta}\right)k^2-\sqrt{\left(2-e^{2\theta}\right)^2k^4-2\left(2e^{-2\theta}-\left(2-e^{2\theta}\right)\right)k^2+1}\right\},
\end{equation}
\begin{equation}
    \label{cos^2_xi_1/2}
    \cos^2{\frac{\xi_1}{2}}=\frac{1}{2}\left\{1+\left(2-e^{2\theta}\right)k^2-\sqrt{\left(2-e^{2\theta}\right)^2k^4-2\left(2e^{-2\theta}-\left(2-e^{2\theta}\right)\right)k^2+1}\right\},
\end{equation}
\begin{equation}
    \label{cos^2_xi_2/2}
    \cos^2{\frac{\xi_2}{2}}=\frac{1}{2}\left\{1+\left(2-e^{2\theta}\right)k^2+\sqrt{\left(2-e^{2\theta}\right)^2k^4-2\left(2e^{-2\theta}-\left(2-e^{2\theta}\right)\right)k^2+1}\right\}.
\end{equation}
In the following, we consider the case that $\xi_1,\xi_2\in(0,\pi)$ so that
$\sin{\frac{\xi_1}{2}}>0$ and $\sin{\frac{\xi_2}{2}}>0$, and define
$\xi_1(k,\theta)$, $\xi_2(k,\theta)$ by the equations \eqref{sin^2_xi_1/2},
\eqref{sin^2_xi_2/2}, \eqref{cos^2_xi_1/2} and \eqref{cos^2_xi_2/2}. By a
simple calculation, if $0<\vert{k}\vert<1$ and $0<\theta<\frac{1}{2}\log{2}$,
we have
\begin{equation*}
    0<\frac{1}{2}\left\{1\pm\left(2-e^{2\theta}\right)k^2\pm\sqrt{\left(2-e^{2\theta}\right)^2k^4-2\left(2e^{-2\theta}-\left(2-e^{2\theta}\right)\right)k^2+1}\right\}<1.
\end{equation*}
Note that
\begin{equation*}
    2-e^{2\theta}>0
\end{equation*}
if $0<\theta<\frac{1}{2}\log{2}$.

We put
\begin{equation}
    \label{a(theta)}
    a=a(\theta)=\sqrt{2-e^{2\theta}},
\end{equation}
\begin{equation}
    \label{b(theta)}
    b=b(\theta)=\frac{2\sinh{\theta}}{\sqrt{2-e^{2\theta}}}.
\end{equation}
Then we have
\begin{lemma}
    \label{lemma_analytic_continuation}
    Assume $0<\theta<\frac{1}{2}\log{2}$. Then $\xi_j(k,\theta)$ has the
    analytic continuations $\zeta_j(z,\theta)$, $j\in\{1,2\}$, with the
    following properties:
    \begin{enumerate}
        \item $\zeta_j(z,\theta)$, $j\in\{1,2\}$, are analytic with respect to
            $z\in\mathbb{C}_+=\{z\in\mathbb{C};\Im{z}>0\}$.
        \item If $k\in(-1,0)\cup(0,1)$, we have
            $\zeta_j(k+i0,\theta)=\xi_j(k,\theta)$, $j\in\{1,2\}$.
        \item If $0<\xi_j<\pi$, $j\in\{1,2\}$, there exist $m_1$,
            $m_2\in\mathbb{Z}$ such that
            \begin{equation}
                \label{Re_zeta_1}
                \Re{\zeta_1(1+iN,\theta)}=2m_1\pi+{\cal{O}}(N^{-3}),
            \end{equation}
            \begin{equation}
                \label{Im_zeta_1}
                \Im{\zeta_1(1+iN,\theta)}=2\log{(b+\sqrt{b^2+1})}+{\cal{O}}(N^{-2}),
            \end{equation}
            \begin{equation}
                \label{Re_zeta_2}
                \Re{\zeta_2(1+iN,\theta)}=(2m_2+1)\pi+{\cal{O}}(N^{-1}),
            \end{equation}
            \begin{equation}
                \label{Im_zeta_2}
                \Im{\zeta_2(1+iN,\theta)}=2\log{N}+\log{(4a^2)}+{\cal{O}}(N^{-2}),
            \end{equation}
            as $N\rightarrow+\infty$.
        \end{enumerate}
\end{lemma}
\begin{proof}
    We put
    \begin{equation*}
        f_{\pm}(x)=\frac{1}{2}\left\{1-(2-e^{2\theta})x\pm\sqrt{\left(1-\left(2-e^{2\theta}\right)x^2\right)^2-4^2x^2\sinh^2{\theta}}\right\}
    \end{equation*}
    for $x\ge0$. Note that
    \begin{equation*}
        0\le{}f_{\pm}(x)\le1,
    \end{equation*}
    \begin{equation*}
        f_-(0)=0,\ f_+(0)=1,
    \end{equation*}
    \begin{equation*}
        \lim_{x\rightarrow+\infty}f_-(x)=1,\ \lim_{x\rightarrow+\infty}f_+(x)=0.
    \end{equation*}
    A simple calculation shows that $f_-(x)$ (resp. $f_+(x)$) are
    increasing (resp. decreasing) monotonically.

    We also put
    \begin{align*}
        g(z)&=\left(1-\left(2-e^{2\theta}\right)z^2\right)^2-4^2z^2\sinh^2{\theta}\\
        &=(2-e^{2\theta})^2z^4-2\left(2e^{-2\theta}-\left(2-e^{2\theta}\right)\right)z^2+1
    \end{align*}
    Another simple calculation shows that all the zeros of $g(z)$ belong to
    $\mathbb{R}$. Then we have the analytic continuations of $f_{\pm}(z^2)$
    with respect to $z\in\mathbb{C}_+$.

    Furthermore, another simple calculation shows that
    $0\le{}f_{\pm}(z^2)\le{1}$ if and only if $z\in\mathbb{R}$. Therefore, we
    have the analytic continuations $\zeta_j(z,\theta)$ of $\xi_j(k,\theta)$,
    $j\in\{1,2\}$, $z\in\mathbb{C}_+$, which have the imaginary parts of fixed
    signs.

    Let $\vert{z}\vert\rightarrow\infty$. Then we have the following
    asymptotic behaviors:
    \begin{equation}
        \label{sin^2_zeta_1/2}
        \sin^2{\frac{\zeta_1}{2}}=-4\frac{\sinh^2{\theta}}{2-e^{2\theta}}-4\frac{e^{-2\theta}\sinh^2{\theta}}{(2-e^{2\theta})^3}z^{-2}+{\cal{O}}(\vert{z}\vert^{-4}),
    \end{equation}
    \begin{equation}
        \label{sin^2_zeta_2/2}
        \sin^2{\frac{\zeta_2}{2}}=-(2-e^{2\theta})z^2+\frac{e^{-2\theta}}{2-e^{2\theta}}+4\frac{e^{-2\theta}\sinh^2{\theta}}{(2-e^{2\theta})^3}z^{-2}+{\cal{O}}(\vert{z}\vert^{-4}),
    \end{equation}
    \begin{equation}
        \label{cos^2_zeta_1/2}
        \cos^2{\frac{\zeta_1}{2}}=\frac{e^{-2\theta}}{2-e^{2\theta}}+4\frac{e^{-2\theta}\sinh^2{\theta}}{(2-e^{2\theta})^3}z^{-2}+{\cal{O}}(\vert{z}\vert^{-4}),
    \end{equation}
    \begin{equation}
        \label{cos^2_zeta_2/2}
        \cos^2{\frac{\zeta_2}{2}}=(2-e^{2\theta})z^2-4\frac{\sinh^2{\theta}}{2-e^{2\theta}}-4\frac{e^{-2\theta}\sinh^2{\theta}}{(2-e^{2\theta})^3}z^{-2}+{\cal{O}}(\vert{z}\vert^{-4}).
    \end{equation}
    Note that $\sin{\frac{\xi_j}{2}}>0$ and $\cos{\frac{\xi_j}{2}}>0$, since
    $0<\xi_j<\pi$, $j\in\{1,2\}$. Therefore, we have
    \begin{equation}
        \label{sin_zeta_1/2}
        \sin{\frac{\zeta_1}{2}}=i\frac{2\sinh{\theta}}{\sqrt{2-e^{2\theta}}}\left\{1+\frac{e^{-2\theta}}{2\left(2-e^{2\theta}\right)^2}z^{-2}+{\cal{O}}(\vert{z}\vert^{-4})\right\},
    \end{equation}
    \begin{equation}
        \label{sin_zeta_2/2}
        \sin{\frac{\zeta_2}{2}}=i\sqrt{2-e^{2\theta}}z\left\{1-\frac{e^{-2\theta}}{2\left(2-e^{2\theta}\right)^2}z^{-2}+{\cal{O}}(\vert{z}\vert^{-4})\right\},
    \end{equation}
    \begin{equation}
        \label{cos_zeta_1/2}
        \cos{\frac{\zeta_1}{2}}=\frac{e^{-\theta}}{\sqrt{2-e^{2\theta}}}\left\{1+2\frac{\sinh^2{\theta}}{\left(2-e^{2\theta}\right)^2}z^{-2}+{\cal{O}}(\vert{z}\vert^{-4})\right\},
    \end{equation}
    \begin{equation}
        \label{cos_zeta_2/2}
        \cos{\frac{\zeta_2}{2}}=\sqrt{2-e^{2\theta}}z\left\{1-2\frac{\sinh^2{\theta}}{\left(2-e^{2\theta}\right)^2}z^{-2}+{\cal{O}}(\vert{z}\vert^{-4})\right\}.
    \end{equation}

    Let $E_j=E_j(z)$ be the analytic continuations of $\xi_j$,
    $j\in\{1,2\}$. We put $z=1+iN$. Then, from \eqref{sin_zeta_1/2} and
    \eqref{cos_zeta_2/2}, we have
    \begin{equation}
        \label{sin_E_1/2}
        \sin{\frac{E_1(z)}{2}}=x_{s,1}+iy_{s,1},
    \end{equation}
    \begin{equation}
        \label{cos_E_2/2}
        \cos{\frac{E_2(z)}{2}}=x_{c,2}+iy_{c,2},
    \end{equation}
    where
    \begin{equation}
        \label{x_s,1}
        x_{s,1}=b(\theta)\left\{\frac{e^{-2\theta}}{\left(2-e^{2\theta}\right)^2}N^{-3}+{\cal{O}}(N^{-5})\right\},
    \end{equation}
    \begin{equation}
        \label{y_s,1}
        y_{s,1}=b(\theta)\left\{1-\frac{e^{-2\theta}}{2\left(2-e^{2\theta}\right)^2}N^{-2}+{\cal{O}}(N^{-4})\right\},
    \end{equation}
    \begin{equation}
        \label{x_c,2}
        x_{c,2}=a(\theta)\left\{1-2\frac{\sinh^2{\theta}}{\left(2-e^{2\theta}\right)^2}N^{-2}+{\cal{O}}(N^{-4})\right\},
    \end{equation}
    \begin{equation}
        \label{y_c,2}
        y_{c,2}=a(\theta)\left\{N+2\frac{\sinh^2{\theta}}{\left(2-e^{2\theta}\right)^2}N^{-1}+{\cal{O}}(N^{-3})\right\},
    \end{equation}
    and $a=a(\theta)$ and $b=b(\theta)$ is defined by \eqref{a(theta)} and
    \eqref{b(theta)}, respectively. We also put $E_j=\eta_j+i\kappa_j$,
    $j\in\{1,2\}$. Then, from \eqref{sin_E_1/2} and \eqref{cos_E_2/2}, we have
    \begin{equation}
        \label{Re_sin_E_1/2}
        \sin{\frac{\eta_1}{2}}\cosh{\frac{\kappa_1}{2}}=x_{s,1},
    \end{equation}
    \begin{equation}
        \label{Im_sin_E_1/2}
        \cos{\frac{\eta_1}{2}}\sinh{\frac{\kappa_1}{2}}=y_{s,1},
    \end{equation}
    \begin{equation}
        \label{Re_cos_E_2/2}
        \cos{\frac{\eta_2}{2}}\cosh{\frac{\kappa_2}{2}}=x_{c,2},
    \end{equation}
    \begin{equation}
        \label{Im_cos_E_2/2}
        -\sin{\frac{\eta_2}{2}}\sinh{\frac{\kappa_2}{2}}=y_{c,2}.
    \end{equation}
 
    For sufficiently large $N>0$, from \eqref{x_s,1} and \eqref{x_c,2}, we
    have $x_{s,1}>0$ and $x_{c,2}>0$, which implies, from \eqref{Re_sin_E_1/2}
    and \eqref{Re_cos_E_2/2}, that $\sin{\frac{\eta_1}{2}}>0$ and
    $\cos{\frac{\eta_2}{2}}>0$, since $\cosh{\frac{\kappa_j}{2}}>0$,
    $j\in\{1,2\}$. We can now prove \eqref{Re_zeta_1}, \eqref{Im_zeta_1},
    \eqref{Re_zeta_2}, and \eqref{Im_zeta_2} in the same way as in Isozaki and
    Korotyaev \cite{2011arXiv1103.2193I} by carefully investigating the
    asymptotic behaviors as $N\rightarrow+\infty$. We shall give the proof of
    \eqref{Re_zeta_1} and \eqref{Im_zeta_1} in
    \ref{proof_of_the_asymptotics_of_zeta_1}, and that of \eqref{Re_zeta_1}
    and \eqref{Im_zeta_1} in \ref{proof_of_the_asymptotics_of_zeta_2},
    respectively.
\end{proof}

\begin{rem}
    In the case that $\xi_1,\xi_2\in(-\pi,0)$, we have
    $\sin{\frac{\xi_j}{2}}<0$ and $\cos{\frac{\xi_j}{2}}>0$,
    $j\in\{1,2\}$. Therefore, also in this case, we can prove the same lemma
    as Lemma \ref{lemma_analytic_continuation}, where we replace the third
    statement in Lemma \ref{lemma_analytic_continuation} with the following
    one:
    \begin{enumerate}
        \renewcommand{\labelenumi}{3$^{\prime}$}
        \item If $-\pi<\xi_j<0$, $j\in\{1,2\}$, there exist $m_3$,
            $m_4\in\mathbb{Z}$ such that
            \begin{equation}
                \label{Re_zeta_1}
                \Re{\zeta_1(1+iN,\theta)}=2m_3\pi+{\cal{O}}(N^{-3}),
            \end{equation}
            \begin{equation}
                \label{Im_zeta_1}
                \Im{\zeta_1(1+iN,\theta)}=-2\log{(b+\sqrt{b^2+1})}+{\cal{O}}(N^{-2}),
            \end{equation}
            \begin{equation}
                \label{Re_zeta_2}
                \Re{\zeta_2(1+iN,\theta)}=(2m_4+1)\pi+{\cal{O}}(N^{-1}),
            \end{equation}
            \begin{equation}
                \label{Im_zeta_2}
                \Im{\zeta_2(1+iN,\theta)}=-2\log{N}-\log{(4a^2)}+{\cal{O}}(N^{-2}),
            \end{equation}
            as $N\rightarrow+\infty$.
    \end{enumerate}
\end{rem}

%
%
\section{Resolvent estimates}

We define the following non-negative quantities: for
$n=(n_1,n_2)\in\mathbb{Z}^2$,
\begin{equation}
    \label{definition_d(n)}
    d(n)=d_{11}(n)=d_{22}(n)=
    \begin{cases}
        \vert{n_1}\vert+\vert{n_2}\vert,&\text{if $n_1\cdot{n_2}\ge{0}$},\\
        \max{\{\vert{n_1}\vert,\vert{n_2}\vert\}},&\text{if
          $n_1\cdot{n_2}\le{0}$},
    \end{cases}
\end{equation}
\begin{equation}
    \label{definition_d_12(n)}
    d_{12}(n)=d_{21}(-n)=
    \begin{cases}
        \vert{n_1}\vert+\vert{n_2}\vert-1,&\text{if $n_1>0,n_2>0$},\\
        \max{\{\vert{n_1}\vert-1,\vert{n_2}\vert\}},&\text{if
          $n_1>0,n_2\le{0}$},\\
        \vert{n_1}\vert+\vert{n_2}\vert,&\text{if $n_1\le0,n_2\le{0}$},\\
        \max{\{\vert{n_1}\vert,\vert{n_2}\vert-1\}},&\text{if
          $n_1\le0,n_2>0$}.
    \end{cases}
\end{equation}
The contour lines of $d_{ij}(\cdot)$, $i,j\in\{1,2\}$, are shown in Figure
\ref{fig_contour_line_d_ij(n)=s}. The quantity $d(\cdot)$ is used to estimate
the support of the Fourier coefficients of $\vert\alpha(\xi)\vert^2$ in Lemma
\ref{Fourier_coefficients_r}, where $\alpha(\xi)$ is defined by \eqref{alpha}.
Also, $d_{ij}(\cdot)$ , $(i,j)=(1,2)$ or $(2,1)$, are used for
$\vert\alpha(\xi)\vert^2\alpha(\xi)$ or
$\vert\alpha(\xi)\vert^2\overline{\alpha}(\xi)$ in Lemmas
\ref{Fourier_coefficients_r_alpha} or
\ref{Fourier_coefficients_r_overline_alpha}, respectively. By making use of
those Lemmas, we can derive the resolvent estimates  at the end this section.
\begin{figure}[htbp]
    \begin{minipage}{0.32\hsize}
        \begin{center}
            \includegraphics[height=3cm]{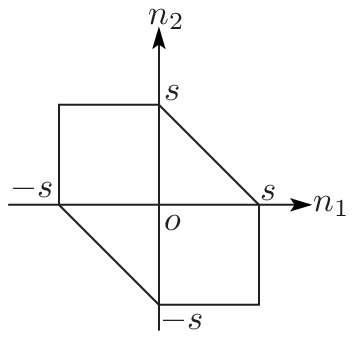}
        \end{center}
    \end{minipage}
    \begin{minipage}{0.32\hsize}
        \begin{center}
            \includegraphics[height=3cm]{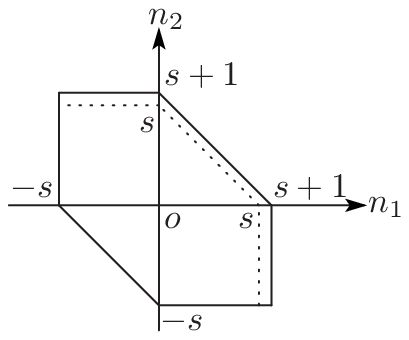}
        \end{center}
    \end{minipage}
    \begin{minipage}{0.32\hsize}
        \begin{center}
            \includegraphics[height=3cm]{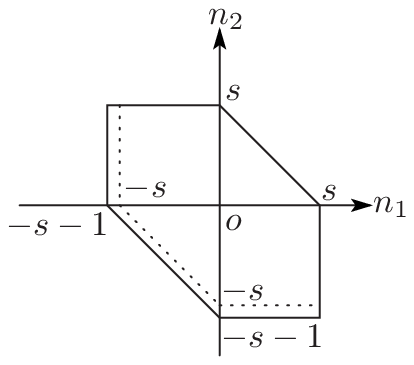}
        \end{center}
    \end{minipage}
    \caption{The shape of contour lines of $d_{ij}(n)=s$; $(i,j)=(1,1)$ or
      $(2,2)$, $(1,2)$, and $(2,1)$ from the left to the
      right.\label{fig_contour_line_d_ij(n)=s}}
\end{figure}

The following two lemmas are immediate consequences from the definitions
\eqref{definition_d(n)} and \eqref{definition_d_12(n)}.
\begin{lemma}
    \label{lemma_d_ij(n)>=|n_2|}
    For $n=(n_1,n_2)\in\mathbb{Z}^2$,
    \begin{equation*}
        d(n)\ge\vert{n_2}\vert,
    \end{equation*}
    \begin{equation*}
        d_{12}(n)\ge
        \begin{cases}
            \vert{n_2}\vert-1,&\text{if $n_2>0$},\\
            \vert{n_2}\vert,&\text{if $n_2\le{0}$},
        \end{cases}
    \end{equation*}
    \begin{equation*}
        d_{21}(n)\ge
        \begin{cases}
            \vert{n_2}\vert,&\text{if $n_2\ge0$},\\
            \vert{n_2}\vert-1,&\text{if $n_2<0$}.
        \end{cases}
    \end{equation*}
\end{lemma}

\begin{lemma}
    \label{d-1=<d_ij=<d}
    For $i,j\in\{1,2\}$,
    \begin{equation*}
        d(n)-1\le{d_{ij}(n)}\le{d(n)}.
    \end{equation*}
\end{lemma}

We can easily show that $d(an)=\vert{a}\vert{}d(n)$ for $a\in\mathbb{Z}$;
also, that $d(n)=0$ if and only if $n=(0,0)$. Moreover, we have
\begin{lemma}
    \label{lemma_d_triangle_ineq}
    $d(\cdot)$ is a norm on $\mathbb{Z}^2$. In particular, it satisfies the
    triangle inequality:
    \begin{equation}
        \label{triangle_inequality_d(n)}
        d(m-n)\le{}d(m-l)+d(l-n).
    \end{equation}
\end{lemma}
\begin{proof}
    We only have to prove \eqref{triangle_inequality_d(n)}. Assume that
    $m_1-n_1\ge0$ and $m_2-n_2\ge0$, where
    $m=(m_1,m_2),n=(n_1,n_2)\in\mathbb{Z}^2$. We treat the following nine
    cases according to the locations of $l=(l_1,l_2)\in\mathbb{Z}^2$:
    \begin{enumerate}[{(}1{)}]
    \item $l_1-n_1\ge0$, $l_2-n_2\ge0$, $m_1-l_1\ge0$, $m_2-l_2\ge0$,
        \label{l_1-n_1=>0,l_2-n_2=>0,m_1-l_1=>0,m_2-l_2=>0}
    \item $l_1-n_1\ge0$, $l_2-n_2\ge0$, $m_1-l_1\ge0$, $m_2-l_2\le0$,
        \label{l_1-n_1=>0,l_2-n_2=>0,m_1-l_1=>0,m_2-l_2=<0}
    \item $l_1-n_1\ge0$, $l_2-n_2\ge0$, $m_1-l_1\le0$, $m_2-l_2\ge0$,
        \label{l_1-n_1=>0,l_2-n_2=>0,m_1-l_1=<0,m_2-l_2=>0}
    \item $l_1-n_1\ge0$, $l_2-n_2\ge0$, $m_1-l_1\le0$, $m_2-l_2\le0$,
        \label{l_1-n_1=>0,l_2-n_2=>0,m_1-l_1=<0,m_2-l_2=<0}
    \item $l_1-n_1\le0$, $l_2-n_2\ge0$, $m_1-l_1\ge0$, $m_2-l_2\ge0$,
        \label{l_1-n_1=<0,l_2-n_2=>0,m_1-l_1=>0,m_2-l_2=>0}
    \item $l_1-n_1\le0$, $l_2-n_2\ge0$, $m_1-l_1\ge0$, $m_2-l_2\le0$,
        \label{l_1-n_1=<0,l_2-n_2=>0,m_1-l_1=>0,m_2-l_2=<0}
    \item $l_1-n_1\ge0$, $l_2-n_2\le0$, $m_1-l_1\ge0$, $m_2-l_2\ge0$,
        \label{l_1-n_1=>0,l_2-n_2=<0,m_1-l_1=>0,m_2-l_2=>0}
    \item $l_1-n_1\ge0$, $l_2-n_2\le0$, $m_1-l_1\le0$, $m_2-l_2\ge0$,
        \label{l_1-n_1=>0,l_2-n_2=<0,m_1-l_1=<0,m_2-l_2=>0}
    \item $l_1-n_1\le0$, $l_2-n_2\le0$, $m_1-l_1\le0$, $m_2-l_2\le0$.
        \label{l_1-n_1=<0,l_2-n_2=<0,m_1-l_1=<0,m_2-l_2=<0}
    \end{enumerate}
    \begin{figure}[htbp]
        \begin{minipage}{0.5\hsize}
            \begin{center}
                \includegraphics[height=4.5cm]{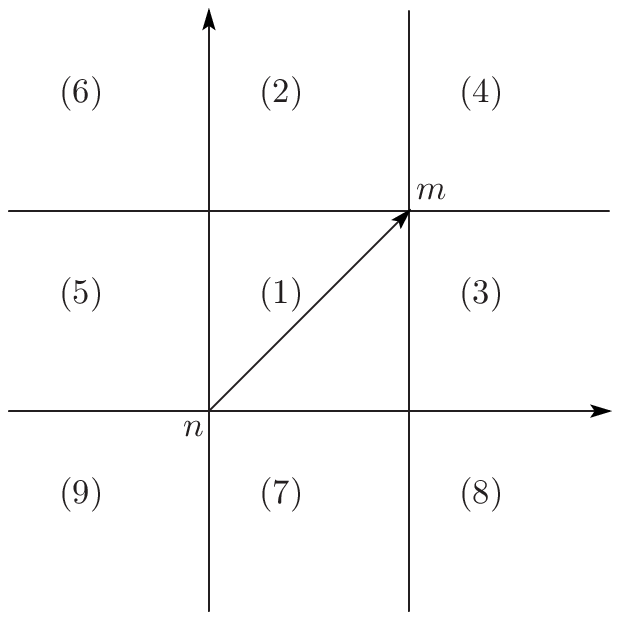}
            \end{center}
        \end{minipage}
        \begin{minipage}{0.5\hsize}
            \begin{center}
                \includegraphics[height=4.5cm]{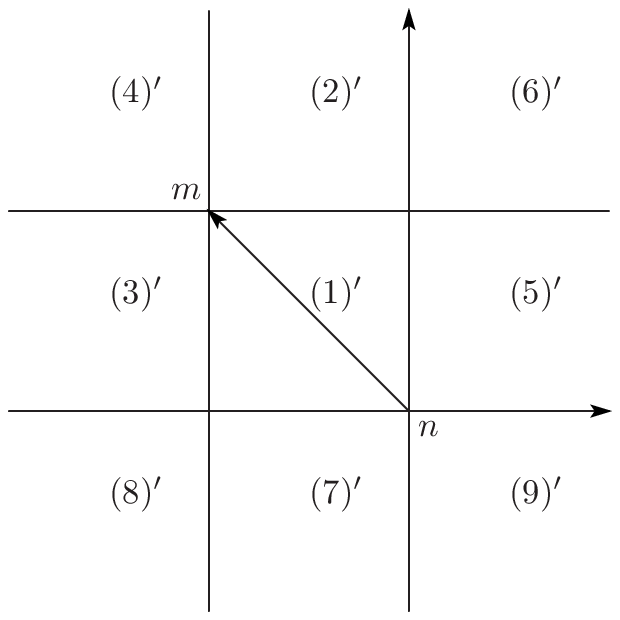}
            \end{center}
        \end{minipage}
        \caption{The left figure shows the location of $l\in\mathbb{Z}^2$
          corresponding to the case from
          \eqref{l_1-n_1=>0,l_2-n_2=>0,m_1-l_1=>0,m_2-l_2=>0} to
          \eqref{l_1-n_1=<0,l_2-n_2=<0,m_1-l_1=<0,m_2-l_2=<0}; the right one
          corresponds that of $l\in\mathbb{Z}^2$ from
          \eqref{l_1-n_1=<0,l_2-n_2=>0,m_1-l_1=<0,m_2-l_2=>0'}$^{\prime}$ to
          \eqref{l_1-n_1=>0,l_2-n_2=<0,m_1-l_1=<0,m_2-l_2=>0'}$^{\prime}$.\label{fig_hexagonal_norm_case}}
    \end{figure}
    The left figure in Figure \ref{fig_hexagonal_norm_case} illustrates the
    location of $l\in\mathbb{Z}^2$ which corresponds to the above nine cases.

    Let us consider, for example, the case
    \eqref{l_1-n_1=>0,l_2-n_2=>0,m_1-l_1=<0,m_2-l_2=>0}. Then we have
    \begin{align*}
        d(m-n)&=\vert{m_1-n_1}\vert+\vert{m_2-n_2}\vert\\
        &=(m_1-l_1)+(l_1-n_1)+\vert{m_2-n_2}\vert.
    \end{align*}
    By the assumption, we have $m_1-l_1\le{0}$, which implies
    \begin{align*}
        d(m-n)&\le\vert{l_1-n_1}\vert+\vert{m_2-l_2}\vert+\vert{l_2-n_2}\vert\\
        &\le\max{\{\vert{m_1-l_1}\vert,\vert{m_2-l_2}\vert\}}+\vert{l_1-n_1}\vert+\vert{l_2-n_2}\vert\\
        &=d(m-l)+d(l-n).
    \end{align*}

    We can treat the other cases in the same way.

    Next, assume that $m_1-n_1\le0$ and $m_2-n_2\ge0$. This time, we treat the
    following nine cases according to the location of $l$:
    \begin{enumerate}[{(}1{)}$^{\prime}$]
    \item $l_1-n_1\le0$, $l_2-n_2\ge0$, $m_1-l_1\le0$, $m_2-l_2\ge0$,
        \label{l_1-n_1=<0,l_2-n_2=>0,m_1-l_1=<0,m_2-l_2=>0'}
    \item $l_1-n_1\le0$, $l_2-n_2\ge0$, $m_1-l_1\le0$, $m_2-l_2\le0$,
        \label{l_1-n_1=<0,l_2-n_2=>0,m_1-l_1=<0,m_2-l_2=<0'}
    \item $l_1-n_1\le0$, $l_2-n_2\ge0$, $m_1-l_1\ge0$, $m_2-l_2\ge0$,
        \label{l_1-n_1=<0,l_2-n_2=>0,m_1-l_1=>0,m_2-l_2=>0'}
    \item $l_1-n_1\le0$, $l_2-n_2\ge0$, $m_1-l_1\ge0$, $m_2-l_2\le0$,
        \label{l_1-n_1=<0,l_2-n_2=>0,m_1-l_1=>0,m_2-l_2=<0'}
    \item $l_1-n_1\ge0$, $l_2-n_2\ge0$, $m_1-l_1\le0$, $m_2-l_2\ge0$,
        \label{l_1-n_1=>0,l_2-n_2=>0,m_1-l_1=<0,m_2-l_2=>0'}
    \item $l_1-n_1\ge0$, $l_2-n_2\ge0$, $m_1-l_1\le0$, $m_2-l_2\le0$,
        \label{l_1-n_1=>0,l_2-n_2=>0,m_1-l_1=<0,m_2-l_2=<0'}
    \item $l_1-n_1\le0$, $l_2-n_2\le0$, $m_1-l_1\le0$, $m_2-l_2\ge0$,
        \label{l_1-n_1=<0,l_2-n_2=<0,m_1-l_1=<0,m_2-l_2=>0'}
    \item $l_1-n_1\le0$, $l_2-n_2\le0$, $m_1-l_1\ge0$, $m_2-l_2\ge0$,
        \label{l_1-n_1=<0,l_2-n_2=<0,m_1-l_1=>0,m_2-l_2=>0'}
    \item $l_1-n_1\ge0$, $l_2-n_2\le0$, $m_1-l_1\le0$, $m_2-l_2\ge0$.
        \label{l_1-n_1=>0,l_2-n_2=<0,m_1-l_1=<0,m_2-l_2=>0'}
    \end{enumerate}
    The right figure in Figure \ref{fig_hexagonal_norm_case} illustrates the
    location of $l\in\mathbb{Z}^2$ which corresponds to the above nine cases.

    Let us consider, for example, the case
    \eqref{l_1-n_1=<0,l_2-n_2=>0,m_1-l_1=>0,m_2-l_2=>0'}$^{\prime}$. Then we
    have
    \begin{align*}
        d(m-n)&=\max{\{\vert{m_1-n_1}\vert,\vert{m_2-n_2}\vert\}}\\
        &=\max{\{(n_1-l_1)+(l_1-m_1),\vert{m_2-n_2}\vert\}}.
    \end{align*}
    By the assumption, we have $l_1-m_1\le{0}$, which implies
    \begin{align*}
        d(m-n)&\le\max{\{\vert{l_1-n_1}\vert,\vert{m_2-l_2}\vert+\vert{l_2-n_2}\vert\}}\\
        &\le\vert{m_2-l_2}\vert+\max{\{\vert{l_1-n_1}\vert,\vert{l_2-n_2}\vert\}}\\
        &\le\vert{m_1-l_2}\vert+\vert{m_2-l_2}\vert+\max{\{\vert{l_1-n_1}\vert,\vert{l_2-n_2}\vert\}}\\
        &=d(m-l)+d(l-n).
    \end{align*}

    We can treat the other cases in the same way.

    The case that $m_1-n_1\le{0}$ and $m_2-n_2\le{0}$ can be reduced to the
    first one; the remaining one to the second.
\end{proof}

\begin{rem}
    One can convince oneself of the above lemma if one bears the shape of the
    contour lines of $d_{ij}(\cdot)$ in mind and shifts them. One can also have
    insight into Lemmas \ref{lemma_d_ij=<d_ii+d_ij} and
    \ref{lemma_d_ii=<d_ij+d_ji+1} in the same way.
\end{rem}

We can prove the following two lemmas in the same way as above, whose proofs
will be given in \ref{appendix_quasi_triangles_inequality}.
\begin{lemma}
    \label{lemma_d_ij=<d_ii+d_ij}
    For $i$, $j\in\{1,2\}$,
    \begin{equation}
        \label{d_ij=<d_ii+d_ij}
        d_{ij}(m-n)\le d_{ii}(m-l)+d_{ij}(l-n),
    \end{equation}
    \begin{equation}
        \label{d_ij=<d_ij+d_jj}
        d_{ij}(m-n)\le d_{ij}(m-l)+d_{jj}(l-n).
    \end{equation}
\end{lemma}
\begin{lemma}
    \label{lemma_d_ii=<d_ij+d_ji+1}
    For $i$, $j\in\{1,2\}$,
    \begin{equation*}
        d_{ii}(m-n)\le d_{ij}(m-l)+d_{ji}(l-n)+1.
    \end{equation*}
\end{lemma}

We put $r(\xi)=\vert{\alpha(\xi})\vert^2$, where $\alpha(\xi)$ is defined in
\eqref{alpha}. Then we have
\begin{lemma}
    \label{Fourier_coefficients_r}
    Let $s\in\mathbb{Z}$, $s\ge{0}$. Then the Fourier coefficients of
    $r(\xi)^s$ are supported on the set $\{n\in\mathbb{Z}^2;d(n)\le{s}\}$,
    that is,
    \begin{equation*}
        \int_{\mathbb{T}^2}r(\xi)^se^{-in\xi}d\xi=0,
    \end{equation*}
    when $d(n)>s$.
\end{lemma}
\begin{proof}
    The claim in the lemma is obvious if $s=0$.
    
    By the definition of $r(\xi)$, we have
    \begin{equation}
        \label{r(xi)}
        r(\xi)=3+e^{i\xi_1}+e^{-i\xi_1}+e^{i\xi_2}+e^{-i\xi_2}+e^{i(\xi_1-\xi_2)}+e^{-i(\xi_1-\xi_2)},
    \end{equation}
    which means that $r(\xi)$ is a linear combination of $e^{in_0\xi}$,
    $d(n_0)\le1$.

    Assume that we have proved the claim for $s\le{}s_0$. Then $r(\xi)^{s_0}$
    is a linear combination of $e^{in\xi}$, $d(n)\le{}s_0$. Therefore,
    $r(\xi)^{s_0+1}$ is a linear combination of $e^{in\xi}r(\xi)$,
    $d(n)\le{}s_0$, which means, from \eqref{r(xi)}, that it is that of
    $e^{im\xi}$, $m=n+n_0$, where $d(n)\le{}s_0$ and $d(n_0)\le1$. By using
    the triangle inequality \eqref{triangle_inequality_d(n)}, we have
    \begin{align*}
        d(m)&=d(n+n_0)\\
        &\le{}d(n)+d(n_0)\\
        &\le{}s_0+1,
    \end{align*}
    which proves the lemma by the induction argument with respect to
    $s\in\mathbb{Z}$, $s\ge0$.
\end{proof}

In a similar way, we have
\begin{lemma}
    \label{Fourier_coefficients_r_alpha}
    Let $s\in\mathbb{Z}$, $s\ge{0}$. Then the Fourier coefficients of
    $r(\xi)^s\alpha(\xi)$ are supported on the set
    $\{n\in\mathbb{Z}^2;d_{12}(n)\le{s}\}$, that is,
    \begin{equation*}
        \int_{\mathbb{T}^2}r(\xi)^s\alpha(\xi)e^{-in\xi}d\xi=0,
    \end{equation*}
    when $d_{12}(n)>s$.
\end{lemma}
\begin{proof}
    We have already proved that $r(\xi)^s$ is a linear combination of
    $e^{in\xi}$, $d(n)\le{}s$. From \eqref{alpha}, $\alpha(\xi)$ is a sum of
    $e^{in_0\xi}$, where $n_0\in\{(0,0),(1,0),(0,1)\}$. Therefore,
    $r(\xi)^s\alpha(\xi)$ is a linear combination of $e^{in\xi}\alpha(\xi)$,
    $d(n)\le{}s$, which means that it is that of $e^{im\xi}$, $m=n+n_0$, where
    $d(n)\le{}s$ and $n_0\in\{(0,0),(1,0),(0,1)\}$. By the definition
    \eqref{definition_d_12(n)}, we have $d_{12}(m)\le{}s$ (see also the shape
    of the contour line of $d_{12}(\cdot)$ in Figure
    \ref{fig_contour_line_d_ij(n)=s}), which proves the lemma.
\end{proof}

Noticing that
\begin{equation*}
    \overline{\int_{\mathbb{T}^2}r(\xi)^s\alpha(\xi)e^{-in\xi}d\xi}=\int_{\mathbb{T}^2}r(\xi)^s\overline{\alpha}(\xi)e^{in\xi}d\xi
\end{equation*}
and that $d_{12}(n)=d_{21}(-n)$, we have
\begin{lemma}
    \label{Fourier_coefficients_r_overline_alpha}
    Let $s\in\mathbb{Z}$, $s\ge{0}$. Then the Fourier coefficients of
    $r(\xi)^s\overline{\alpha}(\xi)$ are supported on the set
    $\{n\in\mathbb{Z}^2;d_{21}(n)\le{s}\}$, that is,
    \begin{equation*}
        \int_{\mathbb{T}^2}r(\xi)^s\overline{\alpha}(\xi)e^{-in\xi}d\xi=0,
    \end{equation*}
    when $d_{21}(n)>s$.
\end{lemma}

Now, let us investigate the asymptotic behavior of the resolvents as
$\vert{z}\vert\rightarrow\infty$. We begin with $\hat{R}_0(z)$. We will also
use $\hat{P}(n)$ as the projection onto the site $n\in\mathbb{Z}^2$ on
$l^2(\mathbb{Z}^2)$.
\begin{lemma}
    \label{estimate_R_0}
    If $\hat{f}=(\hat{f}_1,\hat{f}_2)\in l^2(\mathbb{Z}^2;\mathbb{C}^2)$ is
    compactly supported, then we have
    \begin{equation*}
        \hat{R}_0(z)\hat{f}=
        \begin{pmatrix}
            \displaystyle\sum_{s=0}^{\infty}z^{-2s-1}\hat{R}_{0,s,11}\hat{f}_1+\sum_{s=0}^{\infty}z^{-2s-2}\hat{R}_{0,s,12}\hat{f}_2\\
            \displaystyle\sum_{s=0}^{\infty}z^{-2s-2}\hat{R}_{0,s,21}\hat{f}_1+\sum_{s=0}^{\infty}z^{-2s-1}\hat{R}_{0,s,22}\hat{f}_2
        \end{pmatrix}
    \end{equation*}
    for sufficiently large $\vert{z}\vert$, where
    $\hat{R}_{0,s,ij}\in{\bold{B}}(l^2(\mathbb{Z}^2))$ with the properties
    \begin{equation*}
        \hat{P}(m)\hat{R}_{0,s,ij}(z)\hat{P}(n)=0,
    \end{equation*}
    if $d_{ij}(m-n)>s$, $i, j\in\{1,2\}$.
\end{lemma}
\begin{proof}
    From \eqref{hat_r_0} and \eqref{hat_R_0_hat_f}, we have
    \begin{equation*}
        (\hat{R}_{0,s,ij}\hat{f}_j)(n)=\sum_{m\in\mathbb{Z}^2}\hat{r}_{0,s,ij}(n-m)\hat{f}_j(m),
    \end{equation*}
    where
    \begin{equation*}
        \hat{r}_{0,s,11}(k)=\hat{r}_{0,s,22}(k)=\int_{\mathbb{T}^2}r(\xi)^se^{-ik\xi}d\xi,
    \end{equation*}
    \begin{equation*}
        \hat{r}_{0,s,12}(k)=\int_{\mathbb{T}^2}r(\xi)^s\alpha(\xi)e^{-ik\xi}d\xi,
    \end{equation*}
    \begin{equation*}
        \hat{r}_{0,s,21}(k)=\int_{\mathbb{T}^2}r(\xi)^s\overline{\alpha}(\xi)e^{-ik\xi}d\xi.
    \end{equation*}
    Lemmas \ref{Fourier_coefficients_r}, \ref{Fourier_coefficients_r_alpha},
    and \ref{Fourier_coefficients_r_overline_alpha} imply the above properties
    of $\hat{R}_{0,s,ij}$, $i,j\in\{1,2\}$, which proves the lemma.
\end{proof}

Next, we consider the asymptotic behavior of $\hat{R}(z)$ for large
$\vert{z}\vert$.
\begin{lemma}
    \label{lemma_resolvent_estimate_for_R(z)}
    For sufficiently large $\vert{z}\vert$,
    \begin{equation}
        \label{resolvent_estimate_for_R(z)}
        \hat{P}(m)\hat{R}(z)\hat{P}(n)=
        \begin{pmatrix}
            {\cal{O}}(\langle{z}\rangle^{-2d(m-n)-1})&{\cal{O}}(\langle{z}\rangle^{-2d_{12}(m-n)-2})\\
            {\cal{O}}(\langle{z}\rangle^{-2d_{21}(m-n)-2})&{\cal{O}}(\langle{z}\rangle^{-2d(m-n)-1})
        \end{pmatrix}.
    \end{equation}
\end{lemma}
\begin{proof}
    Let $p=d(m-n)$. Then, by using the resolvent equation repeatedly, we have
    the following expansion:
    \begin{equation}
        \label{expansion_R}
        \hat{R}(z)=\hat{R}_0(z)-\hat{R}_0(z)\hat{q}\hat{R}_0(z)+\cdots+(-1)^{2p-1}\underbrace{\hat{R}_0(z)\hat{q}\hat{R}_0(z)\cdots\hat{q}\hat{R}_0(z)}_{\text{$2p$
            $\hat{R}_0(z)$'s and $2p-1$ $\hat{q}$'s}}+{\cal{O}}(\langle{z}\rangle^{-2p-1}).
    \end{equation}

    By multiplying $\hat{P}(m)$ and $\hat{P}(n)$ from the left and the right
    of \eqref{expansion_R}, respectively, we have
    \begin{equation*}
        \hat{P}(m)\hat{R}_0(z)\hat{P}(n)=
        \begin{pmatrix}
            \displaystyle\sum_{s=0}^{\infty}z^{-2z-1}\hat{P}(m)\hat{R}_{0,s,11}\hat{P}(n)&\displaystyle\sum_{s=0}^{\infty}z^{-2z-2}\hat{P}(m)\hat{R}_{0,s,12}\hat{P}(n)\\
            \displaystyle\sum_{s=0}^{\infty}z^{-2z-2}\hat{P}(m)\hat{R}_{0,s,21}\hat{P}(n)&\displaystyle\sum_{s=0}^{\infty}z^{-2z-1}\hat{P}(m)\hat{R}_{0,s,22}\hat{P}(n)
        \end{pmatrix}
    \end{equation*}
    as the fist term of the right hand side. Then, from Lemma
    \ref{estimate_R_0}, we have
    \begin{align*}
        \sum_{s=0}^{\infty}z^{-2z-1}\hat{P}(m)\hat{R}_{0,s,ii}\hat{P}(n)&=\sum_{s\ge{}d(m-n)}z^{-2z-1}\hat{P}(m)\hat{R}_{0,s,ii}\hat{P}(n)\\
        &={\cal{O}}(\langle{z}\rangle^{-2d(m-n)-1}),
    \end{align*}
    \begin{align*}
        \sum_{s=0}^{\infty}z^{-2z-2}\hat{P}(m)\hat{R}_{0,s,ij}\hat{P}(n)&=\sum_{s\ge{}d_{ij}(m-n)}z^{-2z-2}\hat{P}(m)\hat{R}_{0,s,ij}\hat{P}(n)\\
        &={\cal{O}}(\langle{z}\rangle^{-2d_{ij}(m-n)-2}),
    \end{align*}
    for $(i,j)=(1,2)$ or $(2,1)$.

    About the $k$-th term, $1\le{k}\le{p}$, we use inductions. Such a term is
    written as a sum of
    \begin{equation}
        \label{k-th_term}
        \hat{P}(m)\hat{R}_0(z)\hat{P}(l^{(1)})\hat{q}(l^{(1)})\hat{P}(l^{(1)})\hat{R}_0(z)\cdots\hat{P}(l^{(k)})\hat{q}(l^{(k)})\hat{P}(l^{(k)})\hat{R}_0(z)\hat{P}(n),
    \end{equation}
    since
    $\hat{q}=\sum_{\sharp\{l\in\mathbb{Z}^2\}<\infty}{\hat{P}(l)\hat{q}(l)\hat{P}(l)}$.

    Assume that \eqref{k-th_term} is estimated as
    \eqref{resolvent_estimate_for_R(z)} for $k=k_0-1$. Then, by the hypothesis
    of the induction, we have
    \begin{align*}
        &\hat{P}(m)\hat{R}_0(z)\hat{P}(l^{(1)})\hat{q}(l^{(1)})\hat{P}(l^{(1)})\hat{R}_0(z)\cdots\hat{P}(l^{(k_0-1)})\hat{q}(l^{(k_0-1)})\hat{P}(l^{(k_0-1)})\hat{R}_0(z)\hat{P}(l^{(k_0)})\\
        &\quad\cdot\hat{q}(l^{(k_0)})\cdot\hat{P}(l^{(k_0)})\hat{R}_0(z)\hat{P}(n)\\
        &=
        \begin{pmatrix}
            {\cal{O}}(\langle{z}\rangle^{-2d(m-l^{(k_0)})-1})&{\cal{O}}(\langle{z}\rangle^{-2d_{12}(m-l^{(k_0)})-2})\\
            {\cal{O}}(\langle{z}\rangle^{-2d_{21}(m-l^{(k_0)})-2})&{\cal{O}}(\langle{z}\rangle^{-2d(m-l^{(k_0)})-1})
        \end{pmatrix}\\
        &\quad\quad\cdot
        \begin{pmatrix}
            \hat{q}_1(l^{(k_0)})&0\\
            0&\hat{q}_2(l^{(k_0)})
        \end{pmatrix}
        \cdot
        \begin{pmatrix}
            {\cal{O}}(\langle{z}\rangle^{-2d(l^{(k_0)}-n)-1})&{\cal{O}}(\langle{z}\rangle^{-2d_{12}(l^{(k_0)}-n)-2})\\
            {\cal{O}}(\langle{z}\rangle^{-2d_{21}(l^{(k_0)}-n)-2})&{\cal{O}}(\langle{z}\rangle^{-2d(l^{(k_0)}-n)-1})
        \end{pmatrix},
    \end{align*}
    which reads, from Lemmas \ref{lemma_d_triangle_ineq},
    \ref{lemma_d_ij=<d_ii+d_ij}, and \ref{lemma_d_ii=<d_ij+d_ji+1}, that the
    upper left element is
    \begin{align*}
        &{\cal{O}}(\langle{z}\rangle^{-2d(m-l^{(k_0)})-2d(l^{(k_0)}-n)-2})+{\cal{O}}(\langle{z}\rangle^{-2d_{12}(m-l^{(k_0)})-2d_{21}(l^{(k_0)}-n)-4})\\
        &={\cal{O}}(\langle{z}\rangle^{-2d(m-n)-2}),
    \end{align*}
    and that the upper right one is
    \begin{align*}
        &{\cal{O}}(\langle{z}\rangle^{-2d(m-l^{(k_0)})-2d_{12}(l^{(k_0)}-n)-3})+{\cal{O}}(\langle{z}\rangle^{-2d_{12}(m-l^{(k_0)})-2d(l^{(k_0)}-n)-3})\\
        &={\cal{O}}(\langle{z}\rangle^{-2d_{12}(m-n)-3}).
    \end{align*}

    We also have the lower left and right elements as desired, which proves
    the lemma.
\end{proof}

\begin{rem}
    A more detailed argument in the induction above gives us the estimate
    \eqref{resolvent_estimate_for_R(z)} for the non-diagonal $\hat{q}(n)$'s.
\end{rem}


%
%
\section{Proof of the main theorem}

We will essentially follow the idea of Isozaki and Korotyaev
\cite{2011arXiv1103.2193I}. Assume that $\lambda>0$; for $\lambda<0$, we can
argue in the same way. Put
\begin{equation*}
    B(k,\theta,\theta^{\prime})=4(2\pi)^2\bigl(J(k,\theta)J(k,\theta^{\prime})\bigr)^{-1}A(k,\theta,\theta^{\prime}),
\end{equation*}
where $A(k,\theta,\theta^{\prime})$ is the integral kernel of the scattering
amplitude defined in Section \ref{section_scattering_matrix}. Then we have
$B(k,\theta,\theta^{\prime})=B_0(k,\theta,\theta^{\prime})-B_1(k,\theta,\theta^{\prime})$, where
\begin{align*}
    B_0(k,\theta,\theta^{\prime})&=4(2\pi)^2\bigl(J(k,\theta)J(k,\theta^{\prime})\bigr)^{-1}\hat{{\cal{F}}}_0(\lambda)\hat{q}\hat{{\cal{F}}}_0(\lambda)^*\\
    &=
    \begin{pmatrix}
        B_{0,11}(k,\theta,\theta^{\prime})&B_{0,12}(k,\theta,\theta^{\prime})\\
        B_{0,21}(k,\theta,\theta^{\prime})&B_{0,22}(k,\theta,\theta^{\prime})
    \end{pmatrix},
\end{align*}
\begin{equation*}
    B_{0,11}(k,\theta,\theta^{\prime})=\sum_{n\in\mathbb{Z}^2}e^{in(\xi(k,\theta)-\xi(k,\theta^{\prime}))}\bigl(\hat{q}_1(n)+\frac{\alpha(\xi(k,\theta))}{\sqrt{8k^2+1}}\frac{\overline{\alpha}(\xi(k,\theta^{\prime}))}{\sqrt{8k^2+1}}\hat{q}_2(n)\bigr),
\end{equation*}
\begin{equation*}
    B_{0,22}(k,\theta,\theta^{\prime})=\sum_{n\in\mathbb{Z}^2}e^{in(\xi(k,\theta)-\xi(k,\theta^{\prime}))}\bigl(\frac{\overline{\alpha}(\xi(k,\theta))}{\sqrt{8k^2+1}}\frac{\alpha(\xi(k,\theta^{\prime}))}{\sqrt{8k^2+1}}\hat{q}_1(n)+\hat{q}_2(n)\bigr)
\end{equation*}
($B_{0,12}$ and $B_{0,21}$ are omitted), and
\begin{align*}
    B_1(k,\theta,\theta^{\prime})&=4(2\pi)^2\bigl(J(k,\theta)J(k,\theta^{\prime})\bigr)^{-1}\hat{{\cal{F}}}_0(\lambda)\hat{q}\hat{R}(\lambda+i0)\hat{q}\hat{{\cal{F}}}_0(\lambda)^*\\
    &=
    \begin{pmatrix}
        B_{1,11}(k,\theta,\theta^{\prime})&B_{1,12}(k,\theta,\theta^{\prime})\\
        B_{1,21}(k,\theta,\theta^{\prime})&B_{1,22}(k,\theta,\theta^{\prime})
    \end{pmatrix},
\end{align*}
\begin{align*}
    B_{1,11}(k,\theta,\theta^{\prime})&=\sum_{n\in\mathbb{Z}^2}e^{in\xi(k,\theta)}\hat{q}_1(n)\bigl(\hat{R}(\lambda+i0)\hat{q}\hat{\varphi}^{(0)}(k,\theta^{\prime})\bigr)_{11}(n)\\
    &\quad+\frac{\alpha(\xi(k,\theta))}{\sqrt{8k^2+1}}\sum_{n\in\mathbb{Z}^2}e^{in\xi(k,\theta)}\hat{q}_2(n)\bigl(\hat{R}(\lambda+i0)\hat{q}\hat{\varphi}^{(0)}(k,\theta^{\prime})\bigr)_{21}(n),
\end{align*}
\begin{align*}
    B_{1,22}(k,\theta,\theta^{\prime})&=\frac{\overline{\alpha}(\xi(k,\theta))}{\sqrt{8k^2+1}}\sum_{n\in\mathbb{Z}^2}e^{in\xi(k,\theta)}\hat{q}_1(n)\bigl(\hat{R}(\lambda+i0)\hat{q}\hat{\varphi}^{(0)}(k,\theta^{\prime})\bigr)_{12}(n)\\
    &\quad+\sum_{n\in\mathbb{Z}^2}e^{in\xi(k,\theta)}\hat{q}_2(n)\bigl(\hat{R}(\lambda+i0)\hat{q}\hat{\varphi}^{(0)}(k,\theta^{\prime})\bigr)_{22}(n)
\end{align*}
($B_{1,12}$ and $B_{1,21}$ are omitted). Here we put
\begin{align*}
    \hat{\varphi}^{(0)}(k,\theta^{\prime})&=2(2\pi)J(k,\theta^{\prime})^{-1}\hat{\psi}^{(0)}(k,\theta^{\prime})\\
    &=(\hat{\varphi}^{(0)}(n;k,\theta^{\prime}))_{n\in\mathbb{Z}^2},
\end{align*}
where
\begin{align*}
    \hat{\varphi}^{(0)}(n;k,\theta^{\prime})&=2(2\pi)J(k,\theta^{\prime})^{-1}\hat{\psi}^{(0)}(n;k,\theta^{\prime})\\
    &=e^{-in\xi(k,\theta^{\prime})}\hat{\eta}^{(0)}(k,\theta^{\prime}),\
    n\in\mathbb{Z}^2,
\end{align*}
\begin{equation}
    \label{asymptotic_of_eta}
    \hat{\eta}^{(0)}(k,\theta^{\prime})=
    \begin{pmatrix}
        1&\displaystyle\frac{\alpha(\xi(k,\theta^{\prime}))}{\sqrt{8k^2+1}}\\
        \displaystyle\frac{\overline{\alpha}(\xi(k,\theta^{\prime}))}{\sqrt{8k^2+1}}&1
    \end{pmatrix},
\end{equation}
and $\hat{\psi}(k,\theta^{\prime})$ and $\hat{\psi}(n;k,\theta^{\prime})$ are
defined by \eqref{hat_psi} and \eqref{hat_psi^(0)}, respectively

Let $\zeta_{\pm}(z,\theta)=(\zeta_{\pm,1}(z,\theta),\zeta_{\pm,2}(z,\theta))$
be the analytic continuation of
$\xi(k,\theta)=(\xi_1(k,\theta),\xi_2(k,\theta))$ in Lemma
\ref{lemma_analytic_continuation}. Then $B_0(k,\theta,\theta^{\prime})$ and
$B_1(k,\theta,\theta^{\prime})$ have the analytic continuations
$B_0(z,\theta,\theta^{\prime})$ and $B_1(z,\theta,\theta^{\prime})$ for
$z\in\mathbb{C}_+$, respectively, which are defined with $k$ replaced by $z$,
$\xi(k,\theta)$ by $\zeta_+(z,\theta)$ and $\xi(k,\theta^{\prime})$ by
$\zeta_-(z,\theta^{\prime})$.

Let us take $-\pi<\xi_j(k,\theta)<0$ and $0<\xi_j(k,\theta^{\prime})<\pi$ for
$j\in\{1,2\}$. From Lemma \ref{lemma_analytic_continuation}, we have
\begin{equation}
    \label{asymptotic_e^{zeta_+}}
    e^{in\zeta_+(z,\theta)}\sim N^{2n_2}b_1(\theta)^{n_1}a_1(\theta)^{n_2},
\end{equation}
\begin{equation}
    \label{asymptotic_e^{zeta_-}}
    e^{-in\zeta_-(z,\theta^{\prime})}\sim{N^{2n_2}b_1(\theta^{\prime})^{n_1}a_1(\theta^{\prime})^{n_2}}
\end{equation}
as $N\rightarrow\infty$, where $z=1+iN$ and
\begin{equation}
    \label{a_1}
    a_1(\theta)=(2a(\theta))^2=4(2-e^{2\theta}),
\end{equation}
\begin{equation}
    \label{b_1}
    b_1(\theta)=(b(\theta)+\sqrt{b(\theta)^2+1})^2=\frac{e^{2\theta}}{2-e^{2\theta}}.
\end{equation}
We also have
\begin{equation}
    \label{asymptotic_of_alpha_zeta_+/sqrt_8z^2+1}
    \frac{\alpha(\zeta_+(z,\theta))}{\sqrt{8z^2+1}}=\frac{1+e^{i\zeta_{+,1}(z,\theta)}+e^{i\zeta_{+,2}(z,\theta)}}{\sqrt{8z^2+1}}\sim{}a_2(\theta)N+{\cal{O}}(1),
\end{equation}
\begin{equation}
    \label{asymptotic_of_adjoint_alpha_zeta_+/sqrt_8z^2+1}
    \frac{\overline{\alpha}(\zeta_+(z,\theta))}{\sqrt{8z^2+1}}=\frac{1+e^{-i\zeta_{+,1}(z,\theta)}+e^{-i\zeta_{+,2}(z,\theta)}}{\sqrt{8z^2+1}}\sim{}b_2(\theta)N^{-1}+{\cal{O}}(N^{-2}),
\end{equation}
\begin{equation}
    \label{asymptotic_of_alpha_zeta_-/sqrt_8z^2+1}
    \frac{\alpha(\zeta_-(z,\theta^{\prime}))}{\sqrt{8z^2+1}}=\frac{1+e^{i\zeta_{-,1}(z,\theta)}+e^{i\zeta_{-,2}(z,\theta)}}{\sqrt{8z^2+1}}\sim{}b_2(\theta^{\prime})N^{-1}+{\cal{O}}(N^{-2}),
\end{equation}
\begin{equation}
    \label{asymptotic_of_adjoint_alpha_zeta_-/sqrt_8z^2+1}
    \frac{\overline{\alpha}(\zeta_-(z,\theta^{\prime}))}{\sqrt{8z^2+1}}=\frac{1+e^{-i\zeta_{-,1}(z,\theta)}+e^{-i\zeta_{-,2}(z,\theta)}}{\sqrt{8z^2+1}}\sim{}a_2(\theta^{\prime})N+{\cal{O}}(1),
\end{equation}
where
\begin{equation*}
    a_2(\theta)=-\frac{a_1(\theta)}{2\sqrt{2}}i=-\sqrt{2}(2-e^{2\theta})i,
\end{equation*}
\begin{equation*}
    b_2(\theta)=-\frac{1+b_1(\theta)^{-1}}{2\sqrt{2}}i=-\frac{e^{-2\theta}}{\sqrt{2}}i.
\end{equation*}

By our assumption \ref{assumption_potential}, the potential $\hat{q}$ is
compactly supported; more precisely, for some $M\ge0$, $M\in\mathbb{Z}$,
\begin{equation*}
    \vert{n}\vert_{l^1}=\vert{n_1}\vert+\vert{n_2}\vert>M\Longrightarrow\hat{q}(n)=0.
\end{equation*}
We put
\begin{equation*}
    n^{(M)}=(0,M)\in\mathbb{Z}^2.
\end{equation*}
Then, by using \eqref{asymptotic_e^{zeta_+}}, \eqref{asymptotic_e^{zeta_-}},
\eqref{asymptotic_of_adjoint_alpha_zeta_+/sqrt_8z^2+1}, and
\eqref{asymptotic_of_alpha_zeta_-/sqrt_8z^2+1}, we have the following
asymptotic expansion:
\begin{equation}
    \label{B_0,22-1}
    B_{0,22}(z,\theta,\theta^{\prime})\sim{N^{4M}}(a_1(\theta)a_1(\theta^{\prime}))^{4M}\hat{q}_2(n^{(M)}).
\end{equation}

Let us investigate the asymptotic behavior of
$B_{1,22}(z,\theta,\theta^{\prime})$. By introducing Dirac's notation, we have
\begin{equation*}
    \hat{P}(m)=\vert{\hat{p}_m}\rangle\langle\hat{p}_m\vert,
\end{equation*}
where
\begin{equation*}
    (\hat{p}_m)(n)=
    \begin{pmatrix}
        \delta_{mn}\\
        \delta_{mn}
    \end{pmatrix}
    \in{}l^2(\mathbb{Z}^2;\mathbb{C}^2).
\end{equation*}
From \eqref{asymptotic_of_alpha_zeta_-/sqrt_8z^2+1} and
\eqref{asymptotic_of_adjoint_alpha_zeta_-/sqrt_8z^2+1}, we note that
\begin{equation}
    \label{asymptotic_of_eta}
        \hat{\eta}^{(0)}(z,\theta^{\prime})=
        \begin{pmatrix}
            1&\displaystyle\frac{\alpha(\zeta_-(z,\theta^{\prime}))}{\sqrt{8z^2+1}}\\
            \displaystyle\frac{\overline{\alpha}(\zeta_-(z,\theta^{\prime}))}{\sqrt{8z^2+1}}&1
        \end{pmatrix}
        \sim
        \begin{pmatrix}
            1&{\cal{O}}(N^{-1})\\
            {\cal{O}}(N)&1
        \end{pmatrix}.
\end{equation}
From Lemma \ref{lemma_resolvent_estimate_for_R(z)}, by using
\eqref{asymptotic_of_eta}, we have
\begin{equation}
    \label{estimate_for_<P(n)|R(z)|P(m)>q(m)_eta}
    \langle\hat{P}(n)\vert\hat{R}(z)\vert\hat{P}(m)\rangle\hat{q}(m)\hat{\eta}^{(0)}(z,\theta^{\prime})=
    \begin{pmatrix}
        {\cal{O}}_{11}&{\cal{O}}_{12}\\
        {\cal{O}}_{21}&{\cal{O}}_{22}
    \end{pmatrix},
\end{equation}
where
\begin{equation*}
    {\cal{O}}_{11}={\cal{O}}(\langle{z}\rangle^{-2d(n-m)-1})\hat{q}_1(m)+{\cal{O}}(\langle{z}\rangle^{-2d_{12}(n-m)-2})\hat{q}_2(m)\displaystyle\frac{\overline{\alpha}(\zeta_-(z,\theta^{\prime}))}{\sqrt{8z^{2}+1}},
\end{equation*}
\begin{equation*}
    {\cal{O}}_{12}={\cal{O}}(\langle{z}\rangle^{-2d(n-m)-1})\hat{q}_1(m)\displaystyle\frac{\alpha(\zeta_-(z,\theta^{\prime}))}{\sqrt{8z^2+1}}+{\cal{O}}(\langle{z}\rangle^{-2d_{12}(n-m)-2})\hat{q}_2(m),
\end{equation*}
\begin{equation*}
    {\cal{O}}_{21}={\cal{O}}(\langle{z}\rangle^{-2d_{21}(n-m)-2})\hat{q}_1(m)+{\cal{O}}(\langle{z}\rangle^{-2d(n-m)-1})\hat{q}_2(m)\displaystyle\frac{\overline{\alpha}(\zeta_-(z,\theta^{\prime}))}{\sqrt{8z^2+1}},
\end{equation*}
\begin{equation*}
    {\cal{O}}_{22}={\cal{O}}(\langle{z}\rangle^{-2d_{21}(n-m)-2})\hat{q}_1(m)\displaystyle\frac{\alpha(\zeta_-(z,\theta^{\prime}))}{\sqrt{8z^2+1}}+{\cal{O}}(\langle{z}\rangle^{-2d(n-m)-1})\hat{q}_2(m),
\end{equation*}
which reads, by using \eqref{asymptotic_of_alpha_zeta_-/sqrt_8z^2+1} and
\eqref{asymptotic_of_adjoint_alpha_zeta_-/sqrt_8z^2+1}, that
\begin{equation}
    \label{asymptotic_of_<P(n)|R(z)|P(m)>q(m)eta(z,theta)}
    \langle\hat{P}(n)\vert\hat{R}(z)\vert\hat{P}(m)\rangle\hat{q}(m)\hat{\eta}^{(0)}(z,\theta^{\prime})\sim
    \begin{pmatrix}
        {\cal{O}}(N^{-1})&{\cal{O}}(N^{-2})\\
        {\cal{O}}(1)&{\cal{O}}(N^{-1})
    \end{pmatrix},
\end{equation}
since $d_{ij}(n)\ge{0}$. Then we have
\begin{align*}
    B_{1,22}(z,\theta,\theta^{\prime})=&\frac{\overline{\alpha}(\zeta_+(z,\theta))}{\sqrt{8z^2+1}}\sum_{n_2\le{M}}e^{in\zeta_+(z,\theta)}\hat{q}_1(n)\sum_{m_2\le{M}}e^{-im\zeta_-(z,\theta^{\prime})}\\
    &\quad\cdot\bigl(\langle\hat{P}(n)\vert\hat{R}(z)\vert\hat{P}(m)\rangle\hat{q}(m)\hat{\eta}^{(0)}(z,\theta^{\prime})\bigr)_{12}\\
    &+\sum_{n_2\le{M}}e^{in\zeta_+(z,\theta)}\hat{q}_2(n)\sum_{m_2\le{M}}e^{-im\zeta_-(z,\theta^{\prime})}\\
    &\quad\quad\cdot\bigl(\langle\hat{P}(n)\vert\hat{R}(z)\vert\hat{P}(m)\rangle\hat{q}(m)\hat{\eta}^{(0)}(z,\theta^{\prime})\bigr)_{22}\\
    \sim&{\cal{O}}(N^{4M-1}).
\end{align*}

Therefore, noticing that $a_1(\cdot)$ in \eqref{B_0,22-1} is defined as
\eqref{a_1}, we can compute $\hat{q}_2(n^{(M)})$ from the asymptotic expansion
of $B_{22}(z,\theta,\theta^{\prime})$. 

We turn to $B_{11}(z,\theta,\theta^{\prime})$. By using
\eqref{asymptotic_e^{zeta_+}}, \eqref{asymptotic_e^{zeta_-}},
\eqref{asymptotic_of_alpha_zeta_+/sqrt_8z^2+1},
\eqref{asymptotic_of_adjoint_alpha_zeta_-/sqrt_8z^2+1}, and
\eqref{asymptotic_of_<P(n)|R(z)|P(m)>q(m)eta(z,theta)}, we have the following
asymptotic expansion:
\begin{align}
    \label{asymptotic_B_0,11}
    \begin{split}
        B_{0,11}(z,\theta,\theta^{\prime})&\sim{N^{4M}}(a_1(\theta)a_1(\theta^{\prime}))^{M}\hat{q}_1(n^{(M)})\\
        &\quad+e^{in^{(M)}(\zeta_+(z,\theta)-\zeta_-(z,\theta^{\prime}))}\frac{\alpha(\zeta_+(z,\theta))}{\sqrt{8z^2+1}}\frac{\overline{\alpha}(\zeta_-(z,\theta^{\prime}))}{\sqrt{8z^2+1}}\hat{q}_2(n^{(M)})\\
        &\quad+{\cal{O}}(N^{4M-2}).
    \end{split}
\end{align}
Note that the asymptotic behavior of the second term of the right hand side is
known, since $\hat{q}_2(n^{(M)})$ has already been computed.

Let us investigate the asymptotic behavior of
$B_{1,11}(z,\theta,\theta^{\prime})$. We put
\begin{equation*}
    \hat{q}_{(\le r,\le s)}=\sum_{n_2\le r}\hat{P}(n)
    \begin{pmatrix}
        \hat{q}_1(n)&0\\
        0&0
    \end{pmatrix}
    \hat{P}(n)+\sum_{n_2\le s}\hat{P}(n)
    \begin{pmatrix}
        0&0\\
        0&\hat{q}_2(n)
    \end{pmatrix}
    \hat{P}(n),
\end{equation*}
\begin{equation*}
    \hat{q}_{(>r,>s)}=\hat{q}-\hat{q}_{(\le r,\le s)},
\end{equation*}
\begin{equation*}
    \hat{H}_{(>r,>s)}=\hat{H}_0+\hat{q}_{(>r,>s)},
\end{equation*}
\begin{equation*}
    \hat{R}_{(>r,>s)}(z)=(\hat{H}_{(>r,>s)}-z)^{-1}.
\end{equation*}
If $r=s$, we write $\hat{q}_{\le r}=\hat{q}_{(\le r,\le r)}$,
$\hat{q}_{>r}=\hat{q}_{(>r,>r)}$, $\hat{H}_{>r}=\hat{H}_{(>r,>r)}$, and
$\hat{R}_{>r}(z)=\hat{R}_{(>r,>r)}(z)$ for brevity.

From \eqref{asymptotic_of_<P(n)|R(z)|P(m)>q(m)eta(z,theta)}, we have
\begin{equation*}
    \bigl(\langle\hat{P}(n)\vert\hat{R}(z)\vert\hat{P}(m)\rangle\hat{q}(m)\hat{\eta}^{(0)}(z,\theta^{\prime})\bigr)_{11}\sim{\cal{O}}(N^{-1}),
\end{equation*}
which implies, by using \eqref{asymptotic_e^{zeta_+}} and
\eqref{asymptotic_e^{zeta_-}}, the following asymptotic expansion of the fist
term of $B_{1,11}(z,\theta,\theta^{\prime})$:
\begin{align*}
    &\sum_{n_2\le{M}}e^{in\zeta_{+}(z,\theta)}\hat{q}_1(n)\sum_{m_2\le{M}}e^{-im\zeta_-(z,\theta^{\prime})}\bigl(\langle\hat{P}(n)\vert\hat{R}(z)\vert\hat{P}(m)\rangle\hat{q}(m)\hat{\eta}^{(0)}(z,\theta^{\prime})\bigr)_{11}\\
    &\sim{\cal{O}}(N^{4M-1}).
\end{align*}

We put the second term of $B_{1,11}(z,\theta,\theta^{\prime})$ as
\begin{equation*}
    I=\frac{\alpha(\zeta_+(z,\theta))}{\sqrt{8z^2+1}}\sum_{n\in\mathbb{Z}^2}e^{in\zeta_+(z,\theta)}\hat{q}_2(n)\sum_{m\in\mathbb{Z}^2}e^{-im\zeta_-(z,\theta^{\prime})}\bigl(\langle\hat{P}(n)\vert\hat{R}(z)\vert\hat{P}(m)\rangle\hat{q}(m)\hat{\eta}^{(0)}(z,\theta^{\prime})\bigr)_{21},
\end{equation*}
and split it into four parts: $I=I_1+I_2+I_3+I_4$, where
\begin{equation*}
    I_1=\sum_{n=n^{(M)}}\sum_{m=n^{(M)}},\
    I_2=\sum_{n=n^{(M)}}\sum_{m_2\le{M-1}},\
    I_3=\sum_{n_2\le{M-1}}\sum_{m=n^{(M)}},\
    I_4=\sum_{n_2\le{M-1}}\sum_{m_2\le{M-1}}.
\end{equation*}

By using \eqref{asymptotic_e^{zeta_+}}, \eqref{asymptotic_e^{zeta_-}},
\eqref{asymptotic_of_alpha_zeta_+/sqrt_8z^2+1}, and
\eqref{asymptotic_of_<P(n)|R(z)|P(m)>q(m)eta(z,theta)}, we have
$I_j\sim{\cal{O}}(N^{4M-1})$ for $j=2,3$; also, 
$I_4\sim{\cal{O}}(N^{4M-3})$.

About $I_1$, by using the resolvent equation
\begin{equation}
    \label{R(z)=R_>M,>M-1(z)-R_>M,>M-1(z)qR(z)}
    \hat{R}(z)=\hat{R}_{(>M,>M-1)}(z)-\hat{R}_{(>M,>M-1)}(z)\hat{q}_{(\le{M},\le{M-1})}\hat{R}(z),
\end{equation}
we split it into four parts: $I_1=I_{1,1}+I_{1,2,1}-I_{1,2,2,1}-I_{1,2,2,2}$,
where 
\begin{align*}
    I_{1,1}=&\frac{\alpha(\zeta_+(z,\theta))}{\sqrt{8z^2+1}}e^{in^{(M)}(\zeta_+(z,\theta)-\zeta_-(z,\theta^{\prime}))}\hat{q}_2(n^{(M)})\\
    &\cdot\bigl(\langle\hat{P}(n^{(M)})\vert\hat{R}(z)\vert\hat{P}(n^{(M)})\rangle
    \begin{pmatrix}
        \hat{q}_1(n^{(M)})&0\\
        0&0
    \end{pmatrix}
    \hat{\eta}^{(0)}(z,\theta^{\prime})\bigr)_{21},
\end{align*}
\begin{align*}
    I_{1,2,1}=&\frac{\alpha(\zeta_+(z,\theta))}{\sqrt{z^2+1}}e^{in^{(M)}(\zeta_+(z,\theta)-\zeta_-(z,\theta^{\prime}))}\hat{q}_2(n^{(M)})\\
    &\cdot\bigl(\langle\hat{P}(n^{(M)})\vert\hat{R}_{(>M,>M-1)}(z)\vert\hat{P}(n^{(M)})\rangle
    \begin{pmatrix}
        0&0\\
        0&\hat{q}_2(n^{(M)})
    \end{pmatrix}
    \hat{\eta}^{(0)}(z,\theta^{\prime})\bigr)_{21},
\end{align*}
\begin{align*}
    I_{1,2,2,1}=&\frac{\alpha(\zeta_+(z,\theta))}{\sqrt{8z^2+1}}e^{in^{(M)}(\zeta_+(z,\theta)-\zeta_-(z,\theta^{\prime}))}\hat{q}_2(n^{(M)})\\
    &\cdot\bigl(\langle\hat{P}(n^{(M)})\vert\hat{R}_{(>M,>M-1)}(z)\hat{q}_{\le{M-1}}\hat{R}(z)\vert\hat{P}(n^{(M)})\rangle\\
    &\quad\cdot
    \begin{pmatrix}
        0&0\\
        0&\hat{q}_2(n^{(M)})
    \end{pmatrix}
    \hat{\eta}^{(0)}(z,\theta^{\prime})\bigr)_{21},
\end{align*}
\begin{align*}
    I_{1,2,2,2}=&\frac{\alpha(\zeta_+(z,\theta))}{\sqrt{8z^2+1}}e^{in^{(M)}(\zeta_+(z,\theta)-\zeta_-(z,\theta^{\prime}))}\hat{q}(n^{(M)})\\
    &\cdot\bigl(\langle\hat{P}(n^{(M)})\vert\hat{R}_{(>M,>M-1)}(z)\hat{P}(n^{(M)})
    \begin{pmatrix}
        \hat{q}_1(n^{(M)})&0\\
        0&0
    \end{pmatrix}
    \hat{P}(n^{(M)})\hat{R}(z)\vert\hat{P}(n^{(M)})\rangle\\
    &\quad\cdot
    \begin{pmatrix}
        0&0\\
        0&\hat{q}_2(n^{(M)})
    \end{pmatrix}
    \hat{\eta}^{(0)}(z,\theta^{\prime})\bigr)_{21}.
\end{align*}

We know $\hat{R}_{(>M,>M-1)}(z)$, since we have already computed
$\hat{q}_2(n^{(M)})$, which implies that $I_{1,2,1}$ is a known term.

By using \eqref{estimate_for_<P(n)|R(z)|P(m)>q(m)_eta} and
\eqref{asymptotic_of_eta}, we have
\begin{equation*}
    \bigl(\langle\hat{P}(n^{(M)})\vert\hat{R}(z)\vert\hat{P}(n^{(M)})\rangle
    \begin{pmatrix}
        \hat{q}_1(n^{(M)})&0\\
        0&0
    \end{pmatrix}
    \hat{\eta}^{(0)}(z,\theta^{\prime})\bigr)_{21}\sim{\cal{O}}(N^{-2}),
\end{equation*}
which implies, by using \eqref{asymptotic_e^{zeta_+}},
\eqref{asymptotic_e^{zeta_-}}, and
\eqref{asymptotic_of_alpha_zeta_+/sqrt_8z^2+1}, that
$I_{1,1}\sim{\cal{O}}(N^{4M-1})$.

We can write $I_{1,2,2,1}$ as a sum of the following terms:
\begin{align*}
    &\frac{\alpha(\zeta_+(z,\theta))}{\sqrt{8z^2+1}}e^{in^{(M)}(\zeta_+(z,\theta)-\zeta_-(z,\theta^{\prime}))}\\
    &\cdot\bigl(\langle\hat{P}(n^{(M)})\vert\hat{R}_{(>M,>M-1)}(z)\vert\hat{P}(k)\rangle\hat{q}(k)\langle\hat{P}(k)\vert\hat{R}(z)\vert\hat{P}(n^{(M)})\rangle\\
    &\quad\cdot
    \begin{pmatrix}
        0&0\\
        0&\hat{q}(n^{(M)})
    \end{pmatrix}
    \hat{\eta}^{(0)}(z,\theta^{\prime})\bigr)_{21},
\end{align*}
where $k=(k_1,k_2)\in\mathbb{Z}^2$, $k_2\le{M-1}$. From Lemma
\ref{lemma_d_ij(n)>=|n_2|}, we have
\begin{equation}
    \label{d(n^M-k)=>1}
    d(n^{(M)}-k)=d(k-n^{(M)})\ge1,
\end{equation}
\begin{equation}
    \label{d_21(n^M-k)=>1}
    d_{21}(n^{(M)}-k)=d_{12}(k-n^{(M)})\ge1,
\end{equation}
since $(n^{(M)}-k)_2>0$ and $(k-n^{(M)})_2<0$. Then, by using
\eqref{resolvent_estimate_for_R(z)}, \eqref{q_n},
\eqref{estimate_for_<P(n)|R(z)|P(m)>q(m)_eta},
\eqref{asymptotic_of_adjoint_alpha_zeta_-/sqrt_8z^2+1},
\eqref{d(n^M-k)=>1}, and \eqref{d_21(n^M-k)=>1}, we have
\begin{align*}
    &\bigl(\langle\hat{P}(n^{(M)})\vert\hat{R}_{(>M,>M-1)}(z)\vert\hat{P}(k)\rangle\hat{q}(k)\langle\hat{P}(k)\vert\hat{R}(z)\vert\hat{P}(n^{(M)})\rangle\\
    &\ \cdot
    \begin{pmatrix}
        0&0\\
        0&\hat{q}_2(n^{(M)})
    \end{pmatrix}
    \hat{\eta}^{(0)}(z,\theta^{\prime})\bigr)_{21}\\
    &={\cal{O}}(\langle{z}\rangle^{-2d_{21}(n^{(M)}-k)-2})\hat{q}_1(k){\cal{O}}(\langle{z}\rangle^{-2d_{12}(k-n^{(M)})-2})\hat{q}_2(n^{(M)})\frac{\overline{\alpha}({\zeta_-(z,\theta^{\prime})})}{\sqrt{8z^2+1}}\\
    &\quad+{\cal{O}}(\langle{z}\rangle^{-2d(n^{(M)}-k)-1})\hat{q}_2(k){\cal{O}}(\langle{z}\rangle^{-2d(k-n^{(M)})-1})\hat{q}_2(n^{(M)})\frac{\overline{\alpha}(\zeta_-(z,\theta^{\prime}))}{\sqrt{8z^2+1}}\\
    &\sim{\cal{O}}(N^{-5}),
\end{align*}
which infers, by using \eqref{asymptotic_e^{zeta_+}},
\eqref{asymptotic_e^{zeta_-}}, and
\eqref{asymptotic_of_alpha_zeta_+/sqrt_8z^2+1}, that
$I_{1,2,2,1}\sim{\cal{O}}(N^{4M-4})$.

In the same way as above, we have
\begin{align*}
    &\bigl(\langle\hat{P}(n^{(M)})\vert\hat{R}_{(>M,>M-1)}(z)\vert\hat{P}(n^{(M)})\rangle
    \begin{pmatrix}
        \hat{q}_1(n^{(M)})&0\\
        0&0
    \end{pmatrix}\\
    &\ \cdot\langle\hat{P}(n^{(M)})\vert\hat{R}(z)\vert\hat{P}(n^{(M)})\rangle
    \begin{pmatrix}
        0&0\\
        0&\hat{q}_2(n^{(M)})
    \end{pmatrix}
    \hat{\eta}^{(0)}(z,\theta^{\prime})\bigr)_{21}\\
    &={\cal{O}}(\langle{z}\rangle^{-2d_{21}((0,0))-2d_{12}((0,0))-4})\frac{\overline{\alpha}(\zeta_-(z,\theta^{\prime}))}{\sqrt{8z^2+1}}\\
    &\sim{\cal{O}}(N^{-3}),
\end{align*}
which implies $I_{1,2,2,2}\sim{\cal{O}}(N^{-4M-2})$.

Therefore, by taking the known terms into consideration, we can compute
$\hat{q}_1(n^{(M)})$ from the asymptotic expansion of
$B_{11}(z,\theta,\theta^{\prime})$.

Suppose that we have computed $\hat{q}(n)$ for $n_2\ge{p+1}$. Then, by using
\eqref{asymptotic_e^{zeta_+}}, \eqref{asymptotic_e^{zeta_-}},
\eqref{asymptotic_of_alpha_zeta_+/sqrt_8z^2+1}, and
\eqref{asymptotic_of_alpha_zeta_-/sqrt_8z^2+1}, we have
\begin{align}
    \label{asymptotic_B_0,22_for_n_2=<p}
    \begin{split}
        &B_{0,22}(z,\theta,\theta^{\prime})-\sum_{n_2\ge{p+1}}e^{in(\zeta_+(z,\theta)-\zeta_-(z,\theta^{\prime}))}\bigl(\frac{\overline{\alpha}(\zeta_+(z,\theta))}{\sqrt{8z^2+1}}\frac{\alpha(\zeta_-(z,\theta^{\prime}))}{\sqrt{8z^2+1}}\hat{q}_1(n)+\hat{q}_2(n)\bigr)\\
        &\sim{N^{4p}}(a_1(\theta)a_1(\theta^{\prime}))^p\sum_{n_2=p}(b_1(\theta)b_1(\theta^{\prime}))^{n_1}\hat{q}_2(n).
    \end{split}
\end{align}

Let us investigate the asymptotic behavior of
$B_{1,22}(z,\theta,\theta^{\prime})=I\hspace{-.1em}I+I\hspace{-.1em}I\hspace{-.1em}I$,
where
\begin{align*}
    I\hspace{-.1em}I=&\frac{\overline{\alpha}(\zeta_+(z,\theta))}{\sqrt{8z^2+1}}\sum_{n\in\mathbb{Z}^2}e^{in\zeta_+(z,\theta)}\hat{q}_1(n)\\
    &\cdot\sum_{m\in\mathbb{Z}^2}e^{-im\zeta_-(z,\theta^{\prime})}\bigl(\langle\hat{P}(n)\vert\hat{R}(z)\vert\hat{P}(m)\rangle\hat{q}(m)\hat{\eta}^{(0)}(z,\theta^{\prime})\bigr)_{12},
\end{align*}
\begin{equation*}
    I\hspace{-.1em}I\hspace{-.1em}I=\sum_{n\in\mathbb{Z}^2}e^{in\zeta_+(z,\theta)}\hat{q}_2(n)\sum_{m\in\mathbb{Z}^2}e^{-im\zeta_-(z,\theta^{\prime})}\bigl(\langle\hat{P}(n)\vert\hat{R}(z)\vert\hat{P}(m)\rangle\hat{q}(m)\hat{\eta}^{(0)}(z,\theta^{\prime})\bigr)_{22}.
\end{equation*}
The following lemma is proved in the same way as above in
\ref{appendix_asymptotics_of_scattering_amplitude}; so is Lemma
\ref{lemma_asymptotic_of_IV_and_V}.
\begin{lemma}\label{lemma_asymptotic_of_II_and_III}
    We have $I\hspace{-.1em}I\sim{\cal{O}}(N^{4p-1})$ up to a known term,
    which is written as
    \begin{align*}
        &\frac{\overline{\alpha}(\zeta_+(z,\theta))}{\sqrt{8z^2+1}}\sum_{n_2\ge{p+1}}e^{in\zeta_+(z,\theta)}\hat{q}_1(n)\sum_{m_2\ge{p+1}}e^{-im\zeta_-(z,\theta^{\prime})}\\
        &\cdot\bigl(\langle\hat{P}(n)\vert\hat{R}_{>p}(z)\vert\hat{P}(m)\rangle\hat{q}_{>p}(m)\hat{\eta}^{(0)}(z,\theta^{\prime})\bigr)_{12},
    \end{align*}
    and $I\hspace{-.1em}I\hspace{-.1em}I\sim{\cal{O}}(N^{4p-1})$ up to a known
    term, which is
    \begin{equation*}
        \sum_{n_2\ge{p+1}}e^{in\zeta_+(z,\theta)}\hat{q}_2(n)\sum_{m_2\ge{p+1}}e^{-im\zeta_-(z,\theta^{\prime})}\bigl(\langle\hat{P}(n)\vert\hat{R}_{>p}(z)\vert\hat{P}(m)\rangle\hat{q}_{>p}(m)\hat{\eta}^{(0)}(z,\theta^{\prime})\bigr)_{22}.
    \end{equation*}
\end{lemma}

Therefore, we can compute $\hat{q}_2(n)$'s for $n_2=p$ from
\eqref{asymptotic_B_0,22_for_n_2=<p}, since the set
$\{b_1(\theta)b_1(\theta^{\prime})\in\mathbb{R};\theta,\theta^{\prime}\in(0,(1/2)\log{2})\}$
contains open intervals, where $b_1(\cdot)$ is defined by \eqref{b_1}.

Finally, we compute $\hat{q}_1(n)$ for $n_2=p$. By using
\eqref{asymptotic_e^{zeta_+}}, \eqref{asymptotic_e^{zeta_-}},
\eqref{asymptotic_of_alpha_zeta_+/sqrt_8z^2+1}, and
\eqref{asymptotic_of_adjoint_alpha_zeta_-/sqrt_8z^2+1}, we have
\begin{align}
    \label{asymptotic_B_0,11:2}
    \begin{split}
        &B_{0,11}(z,\theta,\theta^{\prime})-\sum_{n_2\ge{p+1}}e^{in(\zeta_+(z,\theta)-\zeta_-(z,\theta^{\prime}))}\bigl(\hat{q}_1(n)+\frac{\alpha(\zeta_+(z,\theta))}{\sqrt{8z^2+1}}\frac{\overline{\alpha}(\zeta_-(z,\theta^{\prime}))}{\sqrt{8z^2+1}}\hat{q}_2(n)\bigr)\\
        &\sim{}N^{4p}(a_1(\theta)a_1(\theta^{\prime}))^p\sum_{n_2=p}(b_1(\theta)b_1(\theta^{\prime}))^{n_1}\hat{q}_1(n)\\
        &\quad+\sum_{n_2=p}e^{in(\zeta_+(z,\theta)-\zeta_-(z,\theta^{\prime}))}\frac{\alpha(\zeta_+(z,\theta))}{\sqrt{8z^2+1}}\frac{\overline{\alpha}(\zeta_-(z,\theta^{\prime}))}{\sqrt{8z^2+1}}\hat{q}_2(n)\\
        &\quad+{\cal{O}}(N^{4p-2}).
    \end{split}
\end{align}
Note that the second term of the right-hand side in \eqref{asymptotic_B_0,11:2}
is known, since we have already computed $\hat{q}_2(n)$'s for $n_2=p$.

Let us investigate the asymptotic behavior of
$B_{1,11}(z,\theta,\theta^{\prime})=I\hspace{-.1em}V+V$, where
\begin{equation*}
    I\hspace{-.1em}V=\sum_{n\in\mathbb{Z}^2}e^{in\zeta_+(z,\theta)}\hat{q}_1(n)\sum_{m\in\mathbb{Z}^2}e^{-im\zeta_-(z,\theta^{\prime})}\bigl(\langle\hat{P}(n)\vert\hat{R}(z)\vert\hat{P}(m)\rangle\hat{q}(m)\hat{\eta}^{(0)}(z,\theta^{\prime})\bigr)_{11},
\end{equation*}
and $V=I$. By using the next lemma, we can compute $\hat{q}_1(n)$'s for $n_2=p$ from
\eqref{asymptotic_B_0,11:2} as before.
\begin{lemma}
    \label{lemma_asymptotic_of_IV_and_V}
    We have $I\hspace{-.1em}V\sim{\cal{O}}(N^{4p-1})$ up to a known term,
    which is written as
    \begin{align*}
        &\sum_{n_2\ge{p+1}}e^{in\zeta_+(z,\theta)}\hat{q}_1(n)\sum_{m_2\ge{p+1}}e^{-im\zeta_-(z,\theta^{\prime})}\\
        &\cdot\bigl(\langle\hat{P}(n)\vert\hat{R}_{(>p,>p-1)}(z)\vert\hat{P}(m)\rangle\hat{q}_{(>p,>p-1)}(m)\hat{\eta}^{(0)}(z,\theta^{\prime})\bigr)_{11},
    \end{align*}
    and $V\sim{\cal{O}}(N^{4p-1})$ up to a known term, which is
    \begin{align*}
        &\frac{\alpha(\zeta_+(z,\theta))}{\sqrt{8z^2+1}}\sum_{n_2\ge{p}}e^{in\zeta_+(z,\theta)}\hat{q}_2(n)\sum_{m_2\ge{p+1}}e^{-im\zeta_-(z,\theta^{\prime})}\\
        &\cdot\bigl(\langle\hat{P}(n)\vert\hat{R}_{(>p,>p-1)}(z)\vert\hat{P}(m)\rangle\hat{q}_{(>p,>p-1)}(m)\hat{\eta}^{(0)}(z,\theta^{\prime})\bigr)_{21}.
    \end{align*}
\end{lemma}

Thus, we can compute all the $q(n)$'s inductively.


%
%
\section{The triangle lattice}

Let $G$ be the triangle lattice. Then, in the same way as in Section 2, we
have $\hat{H}_0\simeq{}6(\Delta_d+1)$ on $\l^2(\mathbb{Z}^2)\simeq{}l^2(G)$ as
follows:
\begin{align*}
    (\hat{H}_0\hat{f})(n)=&\hat{f}(n_1+1,n_2)+\hat{f}(n_1-1,n_2)+\hat{f}(n_1,n_2+1)+\hat{f}(n_1,n_2-1)\\
    &+\hat{f}(n_1-1,n_2+1)+\hat{f}(n_1+1,n_2-1)
\end{align*}
for $\hat{f}=(\hat{f}(n))_{n\in\mathbb{Z}^2}$. See Figure
\ref{fig_triangle_lattice}.
\begin{figure}[htbp]
    \begin{center}
        \includegraphics[height=3cm]{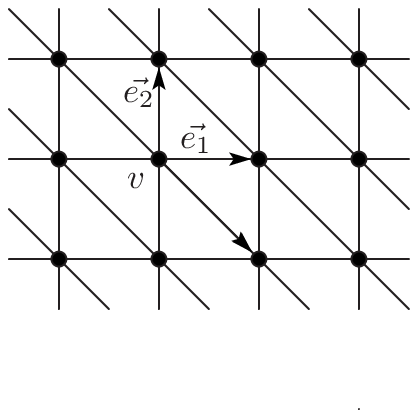}
    \end{center}
    \caption{the triangle lattice\label{fig_triangle_lattice}}
\end{figure}

In this case, the Fourier transform ${\cal{F}}$ is just the Fourier series on
$\mathbb{Z}^2$, which implies that $H_0={\cal{F}}\hat{H}_0{\cal{F}}^*$ is a
multiplication operator by
\begin{equation}
    \label{p(xi):triangle}
    p(\xi)=2(\cos{\xi_1}+\cos{\xi_2}+\cos{(\xi_1-\xi_2)})
\end{equation}
on $L^2(\mathbb{T}^2)$. Here, we note the similarity between
\eqref{p(xi):triangle} and the square of \eqref{p}. Therefore, Mourre estimate
is obtained with the conjugate operator
$A=\nabla{p}\cdot\nabla+\nabla\cdot\nabla{p}$ on any closed interval
$I\subset(-3,6)\setminus\{0\}$. Note also that we no longer need cut-off
functions which appear in the hexagonal lattice case. By using limiting
absorption principle, we have a spectral representation for
$\hat{H}=\hat{H}_0+\hat{q}$; also, a representation for the kernel of the
scattering amplitude $A(\lambda;\theta,\theta^{\prime})$, where $\theta$ and
$\theta^{\prime}\in\mathbb{R}$ are local coordinates of
${\cal{M}_{\lambda}}=\{\xi\in\mathbb{T}^2; p(\xi)=\lambda\}$, which satisfy
\begin{equation*}
    2(\cos{\xi_1(k,\theta)}+\cos{\xi_2(k,\theta)}+\cos{(\xi_1(k,\theta)-\xi_2(k,\theta))})=\pm\sqrt{k}.
\end{equation*}

We have the analytic continuations
$\zeta_{\pm}(z,\theta)=(\zeta_{\pm,1}(z,\theta),\zeta_{\pm,2}(z,\theta))$ from
$\xi(\lambda,\theta)=(\xi_1(\lambda,\theta),\xi_2(\lambda,\theta))$ in the
same way as in Section 4; also, the resolvent estimates as in Isozaki and
Korotyaev \cite{2011arXiv1103.2193I}, where we only have to replace
$\vert\cdot\vert_{l^1}$ with $d(\cdot)$. Therefore, we can adopt the same
strategy for our reconstruction procedure for the potential as in the
hexagonal lattice.


%
%
%
\section*{Acknowledgment}

I would like to thank Professors Hiroshi Isozaki and Evgeny Korotyaev for
their encouragements and invaluable suggestions about the scattering and the
inverse scattering theories for discrete Schr\"{o}dinger operators. I also
thank Mr. Hisashi Morioka for helpful discussions.

%
%
\appendix

%
%
\def\thesection{Appendix \Alph{section}}
\section{}
\label{appendix_spectral_representation}

\def\thesection{\Alph{section}}
\subsection{Proof of Theorem  \ref{spec_rep_1}}

From Lemma \ref{F_0_unitary}, it is uniquely extended to
\begin{equation*}
    {\cal{F}}_0:L^2(\mathbb{T}^2;\mathbb{C}^2)\rightarrow{}L^2((-3,3);L^2({\cal{M}}_{\vert\lambda\vert};\mathbb{C}^2);d\lambda)
\end{equation*}
as an isometry. To prove that ${\cal{F}}_0$ is onto, we note that
${\cal{F}}_0$ maps
$\cup_{\varepsilon>0}C_{\varepsilon}^{\infty}(\mathbb{T}^2;\mathbb{C}^2)$ to a
dense set in 
$L^2((-3,3);L^2({\cal{M}}_{\vert\lambda\vert};\mathbb{C}^2);d\lambda)$.

The second statement is already proved in \eqref{eigen_op_2}.

We note that
\begin{align*}
    (\int_{I_N}{\cal{F}}_0(\lambda)^*g(\lambda)d\lambda,f)_{L^2(\mathbb{T}^2;\mathbb{C}^2)}&=\int_{I_N}({\cal{F}}_0(\lambda)^*g(\lambda),f)_{L^2(\mathbb{T}^2;\mathbb{C}^2)}d\lambda\\
    &=\int_{I_N}(g(\lambda),{\cal{F}}_0(\lambda)f)_{L^2({\cal{M}}_{\vert\lambda\vert};\mathbb{C}^2)}d\lambda
\end{align*}
for $f\in{\cal{H}}^s$, $s>1/2$, and
$g\in{}L^2((-3,3);L^2({\cal{M}}_{\vert\lambda\vert};\mathbb{C}^2);d\lambda)$
(Reed-Simon \cite{MR0493421}), which implies that
\begin{equation*}
    \vert(\int_{I_N}{\cal{F}}_0(\lambda)^*g(\lambda)d\lambda,f)_{L^2(\mathbb{T}^2;\mathbb{C}^2)})\vert\le\Vert{g}\Vert_{L^2((-3,3);L^2({\cal{M}}_{\vert\lambda\vert};\mathbb{C}^2);d\lambda)}\Vert{f}\Vert_{L^2(\mathbb{T}^2;\mathbb{C}^2)}.
\end{equation*}
Then, by Riesz's theorem, we have
\begin{equation*}
    \int_{I_N}{\cal{F}}_0(\lambda)^*g(\lambda)d\lambda\in{}L^2(\mathbb{T}^2;\mathbb{C}^2),
\end{equation*}
\begin{equation*}
    \Vert\int_{I_N}{\cal{F}}_0(\lambda)^*g(\lambda)d\lambda\Vert\le\Vert{g}\Vert_{L^2((-3,3);L^2({\cal{M}}_{\vert\lambda\vert};\mathbb{C}^2);d\lambda)}.
\end{equation*}

Let $I_N$, $N=1,2,\cdots$, be a finite union of compact intervals as such in
the third statement. Then we have
\begin{align*}
    &\Vert\int_{I_N}{\cal{F}}_0(\lambda)^*({\cal{F}}_0f)(\lambda)d\lambda-\int_{I_M}{\cal{F}}_0(\lambda)^*({\cal{F}}_0f)(\lambda)d\lambda\Vert_{L^2(\mathbb{T}^2;\mathbb{C}^2)}^2\\
    &=\Vert\int_{I_N\vartriangle{}I_M}{\cal{F}}_0(\lambda)^*({\cal{F}}_0f)(\lambda)d\lambda\Vert_{L^2(\mathbb{T}^2;\mathbb{C}^2)}^2\\
    &\le\Vert\chi_{I_N\vartriangle{}I_M}{\cal{F}}_0f\Vert_{L^2((-3,3);L^2({\cal{M}}_{\vert\lambda\vert};\mathbb{C}^2);d\lambda)}^2\\
    &=\int_{I_N\vartriangle{}I_M}\Vert{\cal{F}}_0f\Vert_{L^2({\cal{M}}_{\vert\lambda\vert};\mathbb{C}^2)}d\lambda\\
    &\rightarrow{0}
\end{align*}
as $M$, $N\rightarrow\infty$, since
${\cal{F}}_0f\in{}L^2((-3,3);L^2({\cal{M}}_{\vert\lambda\vert};\mathbb{C}^2;d\lambda))$,
which implies that
$\int_{I_N}{\cal{F}}_0(\lambda)^*({\cal{F}}_0f)(\lambda)d\lambda$,
$N=1,2,\cdots$, is Cauchy.

Therefore, there exists $f_0\in{}L^2(\mathbb{T}^2;\mathbb{C}^2)$ such that
\begin{equation*}
    f_0=\text{s-}\lim_{N\rightarrow\infty}\int_{I_N}{\cal{F}}_0(\lambda)^*({\cal{F}}_0f)(\lambda)d\lambda.
\end{equation*}
By Parseval's formula, we have $f=f_0$, since
\begin{align*}
    (f_0,h)_{L^2(\mathbb{T}^2;\mathbb{C}^2)}&=\lim_{N\rightarrow\infty}\int_{I_N}({\cal{F}}_0(\lambda)^*({\cal{F}}_0f)(\lambda),h)_{L^2({\cal{M}}_{\vert\lambda\vert};\mathbb{C}^2)}d\lambda\\
    &=\lim_{N\rightarrow\infty}\int_{I_N}({\cal{F}}_0(\lambda)f,{\cal{F}}_0(\lambda)h)_{L^2({\cal{M}}_{\vert\lambda\vert};\mathbb{C}^2)}d\lambda\\
    &=({\cal{F}}_0f,{\cal{F}}_0h)_{L^2((-3,3);L^2({\cal{M}}_{\vert\lambda\vert};\mathbb{C}^2);d\lambda)}\\
    &=(f,h)_{L^2(\mathbb{T}^2;\mathbb{C}^2)}
\end{align*}
for any $h\in{}C^{\infty}(\mathbb{T}^2;\mathbb{C}^2)$.

\subsection{Proof of Theorem \ref{spectral_rep_H}}

By the limiting absorption principle and the virial theorem, the eigenvalues
of $H$ in $(-3,3)\setminus\{\pm{1},0\}$ accumulate, if possible, only at
$\pm{3}$, $\pm{1}$ or $0$ and there is no singularly continuous spectrum.

Let $-3<a<b<3$. Assume that the closed interval $[a,b]$ does not contain the
eigenvalues of $H$. Then, by Stone's theorem, we have
\begin{align*}
    &\frac{1}{2\pi}\int_a^b((R(\lambda+i0)-R(\lambda-i0))f,g)_{L^2(\mathbb{T}^2;\mathbb{C}^2)}\\
    =&((E_{H_{ac}}((-\infty,b])-E_{H_{ac}}((-\infty,a]))f,g)_{L^2(\mathbb{T}^2;\mathbb{C}^2)},
\end{align*}
which implies that ${\cal{F}}^{(\pm)}$ are partial isometries in the same way
as Theorem \ref{spec_rep_1}.

From \eqref{eigen_op_3}, ${\cal{F}}^{(\pm)}$ diagonalize $H$, which proves the
first statement.

Let $f^{\perp}\in{\cal{H}}_{ac}(H)^{\perp}$. Then, for any
$g\in{}L^2((-3,3);L^2({\cal{M}}_{\vert\lambda\vert};\mathbb{C}^2);d\lambda)$,
we have
\begin{align*}
    (\int_{I_N}{\cal{F}}^{(\pm)}(\lambda)^*g(\lambda),f^{\perp})_{L^2(\mathbb{T}^2;\mathbb{C}^2)}&=\int_{I_N}({\cal{F}}^{(\pm)}(\lambda)^*g(\lambda),f^{\perp})_{L^2(\mathbb{T}^2;\mathbb{C}^2)}d\lambda\\
    &=\int_{I_N}(g(\lambda),{\cal{F}}^{(\pm)}(\lambda)f^{\perp})_{L^2({\cal{M}}_{\vert\lambda\vert};\mathbb{C}^2)}d\lambda\\
    &=(g,{\cal{F}}^{(\pm)}f^{\perp})_{L^2((-3,3);L^2({\cal{M}}_{\vert\lambda\vert};\mathbb{C}^2);d\lambda)}\\
    &=0,
\end{align*}
since ${\cal F}^{(\pm)}f^{\perp}=0$, which implies, By Riesz's theorem, that
\begin{equation*}
    \int_{I_N}{\cal{F}}^{(\pm)}(\lambda)^*g(\lambda)d\lambda\in{}{\cal{H}}_{ac}(H).
\end{equation*}

Thus we can prove the second statement in the same way as in Theorem
\ref{spec_rep_1}.

To prove the third statement, we only have to take the adjoint of
\eqref{eigen_op_3}.


%
%
\def\thesection{Appendix \Alph{section}}
\section{}
\label{appendix_analytic_continuations}

\def\thesection{\Alph{section}}
\subsection{Proof of \eqref{Re_zeta_1} and \eqref{Im_zeta_1}}
\label{proof_of_the_asymptotics_of_zeta_1}

Let $\cos{\frac{E_1}{2}}=x_{c,1}+iy_{c,1}$. Then, from \eqref{cos_zeta_1/2},
we have
\begin{equation}
    \label{x_c,1}
    x_{c,1}=c(\theta)\left\{1-2\frac{\sinh^2{\theta}}{\left(2-e^{2\theta}\right)^2}N^{-2}+{\cal{O}}(N^{-4})\right\},
\end{equation}
\begin{equation}
    \label{y_c,1}
    y_{c,1}=c(\theta)\left\{-4\frac{\sinh^2{\theta}}{\left(2-e^{2\theta}\right)^2}N^{-3}+{\cal{O}}(N^{-5})\right\},
\end{equation}
where
\begin{equation*}
    c=c(\theta)=\frac{e^{-\theta}}{\sqrt{2-e^{2\theta}}}.
\end{equation*}
Also, we have
\begin{equation}
    \label{re_1}
    \cos{\frac{\eta_1}{2}}\cosh{\frac{\kappa_1}{2}}=x_{c,1},
\end{equation}
\begin{equation}
    \label{im_1}
    -\sin{\frac{\eta_1}{2}}\sinh{\frac{\kappa_1}{2}}=y_{c,1},
\end{equation}
where $E_1=\eta_1+i\kappa_1$. From \eqref{y_c,1}, we have $y_{c,1}<0$ for
sufficiently large $N>0$, which implies that $\sinh{\frac{\kappa_1}{2}}>0$,
since we have already proved that $\sin{\frac{\eta_1}{2}}>0$ in Section
\ref{analytic_continuation}.

From \eqref{re_1} and \eqref{im_1}, we have
\begin{equation}
    \label{equation_sin^2_eta_1/2}
    (1-\sin^2{\frac{\eta_1}{2}})(\sin^2{\frac{\eta_1}{2}}+{y_{c,1}}^2)={x_{c,1}}^2\sin^2{\frac{\eta_1}{2}}.
\end{equation}
We put
\begin{equation}
    \label{t_1=sin^2_eta_1/2}
    t_1=\sin^2{\frac{\eta_1}{2}}.
\end{equation}
Then, from \eqref{equation_sin^2_eta_1/2}, we have
\begin{equation}
    \label{equation_t_1}
    {t_1}^2+({x_{c,1}}^2+{y_{c,1}}^2-1)t_1-{y_{c,1}}^2=0.
\end{equation}
By solving \eqref{equation_t_1}, we have
\begin{equation}
    \label{t_1}
    t_1=\frac{1}{2}\left\{-\left({x_{c,1}}^2+{y_{c,1}}^2-1\right)+\sqrt{\left({x_{c,1}}^2+{y_{c,1}}^2-1\right)^2+4{y_{c,1}}^2}\right\},
\end{equation}
since $t_1>0$. By using \eqref{x_c,1} and \eqref{y_c,1}, we can calculate
\begin{equation*}
    {x_{c,1}}^2+{y_{c,1}}^2-1=c^2-1-4c^2\frac{\sinh^2{\theta}}{(2-e^{2\theta})^2}N^{-2}+{\cal{O}}(N^{-4}),
\end{equation*}
which implies
\begin{equation}
    \label{x_c,1^2+y_c,1^2-1>0}
    {x_{c,1}}^2+{y_{c,1}}^2-1>0
\end{equation}
for sufficiently large $N>0$, since
\begin{equation*}
    c^2-1=4\frac{\sinh^2{\theta}}{2-e^{2\theta}}>0.
\end{equation*}
Noticing \eqref{x_c,1^2+y_c,1^2-1>0} and
\begin{equation*}
    {y_{c,1}}^2({x_{c,1}}^2+{y_{c,1}}^2-1)^{-2}={\cal{O}}(N^{-6}),
\end{equation*}
we have, from \eqref{t_1},
\begin{align}
    \label{asymptotic_t_1}
    \begin{split}
        t_1&=\frac{1}{2}\Bigl\{-\left({x_{c,1}}^2+{y_{c,1}}^2-1\right)\Bigr.\\
        &\quad\quad\
        \Bigl.+\left({x_{c,1}}^2+{y_{c,1}}^2-1\right)\left(1+2{y_{c,1}}^2\left({x_{c,1}}^2+{y_{c,1}}^2-1\right)^{-2}+{\cal{O}}(N^{-12})\right)\Bigr\}\\
        &={y_{c,1}}^2({x_{c,1}}^2+{y_{c,1}}^2-1)+{\cal{O}}(N^{-12}).
    \end{split}
\end{align}
Since $\sin{\frac{\eta_1}{2}}>0$ and $y_{c,1}<0$ for sufficiently large $N>0$,
we have, from \eqref{t_1=sin^2_eta_1/2} and \eqref{asymptotic_t_1},
\begin{align*}
    \sin{\frac{\eta_1}{2}}&=-y_{c,1}({x_{c,1}}^2+{y_{c,1}}^2-1)^{-1/2}+{\cal{O}}(N^{-9})\\
    &=2\frac{c}{\sqrt{c^2-1}}\frac{\sinh{\theta}}{(2-e^{2\theta})^2}N^{-3}+{\cal{O}}(N^{-5})\\
    &=\frac{e^{-\theta}}{(2-e^{2\theta})^2}N^{-3}+{\cal{O}}(N^{-5}).
\end{align*}
Therefore, we have
\begin{equation*}
    \eta_1=2m_1\pi+{\cal{O}}(N^{-3})
\end{equation*}
for some $m_1\in\mathbb{Z}$.

From \eqref{y_c,1} and \eqref{im_1}, we can calculate
\begin{align}
    \label{sinh_kappa_1/2}
    \begin{split}
        \sinh{\frac{\kappa_1}{2}}&=-\frac{y_{c,1}}{\sin{\frac{\eta_1}{2}}}\\
        &=({x_{c,1}}^2+{y_{c,1}}^2-1)^{-1/2}+{\cal{O}}(N^{-6})\\
        &=\sqrt{c^2-1}+{\cal{O}}(N^{-2})\\
        &=b+{\cal{O}}(N^{-2}),
    \end{split}
\end{align}
where $b=b(\theta)$ is defined by \eqref{b(theta)}. We put
\begin{equation*}
    s_1=\sinh{\frac{\kappa_1}{2}}.
\end{equation*}
From \eqref{sinh_kappa_1/2}, we note that $s_1\rightarrow{b}>0$ as
$N\rightarrow+\infty$. Then we have
\begin{align*}
    e^{\kappa_1/2}&=s_1+\sqrt{{s_1}^2+1}\\
    &=b+\sqrt{b^2+1}+{\cal{O}}(N^{-2}),
\end{align*}
which implies that
\begin{equation*}
    \frac{\kappa_1}{2}=\log{(b+\sqrt{b^2+1})}+{\cal{O}}(N^{-2}).
\end{equation*}

\subsection{Proof of \eqref{Re_zeta_2} and \eqref{Im_zeta_2}}
\label{proof_of_the_asymptotics_of_zeta_2}

Let $\sin{\frac{E_2}{2}}=x_2+iy_2$. Then, from \eqref{sin_zeta_2/2}, we
have
\begin{equation}
    \label{x_s,2}
    x_{s,2}=a(\theta)\left\{-N-\frac{1}{2}\frac{e^{-2\theta}}{\left(2-e^{2\theta}\right)^2}N^{-1}+{\cal{O}}(N^{-3})\right\},
\end{equation}
\begin{equation}
    \label{y_s,2}
    y_{s,2}=a(\theta)\left\{1-\frac{1}{2}\frac{e^{-2\theta}}{\left(2-e^{2\theta}\right)^2}N^{-2}+{\cal{O}}(N^{-4})\right\},
\end{equation}
where $a=a(\theta)$ is defined by \eqref{a(theta)}. Also, we have
\begin{equation}
    \label{re_2}
    \sin{\frac{\eta_2}{2}}\cosh{\frac{\kappa_2}{2}}=x_{s,2},
\end{equation}
\begin{equation}
    \label{im_2}
    \cos{\frac{\eta_2}{2}}\sinh{\frac{\kappa_2}{2}}=y_{s,2},
\end{equation}
where $E_2=\eta_2+i\kappa_2$. From \eqref{y_s,2}, we have $y_{s,2}>0$ for
sufficiently large $N>0$, which implies that $\sinh{\frac{\kappa_2}{2}}>0$,
since we have already proved that $\cos{\frac{\eta_2}{2}}>0$ in Section
\ref{analytic_continuation}.

From \eqref{re_2} and \eqref{im_2}, we have
\begin{equation}
    \label{equation_cos^2_eta_2/2}
    (1-\cos^2{\frac{\eta_2}{2}})(\cos^2{\frac{\eta_2}{2}}+{y_{s,2}}^2)={x_{s,2}}^2\cos^2\frac{\eta_2}{2}.
\end{equation}
We put
\begin{equation}
    \label{t_2=cos^2_eta_2/2}
    t_2=\cos^2{\frac{\eta_2}{2}}.
\end{equation}
Then, from \eqref{equation_cos^2_eta_2/2}, we have
\begin{equation}
    \label{equation_t_2}
    {t_2}^2+({x_{s,2}}^2+{y_{s,2}}^2-1)t_2-{y_{s,2}}^2=0.
\end{equation}
By solving \eqref{equation_t_2}, we have
\begin{equation}
    \label{t_2}
    t_2=\frac{1}{2}\left\{-\left({x_{s,2}}^2+{y_{s,2}}^2-1\right)+\sqrt{\left({x_{s,2}}^2+{y_{s,2}}^2-1\right)^2+4{y_{s,2}}^2}\right\},
\end{equation}
since $t_2>0$. By using \eqref{x_s,2} and \eqref{y_s,2}, we can calculate
\begin{align*}
    {x_{s,2}}^2+{y_{s,2}}^2-1&=(2-e^{2\theta})N^2+\frac{e^{-2\theta}}{2-e^{2\theta}}+1-e^{2\theta}+{\cal{O}}(N^{-2})\\
    &={\cal{O}}(N^2).
\end{align*}
Noticing that $2-e^{2\theta}>0$ for $0<\theta<\frac{1}{2}\log{2}$, we have
\begin{equation}
    \label{x_s,2^2+y_s,2^2-1>0}
    {x_{s,2}}^2+{y_{s,2}}^2-1>0,
\end{equation}
\begin{equation}
    \label{x_s,2^2+y_s,2^2-1=O(N^-2)}
    {y_{s,2}}^2({x_{s,2}}^2+{y_{s,2}}^2-1)^{-2}={\cal{O}}(N^{-4}),
\end{equation}
for sufficiently large $N>0$. Therefore, by using \eqref{x_s,2^2+y_s,2^2-1>0}
and \eqref{x_s,2^2+y_s,2^2-1=O(N^-2)}, we have, from \eqref{t_2},
\begin{align}
    \label{asymptotic_t_2}
    \begin{split}
        t_2&=\frac{1}{2}\Bigl\{-\left({x_{s,2}}^2+{y_{s,2}}^2-1\right)\Bigr.\\
        &\quad\quad\
        \Bigl.+\left({x_{s,2}}^2+{y_{s,2}}^2-1\right)\left(1+2{y_{s,2}}^2\left({x_{s,2}}^2+{y_{s,2}}^2-1\right)^{-2}+{\cal{O}}(N^{-8})\right)\Bigr\}\\
        &={y_{s,2}}^2({x_{s,2}}^2+{y_{s,2}}^{2}-1)^{-1}+{\cal{O}}(N^{-6})\\
        &=N^{-2}+{\cal{O}}(N^{-4}).
    \end{split}
\end{align}
Since $\cos{\frac{\eta_2}{2}}>0$, we have, from \eqref{t_2=cos^2_eta_2/2} and
\eqref{asymptotic_t_2},
\begin{equation*}
    \cos{\frac{\eta_2}{2}}=N^{-1}+{\cal{O}}(N^{-3}).
\end{equation*}
Thus, we have
\begin{equation*}
    \eta_2=(2m_2+1)\pi+{\cal{O}}(N^{-1})
\end{equation*}
for some $m_2\in\mathbb{Z}$.

From \eqref{y_s,2} and \eqref{im_2}, we can calculate
\begin{align}
    \label{sinh_kapa_2/2}
    \begin{split}
        \sinh{\frac{\kappa_2}{2}}&=\frac{y_{s,2}}{\cos{\frac{\eta_2}{2}}}\\
        &=aN+{\cal{O}}(N^{-1}).
    \end{split}
\end{align}
We put
\begin{equation*}
    s_2=\sinh{\frac{\kappa_2}{2}}.
\end{equation*}
From \eqref{sinh_kapa_2/2}, we note that $s_2\rightarrow+\infty$ as
$N\rightarrow+\infty$. Then we have
\begin{align*}
    e^{\kappa_2/2}&=s_2+\sqrt{{s_2}^2+1}\\
    &=2s_2+\frac{1}{2}{s_2}^{-1}+{\cal{O}}({s_2}^{-3})\\
    &=2aN+{\cal{O}}(N^{-1}),
\end{align*}
which implies that
\begin{equation*}
    \frac{\kappa_2}{2}=\log{(2aN)}+{\cal{O}}(N^{-2}).
\end{equation*}


%
%
\def\thesection{Appendix \Alph{section}}
\section{}
\label{appendix_quasi_triangles_inequality}

\def\thesection{\Alph{section}}
\subsection{The proof of Lemma \ref{lemma_d_ij=<d_ii+d_ij}}

If $i=j$, these inequalities are reduced to the triangle inequality \eqref{triangle_inequality_d(n)}.

Assume that $(i,j)=(1,2)$. We prove the inequality
\eqref{d_ij=<d_ii+d_ij}. For $m=(m_1,m_2)$, $n=(n_1,n_2)\in\mathbb{Z}^2$, we
treat the following four cases:
\begin{enumerate}[{(}i{)}]
\item $m_1-n_1>0$, $m_2-n_2>0$,\label{m_1-n_1>0,m_2-n_2>0a}
\item $m_1-n_1\le{0}$, $m_2-n_2>0$,\label{m_1-n_1=<0,m_2-n_2>0}
\item $m_1-n_1\le{0}$, $m_2-n_2\le{0}$,\label{m_1-n_1=<0,m_2-n_2=<0}
\item $m_1-n_1>0$, $m_2-n_2\le{0}$.\label{m_1-n_1>0,m_2-n_2=<0}
\end{enumerate}

Let us consider the case \eqref{m_1-n_1=<0,m_2-n_2>0}. By taking the location
of $l=(l_1,l_2)\in\mathbb{Z}^2$ into consideration, we treat the following
seven cases:
\begin{enumerate}[{(}\ref{m_1-n_1=<0,m_2-n_2>0}.1{)}]
\item $(m_1-l_1)(m_2-l_2)\ge{0}$, $l_1-n_1\le{0}$, $l_2-n_2>0$,
    \label{(m_1-l_1)(m_2-l_2)=>0,l_1-n_1=<0,l_2-n_2>0}
\item $(m_1-l_1)(m_2-l_2)\le{0}$, $l_1-n_1\le{0}$, $l_2-n_2>0$,
    \label{(m_1-l_1)(m_2-l_2)=<0,l_1-n_1=<0,l_2-n_2>0}
\item $m_1-l_1\le{0}$, $m_2-l_2\ge{0}$, $l_1-n_1\le{0}$, $l_2-n_2\le{0}$,
    \label{m_1-l_1=<0,m_2-l_2=>0,l_1-n_1=<0,l_2-n_2=<0}
\item $m_1-l_1\ge{0}$, $m_2-l_2\ge{0}$, $l_1-n_1\le{0}$, $l_2-n_2\le{0}$,
    \label{m_1-l_1=>0,m_2-l_2=>0,l_1-n_1=<0,l_2-n_2=<0}
\item $m_1-l_1\le{0}$, $m_2-l_2\ge{0}$, $l_1-n_1>0$, $l_2-n_2>0$,
    \label{m_1-l_1=<0,m_2-l_2=>0,l_1-n_1>0,l_2-n_2>0}
\item $m_1-l_1\le{0}$, $m_2-l_2\le{0}$, $l_1-n_1>0$, $l_2-n_2>0$,
    \label{m_1-l_1=<0,m_2-l_2=<0,l_1-n_1>0,l_2-n_2>0}
\item $m_1-l_1\le{0}$, $m_2-l_2\ge{0}$, $l_1-n_1\le{0}$, $l_2-n_2\le{0}$.
    \label{m_1-l_1=<0,m_2-l_2=>0,l_1-n_1=<0,l_2-n_2=<0}
\end{enumerate}

Let us consider, for example, the case
(\ref{m_1-n_1=<0,m_2-n_2>0}.\ref{m_1-l_1=<0,m_2-l_2=>0,l_1-n_1>0,l_2-n_2>0}). Then
we have
\begin{align*}
    d_{12}(m-n)=&\max{\{\vert{m_1-n_1}\vert,\vert{m_2-n_2}\vert-1\}}\\
    \le&\max{\{\vert{m_1-l_1}\vert,\vert{m_2-l_2}\vert\}}+\max{\{\vert{l_1-n_1}\vert,\vert{l_2-n_2}\vert-1\}}.
\end{align*}
By the assumption, we have $\vert{l_2-n_2}\vert-1\ge{0}$, which implies
\begin{align*}
    d_{12}(m-n)\le&\max{\{\vert{m_1-l_1}\vert,\vert{m_2-l_2}\vert\}}+\vert{l_1-n_1}\vert+\vert{l_2-n_2}\vert-1\\
    =&d_{11}(m-l)+d_{12}(l-n).
\end{align*}

We can treat the other cases of \eqref{m_1-n_1=<0,m_2-n_2>0} similarly; also,
the cases \eqref{m_1-n_1>0,m_2-n_2>0a}, \eqref{m_1-n_1=<0,m_2-n_2=<0}, and
\eqref{m_1-n_1>0,m_2-n_2=<0}. We can also prove \eqref{d_ij=<d_ii+d_ij} for
$(i,j)=(2,1)$ by the similar arguments.

By using \eqref{d_ij=<d_ii+d_ij}, we have
\begin{align*}
    d_{ij}(m-n)=&d_{ji}(n-m)\\
    \le&d_{jj}(n-l)+d_{ji}(l-m)\\
    =&d_{ij}(m-l)+d_{jj}(l-n),
\end{align*}
which proves the inequality \eqref{d_ij=<d_ij+d_jj}.

\subsection{The proof of Lemma \ref{lemma_d_ii=<d_ij+d_ji+1}}

If $i=j$, these inequalities are reduced to the triangle inequality \eqref{triangle_inequality_d(n)}.

Assume $(i,j)=(1,2)$. For $m=(m_1,m_2)$, $n=(n_1,n_2)\in\mathbb{Z}^2$, we
treat the following eight cases:
\begin{enumerate}[{(}i{)}]
\item $m_1-n_1>0$, $m_2-n_2>0$,\label{m_1-n_1>0,m_2-n_2>0b}
\item $m_1-n_1<0$, $m_2-n_2>0$,\label{m_1-n_1<0,m_2-n_2>0}
\item $m_1-n_1<0$, $m_2-n_2<0$,\label{m_1-n_1<0,m_2-n_2<0}
\item $m_1-n_1>0$, $m_2-n_2<0$,\label{m_1-n_1>0,m_2-n_2<0}
\item $m_1=n_1$, $m_2-n_2>0$,\label{m_1=n_1,m_2-n_2>0}
\item $m_1=n_1$, $m_2-n_2<0$,\label{m_1=n_1,m_2-n_2<0}
\item $m_1-n_1>0$, $m_2=n_2$,\label{m_1-n_1>0,m_2=n_2}
\item $m_1-n_1<0$, $m_2=n_2$.\label{m_1-n_1<0,m_2=n_2}
\end{enumerate}

Let us consider the case \eqref{m_1-n_1<0,m_2-n_2>0}. We treat the following
nine cases, depending on the location of $l=(l_1,l_2)\in\mathbb{Z}^2$:
\begin{enumerate}[{(}\ref{m_1-n_1<0,m_2-n_2>0}.1{)}]
\item $l_1-m_1<0$, $l_2-n_2\ge0$, $m_1-l_1\le0$, $m_2-l_2>0$,
    \label{l_1-m_1<0,l_2-n_2=>0,m_1-l_1=<0,m_2-l_2>0}
\item $l_1-m_1<0$, $l_2-n_2\ge0$, $m_1-l_1\le0$, $m_2-l_2\le0$,
    \label{l_1-m_1<0,l_2-n_2=>0,m_1-l_1=<0,m_2-l_2=<0}
\item $l_1-m_1<0$, $l_2-n_2\ge0$, $m_1-l_1>0$, $m_2-l_2>0$,
    \label{l_1-m_1<0,l_2-n_2=>0,m_1-l_1>0,m_2-l_2>0}
\item $l_1-m_1<0$, $l_2-n_2\ge0$, $m_1-l_1>0$, $m_2-l_2\le0$,
    \label{l_1-m_1<0,l_2-n_2=>0,m_1=l_1>0,m_2-l_2=<0}
\item $l_1-m_1\ge0$, $l_2-n_2\ge0$, $m_1-l_1\le0$, $m_2-l_2>0$,
    \label{l_1-m_1=>0,l_2-n_2=>0,m_1-l_1=<0,m_2-l_2>0}
\item $l_1-m_1\ge0$, $l_2-n_2\ge0$, $m_1-l_1\le0$, $m_2-l_2\le0$,
    \label{l_1-m_1=>0,l_2-n_2=>0,m_1-l_1=<0,m_2-l_2=<0}
\item $l_1-m_1<0$, $l_2-n_2\ge0$, $m_1-l_1\le0$, $m_2-l_2>0$,
    \label{l_1-m_1<0,l_2-n_2=>0,m_1-l_1=<0,m_2-l_2>0}
\item $l_1-m_1<0$, $l_2-n_2<0$, $m_1-l_1>0$, $m_2-l_2>0$,
    \label{l_1-m_1<0,l_2-n_2<0,m_1-l_1>0,m_2-l_2>0}
\item $l_1-m_1\ge0$, $l_2-n_2<0$, $m_1-l_1\le0$, $m_2-l_2>0$.
    \label{l_1-m_1=>0,l_2-n_2<0,m_1-l_1=<0,m_2-l_2>0}
\end{enumerate}

Let us consider, for example, the case
(\ref{m_1-n_1<0,m_2-n_2>0}.\ref{l_1-m_1<0,l_2-n_2=>0,m_1-l_1>0,m_2-l_2>0}). Then
we have
\begin{align*}
    d_{11}(m-n)=&\max{\{\vert{m_1-n_1}\vert,\vert{m_2-n_2}\vert\}}\\
    =&\max{\{(n_1-l_1)+(l_1-m_1),\vert{m_2-n_2}\vert\}}.
\end{align*}
By the assumption, we have $l_1-m_1<0$ and $\vert{m_2-l_2}\vert-1\ge{0}$,
which implies that
\begin{align*}
    d_{11}(m-n)\le&\max{\{\vert{n_1-l_1}\vert,\vert{m_2-n_2}\vert\}}\\
    \le&\max{\{\vert{l_1-n_1}\vert-1,\vert{m_2-l_2}\vert+\vert{l_2-n_2}\vert-1\}}+1\\
    \le&(\vert{m_2-l_2}\vert-1)+\max{\{\vert{l_1-n_1}\vert-1,\vert{l_2-n_2}\vert\}}+1\\
    \le&d_{12}(m-l)+d_{21}(l-n)+1.
\end{align*}

We can treat the other cases of \eqref{m_1-n_1<0,m_2-n_2>0} similarly; also,
the cases \eqref{m_1-n_1>0,m_2-n_2>0b}, \eqref{m_1-n_1<0,m_2-n_2<0}, and
\eqref{m_1-n_1>0,m_2-n_2<0}.

About \eqref{m_1=n_1,m_2-n_2<0}, we treat the following six cases:
\begin{enumerate}[{(}\ref{m_1=n_1,m_2-n_2<0}.1{)}]
\item $l_1-n_1\ge0$, $l_2-n_2<0$, $m_1-l_1\le0$, $m_2-l_2>0$,
    \label{l_1-n_1=>0,l_2-n_2<0,m_1-l_1=<0,m_2-l_2>0}
\item $l_1-n_1\ge0$, $l_2-n_2<0$, $m_1-l_1\le0$, $m_2-l_2\le0$,
    \label{l_1-n_1=>0,l_2-n_2<0,m_1-l_1=<0,m_2-l_2=<0}
\item $l_1-n_1<0$, $l_2-n_2<0$, $m_1-l_1>0$, $m_2-l_2>0$,
    \label{l_1-n_1<0,l_2-n_2<0,m_1-l_1>0,m_2-l_2>0}
\item $l_1-n_1<0$, $l_2-n_2<0$, $m_1-l_1>0$, $m_2-l_2\le0$,
    \label{l_1-n_1<0,l_2-n_2<0,m_1-l_1>0,m_2-l_2=<0}
\item $l_1-n_1\ge0$, $l_2-n_2\ge0$, $m_1-l_1\le0$, $m_2-l_2\le0$,
    \label{l_1-n_1=>0,l_2-n_2=>0,m_1-l_1=<0,m_2-l_2=<0}
\item $l_1-n_1<0$, $l_2-n_2\ge0$, $m_1-l_1>0$, $m_2-l_2\le0$.
    \label{l_1-n_1<0,l_2-n_2=>0,m_1-l_1>0,m_2-l_2=<0}
\end{enumerate}

Let us consider, for example, the case
(\ref{m_1=n_1,m_2-n_2<0}.\ref{l_1-n_1=>0,l_2-n_2<0,m_1-l_1=<0,m_2-l_2>0}). Then
we have
\begin{align*}
    d_{11}(m-n)=&\vert{m_2-n_2}\vert\\
    =&(n_2-l_2)+(l_2-m_2).
\end{align*}
By the assumption, we have $l_2-m_2<0$, which implies
\begin{align*}
    d_{11}(m-n)\le&(\vert{l_2-n_2}\vert-1)+1\\
    \le&\max{\{\vert{l_1-n_1}\vert,\vert{l_2-n_2}\vert-1\}}+1\\
    =&d_{21}(l-n)+1\\
    \le&d_{12}(m-l)+d_{21}(l-n)+1.
\end{align*}

We can treat the other cases of \eqref{m_1=n_1,m_2-n_2<0} similarly; also, the
cases \eqref{m_1=n_1,m_2-n_2>0}, \eqref{m_1-n_1>0,m_2=n_2}, and
\eqref{m_1-n_1<0,m_2=n_2}.

The similar arguments prove the inequality for $(i,j)=(2,1)$.


%
%
\def\thesection{Appendix \Alph{section}}
\section{}
\label{appendix_asymptotics_of_scattering_amplitude}

\def\thesection{\Alph{section}}
\subsection{Proof of Lemma \ref{lemma_asymptotic_of_II_and_III}}

\begin{flushleft}
    {\underline{\it{Asymptotic behavior of $I\hspace{-.1em}I$}}}
\end{flushleft}

We split $I\hspace{-.1em}I$ into four parts:
$I\hspace{-.1em}I=I\hspace{-.1em}I_1+I\hspace{-.1em}I_2+I\hspace{-.1em}I_3+I\hspace{-.1em}I_4$,
where
\begin{equation*}
    I\hspace{-.1em}I_1=\sum_{n_2\ge{p+1}}\sum_{m_2\ge{p+1}},
    I\hspace{-.1em}I_2=\sum_{n_2\ge{p+1}}\sum_{m_2\le{p}},
    I\hspace{-.1em}I_3=\sum_{n_2\le{p}}\sum_{m_2\ge{p+1}},
    I\hspace{-.1em}I_4=\sum_{n_2\le{p}}\sum_{m_2\le{p}}.
\end{equation*}

About $I\hspace{-.1em}I_4$, it is a sum of the following terms:
\begin{equation}
    \label{overline_alpha_zeta_+/sqrt(8z^2+1)e_q_1_e(<P(n)|R(z)|P(m)>q_eta)12}
    \frac{\overline{\alpha}(\zeta_+(z,\theta))}{\sqrt{8z^2+1}}e^{in\zeta_+(z,\theta)}\hat{q}_1(n)e^{-im\zeta_-(z,\theta^{\prime})}\bigl(\langle\hat{P}(n)\vert\hat{R}(z)\vert\hat{P}(m)\rangle\hat{q}(m)\hat{\eta}^{(0)}(z,\theta^{\prime})\bigr)_{12},
\end{equation}
where $n_2\le{p}$ and $m_2\le{p}$. Then, by using
\eqref{asymptotic_of_adjoint_alpha_zeta_+/sqrt_8z^2+1},
\eqref{asymptotic_e^{zeta_+}}, \eqref{asymptotic_e^{zeta_-}}, and
\eqref{asymptotic_of_<P(n)|R(z)|P(m)>q(m)eta(z,theta)}, we have
$I\hspace{-.1em}I_4\sim{\cal{O}}(N^{2(n_2+m_2)-3})\sim{\cal{O}}(N^{4p-3})$.

About $I\hspace{-.1em}I_3$, we note that it is a sum of
\eqref{overline_alpha_zeta_+/sqrt(8z^2+1)e_q_1_e(<P(n)|R(z)|P(m)>q_eta)12}'s,
where $n_2\le{p}$ and $m_2\ge{p+1}$. By using
\eqref{estimate_for_<P(n)|R(z)|P(m)>q(m)_eta}, we have
\begin{align}
    \label{estimate_for_(<P(n)|R(z)|P(m)>q(m)_eta)_12}
    \begin{split}
        &\bigl(\langle\hat{P}(n)\vert\hat{R}(z)\vert\hat{P}(m)\rangle\hat{q}(m)\hat{\eta}(z,\theta^{\prime})\bigr)_{12}\\
        &={\cal{O}}(\langle{z}\rangle^{-2d(n-m)-1})\hat{q}_1(m)\frac{\alpha(\zeta_-(z,\theta^{\prime}))}{\sqrt{8z^2+1}}+{\cal{O}}(\langle{z}\rangle^{-2d_{12}(n-m)-2})\hat{q}_2(m).
    \end{split}
\end{align}
From Lemma \ref{lemma_d_ij(n)>=|n_2|}, we have
\begin{equation}
    \label{d_12(n-m)=>|n_2-m_2|:2.3}
    d_{12}(n-m)\ge\vert{n_2-m_2}\vert,
\end{equation}
since $n_2-m_2\le{0}$; also,
\begin{equation}
    \label{d(n-m)=>|n_2-m_2|:2.3}
    d(n-m)\ge\vert{n_2-m_2}\vert.
\end{equation}
Then, by using \eqref{asymptotic_of_alpha_zeta_-/sqrt_8z^2+1},
\eqref{d_12(n-m)=>|n_2-m_2|:2.3}, and \eqref{d(n-m)=>|n_2-m_2|:2.3}, we have
${\cal{O}}(N^{-2\vert{n_2-m_2}\vert-2})$ as the asymptotic behavior of
\eqref{estimate_for_(<P(n)|R(z)|P(m)>q(m)_eta)_12}. Moreover, by using
\eqref{asymptotic_of_adjoint_alpha_zeta_+/sqrt_8z^2+1},
\eqref{asymptotic_e^{zeta_+}}, and \eqref{asymptotic_e^{zeta_-}}, we have
$I\hspace{-.1em}I_3\sim{\cal{O}}(N^{2(n_2+m_2)-2\vert{n_2-m_2}\vert-3})$. Since
\begin{align}
    \label{2(n_2+m_2)-2|n_2-m_2|=<4p}
    \begin{split}
        2(n_2+m_2)-2\vert{n_2-m_2}\vert&=2((m_2-n_2)-\vert{n_2-m_2}\vert)+4n_2\\
        &\le4n_2,
    \end{split}
\end{align}
we have $I\hspace{-.1em}I_3\sim{\cal{O}}(N^{4p-3})$.

About $I\hspace{-.1em}I_2$, we note that
\begin{equation*}
    d_{12}(n-m)\ge\vert{n_2-m_2}\vert-1,
\end{equation*}
since $n_2-m_2\ge{0}$, which infers
$I\hspace{-.1em}I_2\sim{\cal{O}}(N^{4p-1})$ in the same way as
$I\hspace{-.1em}I_3$.

About $I\hspace{-.1em}I_1$, by using the resolvent equation
\begin{equation}
    \label{R(z)=R_>p(z)-R_>p(z)qR(z)}
    \hat{R}(z)=\hat{R}_{>p}(z)-\hat{R}_{>p}(z)\hat{q}_{\le{p}}\hat{R}(z),
\end{equation}
we split it into two parts:
$I\hspace{-.1em}I_1=I\hspace{-.1em}I_{1,1}-I\hspace{-.1em}I_{1,2}$, where
\begin{align*}
    I\hspace{-.1em}I_{1,1}=&\frac{\overline{\alpha}(\zeta_+(z,\theta))}{\sqrt{8z^2+1}}\sum_{n_2\ge{p+1}}e^{in\zeta_+(z,\theta)}\hat{q}_1(n)\sum_{m_2\ge{p+1}}e^{-im\zeta_-(z,\theta^{\prime})}\\
    &\cdot\bigl(\langle\hat{P}(n)\vert\hat{R}_{>p}(z)\vert\hat{P}(m)\rangle\hat{q}(m)\hat{\eta}^{(0)}(z,\theta^{\prime})\bigr)_{12},
\end{align*}
\begin{align*}
    I\hspace{-.1em}I_{1,2}=&\frac{\overline{\alpha}(\zeta_+(z,\theta))}{\sqrt{8z^2+1}}\sum_{n_2\ge{p+1}}e^{in\zeta_+(z,\theta)}\hat{q}_1(n)\sum_{m_2\ge{p+1}}e^{-im\zeta_-(z,\theta^{\prime})}\\
    &\cdot\bigl(\langle\hat{P}(n)\vert\hat{R}_{>p}(z)\hat{q}_{\le{p}}\hat{R}(z)\vert\hat{P}(m)\rangle\hat{q}(m)\hat{\eta}^{(0)}(z,\theta^{\prime})\bigr)_{12}.
\end{align*}

We know $\hat{R}_{>p}(z)$, since we have already computed $\hat{q}(n)$ for
$n_2>p$ by the assumption, which means that $I\hspace{-.1em}I_{1,1}$ is a
known term.

$I\hspace{-.1em}I_{1,2}$ is a sum of the following terms:
\begin{align*}
    &\frac{\overline{\alpha}(\zeta_+(z,\theta))}{\sqrt{8z^2+1}}e^{in\zeta_+(z,\theta)}\hat{q}_1(n)e^{-im\zeta_-(z,\theta^{\prime})}\\
    &\cdot\bigl(\langle\hat{P}(n)\vert\hat{R}_{>p}(z)\vert\hat{P}(k)\rangle\hat{q}(k)\langle\hat{P}(k)\vert\hat{R}(z)\vert\hat{P}(m)\rangle\hat{q}(m)\hat{\eta}^{(0)}(z,\theta^{\prime})\bigr)_{12},
\end{align*}
where $n_2\ge{p+1}$, $k_2\le{p}$, and $m_2\ge{p+1}$. By using
\eqref{resolvent_estimate_for_R(z)}, \eqref{q_n}, and
\eqref{estimate_for_<P(n)|R(z)|P(m)>q(m)_eta}, we have
\begin{align}
    \label{estimate_(bracket_q_bracket_q_eta)_12}
    \begin{split}
        &\bigl(\langle\hat{P}(n)\vert\hat{R}_{>p}(z)\vert\hat{P}(k)\rangle\hat{q}(k)\langle\hat{P}(k)\vert\hat{R}(z)\vert\hat{P}(m)\rangle\hat{q}(m)\hat{\eta}^{(0)}(z,\theta^{\prime})\bigr)_{12}\\
        &={\cal{O}}(\langle{z}\rangle^{-2d(n-k)-1})\hat{q}_1(k){\cal{O}}(\langle{z}\rangle^{-2d(k-m)-1})\hat{q}_1(m)\frac{\alpha(\zeta_-(z,\theta^{\prime}))}{\sqrt{8z^2+1}}\\
        &\quad+{\cal{O}}(\langle{z}\rangle^{-2d(n-k)-1})\hat{q}_1(k){\cal{O}}(\langle{z}\rangle^{-2d_{12}(k-m)-2})\hat{q}_2(m)\\
        &\quad+{\cal{O}}(\langle{z}\rangle^{-2d_{12}(n-k)-2})\hat{q}_2(k){\cal{O}}(\langle{z}\rangle^{-2d_{21}(k-m)-2})\hat{q}_1(m)\frac{\alpha(\zeta_-(z,\theta^{\prime}))}{\sqrt{8z^2+1}}\\
        &\quad+{\cal{O}}(\langle{z}\rangle^{-2d_{12}(n-k)-2})\hat{q}_2(k){\cal{O}}(\langle{z}\rangle^{-2d(k-m)-1})\hat{q}_2(m).
    \end{split}
\end{align}
From Lemma \ref{lemma_d_ij(n)>=|n_2|}, we have
\begin{equation}
    \label{d_12(n-k)=>|n_2-k_2|-1:2.1.2}
    d_{12}(n-k)\ge\vert{n_2-k_2}\vert-1,
\end{equation}
\begin{equation}
    \label{d_12(k-n)=>|k_2-n_2|:2.1.2}
    d_{12}(k-n)\ge\vert{k_2-n_2}\vert,
\end{equation}
\begin{equation}
    \label{d_21(k-m)=>|k_2-m_2|-1:2.1.2}
    d_{21}(k-m)\ge\vert{k_2-m_2}\vert-1,
\end{equation}
since $k_2-m_2<0$ and $n_2-k_2>0$; also,
\begin{equation}
    \label{d(n-k)=>|n_2-k_2|:2.1.2}
    d(n-k)\ge\vert{n_2-k_2}\vert,
\end{equation}
\begin{equation}
    \label{d(k-m)=>|k_2-m_2|:2.1.2}
    d(k-m)\ge\vert{k_2-m_2}\vert.
\end{equation}
Then, from Lemmas \ref{lemma_d_triangle_ineq}, \ref{lemma_d_ij=<d_ii+d_ij},
and \ref{lemma_d_ii=<d_ij+d_ji+1}, by using
\eqref{asymptotic_of_alpha_zeta_-/sqrt_8z^2+1},
\eqref{d_12(n-k)=>|n_2-k_2|-1:2.1.2}, \eqref{d_12(k-n)=>|k_2-n_2|:2.1.2},
\eqref{d_21(k-m)=>|k_2-m_2|-1:2.1.2}, \eqref{d(n-k)=>|n_2-k_2|:2.1.2}, and
\eqref{d(k-m)=>|k_2-m_2|:2.1.2}, we have
${\cal{O}}(N^{-2\vert{n_2-k_2}\vert-2\vert{k_2-m_2}\vert-1})$ as the
asymptotic behavior of
\eqref{estimate_(bracket_q_bracket_q_eta)_12}. Moreover, by using
\eqref{asymptotic_of_adjoint_alpha_zeta_+/sqrt_8z^2+1},
\eqref{asymptotic_e^{zeta_+}}, and \eqref{asymptotic_e^{zeta_-}}, we have
$I\hspace{-.1em}I_{1,2}\sim{\cal{O}}(N^{2(n_2+m_2)-2\vert{n_2-k_2}\vert-2\vert{k_2-m_2}\vert-2})$. Since
\begin{align}
    \label{2(n_2-m_2)-2|n_2-k_2|-2|k_2-m_2|=<4p}
    \begin{split}
        2(n_2-m_2)-2\vert{n_2-k_2}\vert-2\vert{k_2-m_2}\vert&=2((n_2-k_2)-\vert{n_2-k_2}\vert)\\
        &\quad+2((m_2-k_2)-\vert{m_2-k_2}\vert)+4k_2\\
        &\le4k_2,
    \end{split}
\end{align}
we have $I\hspace{-.1em}I_{1,2}\sim{\cal O}(N^{4p-2})$.

\begin{flushleft}
    {\underline{\it{Asymptotic behavior of
          $I\hspace{-.1em}I\hspace{-.1em}I$}}}
\end{flushleft}

We split $I\hspace{-.1em}I\hspace{-.1em}I$ into four parts:
$I\hspace{-.1em}I\hspace{-.1em}I=I\hspace{-.1em}I\hspace{-.1em}I_1+I\hspace{-.1em}I\hspace{-.1em}I_2+I\hspace{-.1em}I\hspace{-.1em}I_3+I\hspace{-.1em}I\hspace{-.1em}I_4$,
where
\begin{equation*}
    I\hspace{-.1em}I\hspace{-.1em}I_1=\sum_{n_2\ge{p+1}}\sum_{m_2\ge{p+1}},
    I\hspace{-.1em}I\hspace{-.1em}I_2=\sum_{n_2\ge{p+1}}\sum_{m_2\le{p}},
    I\hspace{-.1em}I\hspace{-.1em}I_3=\sum_{n_2\le{p}}\sum_{m_2\ge{p+1}},
    I\hspace{-.1em}I\hspace{-.1em}I_4=\sum_{n_2\le{p}}\sum_{m_2\le{p}}.
\end{equation*}

About $I\hspace{-.1em}I\hspace{-.1em}I_4$, it is a sum of the following terms:
\begin{equation}
    \label{e_q_2_e<P(n)|R(z)|P(m)>q(n)_eta_22}
    e^{in\zeta_+(z,\theta)}\hat{q}_2(n)e^{-im\zeta_-(z,\theta^{\prime})}\bigl(\langle\hat{P}(n)\vert\hat{R}(z)\vert\hat{P}(m)\rangle\hat{q}(m)\hat{\eta}^{(0)}(z,\theta^{\prime})\bigr)_{22},
\end{equation}
where $n_2\le{p}$ and $m_2\le{p}$. Then, by using
\eqref{asymptotic_e^{zeta_+}}, \eqref{asymptotic_e^{zeta_-}}, and
\eqref{asymptotic_of_<P(n)|R(z)|P(m)>q(m)eta(z,theta)}, we have
$I\hspace{-.1em}I\hspace{-.1em}I_4\sim{\cal{O}}(N^{2(n_2+m_2)-1})\sim{\cal{O}}(N^{4p-1})$.

About $I\hspace{-.1em}I\hspace{-.1em}I_3$, it is a sum of
\eqref{e_q_2_e<P(n)|R(z)|P(m)>q(n)_eta_22}'s, where $n_2\le{p}$ and
$m_2\ge{p+1}$. By using \eqref{estimate_for_<P(n)|R(z)|P(m)>q(m)_eta}, we have
\begin{align}
    \label{estimate_for_(<P(n)|R(z)|P(m)>q(m)_eta)_22}
    \begin{split}
        &\bigl(\langle\hat{P}(n)\vert\hat{R}(z)\vert\hat{P}(m)\rangle\hat{q}(m)\hat{\eta}^{(0)}(z,\theta^{\prime})\bigr)_{22}\\
        &={\cal{O}}(\langle{z}\rangle^{-2d_{21}(n-m)-2})\hat{q}_1(m)\frac{\alpha(\zeta_-(z,\theta^{\prime}))}{\sqrt{8z^2+1}}+{\cal{O}}(\langle{z}\rangle^{-2d(n-m)-1})\hat{q}_2(m).
    \end{split}
\end{align}
From Lemma \ref{lemma_d_ij(n)>=|n_2|}, we have
\begin{equation}
    \label{d_21(n-m)=>|n_2-m_2|-1:3.3}
    d_{21}(n-m)\ge\vert{n_2-m_2}\vert-1,
\end{equation}
since $n_2-m_2\le0$; also,
\begin{equation}
    \label{d(n-m)=>|n_2-m_2|:3.3}
    d(n-m)\ge\vert{n_2-m_2}\vert.
\end{equation}
Then, by using \eqref{asymptotic_of_alpha_zeta_-/sqrt_8z^2+1},
\eqref{d_21(n-m)=>|n_2-m_2|-1:3.3}, and \eqref{d(n-m)=>|n_2-m_2|:3.3}, we have
${\cal{O}}(N^{-2\vert{n_2-m_2}\vert-1})$ as the asymptotic behavior of
\eqref{estimate_for_(<P(n)|R(z)|P(m)>q(m)_eta)_22}. Moreover, by using
\eqref{asymptotic_e^{zeta_+}} and \eqref{asymptotic_e^{zeta_-}}, we have
$I\hspace{-.1em}I\hspace{-.1em}I_3\sim{\cal{O}}(N^{2(n_2+m_2)-2\vert{n_2-m_2}\vert-1})$. The
estimate \eqref{2(n_2+m_2)-2|n_2-m_2|=<4p} implies
$I\hspace{-.1em}I\hspace{-.1em}I_3\sim{\cal{O}}(N^{4p-1})$.

About $I\hspace{-.1em}I\hspace{-.1em}I_2$, we note that
\begin{equation*}
    d_{21}(n-m)\ge\vert{n_2-m_2}\vert,
\end{equation*}
since $n_2-m_2\ge0$, which infers
$I\hspace{-.1em}I\hspace{-.1em}I_2\sim{\cal{O}}(N^{4p-1})$ in the same way as
$I\hspace{-.1em}I\hspace{-.1em}I_3$.

About $I\hspace{-.1em}I\hspace{-.1em}I_1$, by using the resolvent estimate
\eqref{R(z)=R_>p(z)-R_>p(z)qR(z)}, we split it into two parts:
$I\hspace{-.1em}I\hspace{-.1em}I_1=I\hspace{-.1em}I\hspace{-.1em}I_{1,1}-I\hspace{-.1em}I\hspace{-.1em}I_{1,2}$,
where
\begin{align*}
    I\hspace{-.1em}I\hspace{-.1em}I_{1,1}=&\sum_{n_2\ge{p+1}}e^{in\zeta_+(z,\theta)}\hat{q}_2(n)\sum_{m_2\ge{p+1}}e^{-im\zeta_-(z,\theta^{\prime})}\\
    &\cdot\bigl(\langle\hat{P}(n)\vert\hat{R}_{>p}(z)\vert\hat{P}(m)\rangle\hat{q}(m)\hat{\eta}^{(0)}(z,\theta^{\prime})\bigr)_{22},
\end{align*}
\begin{align*}
    I\hspace{-.1em}I\hspace{-.1em}I_{1,2}=&\sum_{n_2\ge{p+1}}e^{in\zeta_+(z,\theta)}\hat{q}_2(n)\sum_{m_2\ge{p+1}}e^{-im\zeta_-(z,\theta^{\prime})}\\
    &\cdot\bigl(\langle\hat{P}(n)\vert\hat{R}_{>p}(z)\hat{q}_{\le{p}}\hat{R}(z)\vert\hat{P}(m)\rangle\hat{q}(m)\hat{\eta}^{(0)}(z,\theta^{\prime})\bigr)_{22}.
\end{align*}

$I\hspace{-.1em}I\hspace{-.1em}I_{1,1}$ is a known term as before.

$I\hspace{-.1em}I\hspace{-.1em}I_{1,2}$ is a sum of the following terms:
\begin{align*}
    &e^{in\zeta_+(z,\theta)}\hat{q}_2(n)e^{-im\zeta_-(z,\theta^{\prime})}\\
    &\cdot\bigl(\langle\hat{P}(n)\vert\hat{R}_{>p}(z)\vert\hat{P}(k)\rangle\hat{q}(k)\langle\hat{P}(k)\vert\hat{R}(z)\vert\hat{P}(m)\rangle\hat{q}(m)\hat{\eta}^{(0)}(z,\theta^{\prime})\bigr)_{22},
\end{align*}
where $n_2\ge{p+1}$, $k_2\le{p}$, and $m_2\ge{p+1}$. By using
\eqref{resolvent_estimate_for_R(z)}, \eqref{q_n}, and
\eqref{estimate_for_<P(n)|R(z)|P(m)>q(m)_eta}, we have
\begin{align}
    \label{estimate_(bracket_q_bracket_q_eta)_22}
    \begin{split}
        &\bigl(\langle\hat{P}(n)\vert\hat{R}_{>p}(z)\vert\hat{P}(k)\rangle\hat{q}(k)\langle\hat{P}(k)\vert\hat{R}(z)\vert\hat{P}(m)\rangle\hat{q}(m)\hat{\eta}^{(0)}(z,\theta^{\prime})\bigr)_{22}\\
        &={\cal{O}}(\langle{z}\rangle^{-2d_{21}(n-k)-2})\hat{q}_1(k){\cal{O}}(\langle{z}\rangle^{-2d(k-m)-1})\hat{q}_1(m)\frac{\alpha(\zeta_-(z,\theta^{\prime}))}{\sqrt{8z^2+1}}\\
        &\quad+{\cal{O}}(\langle{z}\rangle^{-2d_{21}(n-k)-2})\hat{q}_1(k){\cal{O}}(\langle{z}\rangle^{-2d_{12}(k-m)-2})\hat{q}_2(m)\\
        &\quad+{\cal{O}}(\langle{z}\rangle^{-2d(n-k)-1})\hat{q}_2(k){\cal{O}}(\langle{z}\rangle^{-2d_{21}(k-m)-2})\hat{q}_1(m)\frac{\alpha(\zeta_-(z,\theta^{\prime}))}{\sqrt{8z^2+1}}\\
        &\quad+{\cal{O}}(\langle{z}\rangle^{-2d(n-k)-1})\hat{q}_2(k){\cal{O}}(\langle{z}\rangle^{-2d(k-m)-1})\hat{q}_2(m).
    \end{split}
\end{align}
From Lemma \ref{lemma_d_ij(n)>=|n_2|}, we have
\begin{equation}
    \label{d_21(n-k)=>|n_2-k_2|:3.1.2}
    d_{21}(n-k)\ge\vert{n_2-k_2}\vert,
\end{equation}
\begin{equation}
    \label{d_21(k-m)=>|k_2-m_2|-1:3.1.2}
    d_{21}(k-m)\ge\vert{k_2-m_2}\vert-1,
\end{equation}
\begin{equation}
    \label{d_12(k-m)=>|k_2-m_2|:3.1.2}
    d_{12}(k-m)\ge\vert{k_2-m_2}\vert,
\end{equation}
since $n_2-k_2>0$ and $k_2-m_2<0$; also,
\begin{equation}
    \label{d(n-k)=>|n_2-k_2|:3.1.2}
    d(n-k)\ge\vert{n_2-k_2}\vert,
\end{equation}
\begin{equation}
    \label{d(k-m)=>|k_2-m_2|:3.1.2}
    d(k-m)\ge\vert{k_2-m_2}\vert.
\end{equation}
Then, by using \eqref{asymptotic_of_alpha_zeta_-/sqrt_8z^2+1},
\eqref{d_21(n-k)=>|n_2-k_2|:3.1.2}, \eqref{d_21(k-m)=>|k_2-m_2|-1:3.1.2},
\eqref{d_12(k-m)=>|k_2-m_2|:3.1.2}, \eqref{d(n-k)=>|n_2-k_2|:3.1.2}, and
\eqref{d(k-m)=>|k_2-m_2|:3.1.2}, we have
${\cal{O}}(N^{-2\vert{n_2-k_2}\vert-2\vert{k_2-m_2}\vert-2})$ as the
asymptotic behavior of
\eqref{estimate_(bracket_q_bracket_q_eta)_22}. Moreover, by using
\eqref{asymptotic_e^{zeta_+}} and \eqref{asymptotic_e^{zeta_-}}, we have
$I\hspace{-.1em}I\hspace{-.1em}I_{1,2}\sim{\cal{O}}(N^{2(n_2+m_2)-2\vert{n_2-k_2}\vert-2\vert{k_2-m_2}\vert-2})$. The
estimate \eqref{2(n_2-m_2)-2|n_2-k_2|-2|k_2-m_2|=<4p} implies
$I\hspace{-.1em}I\hspace{-.1em}I_{1,2}\sim{\cal{O}}(N^{4p-2})$.

\subsection{Proof of Lemma \ref{lemma_asymptotic_of_IV_and_V}}

\begin{flushleft}
    {\underline{\it{Asymptotic behavior of $I\hspace{-.1em}V$}}}
\end{flushleft}

We split $I\hspace{-.1em}V$ into four parts:
$I\hspace{-.1em}V=I\hspace{-.1em}V_1+I\hspace{-.1em}V_2+I\hspace{-.1em}V_3+I\hspace{-.1em}V_4$,
where
\begin{equation*}
    I\hspace{-.1em}V_1=\sum_{n_2\ge{p+1}}\sum_{m_2\ge{p+1}},
    I\hspace{-.1em}V_2=\sum_{n_2\ge{p+1}}\sum_{m_2\le{p}},
    I\hspace{-.1em}V_3=\sum_{n_2\le{p}}\sum_{m_2\ge{p+1}},
    I\hspace{-.1em}V_4=\sum_{n_2\le{p}}\sum_{m_2\le{p}}.
\end{equation*}

About $I\hspace{-.1em}V_4$, it is a sum of the following terms:
\begin{equation}
    \label{e^in_zeta_+q_1(n)e^imzeta_-<P(n)|R(z)|P(m)>q(m)eta^(0)(z,theta)}
    e^{in\zeta_+(z,\theta)}\hat{q}_1(n)e^{-im\zeta_-(z,\theta^{\prime})}\bigl(\langle\hat{P}(n)\vert\hat{R}(z)\vert\hat{P}(m)\rangle\hat{q}(m)\hat{\eta}^{(0)}(z,\theta^{\prime})\bigr)_{11},
\end{equation}
where $n_2\le{p}$ and $m_2\le{p}$. Then, by using
\eqref{asymptotic_e^{zeta_+}}, \eqref{asymptotic_e^{zeta_-}}, and
\eqref{asymptotic_of_<P(n)|R(z)|P(m)>q(m)eta(z,theta)}, we have
$I\hspace{-.1em}V_4\sim{\cal{O}}(N^{2(n_2+m_2)-1})={\cal{O}}(N^{4p-1})$.

About $I\hspace{-.1em}V_3$, it is a sum of
\eqref{e^in_zeta_+q_1(n)e^imzeta_-<P(n)|R(z)|P(m)>q(m)eta^(0)(z,theta)}'s,
where $n_2\le{p}$ and $m_2\ge{p+1}$. By using
\eqref{estimate_for_<P(n)|R(z)|P(m)>q(m)_eta}, we have
\begin{align}
    \label{estimate_for_(<P(n)|R(z)|P(m)>q(m)_eta)_11}
    \begin{split}
        &\bigl(\langle\hat{P}(n)\vert\hat{R}(z)\vert\hat{P}(m)\rangle\hat{q}(m)\hat{\eta}^{(0)}(z,\theta^{\prime})\bigr)_{11}\\
        &={\cal{O}}(\langle{z}\rangle^{-2d(n-m)-1})\hat{q}_1(m)+{\cal{O}}(\langle{z}\rangle^{-2d_{12}(n-m)-2})\hat{q}_2(m)\frac{\overline{\alpha}(\zeta_-(z,\theta^{\prime}))}{\sqrt{8z^2+1}}.
    \end{split}
\end{align}
From Lemma \ref{lemma_d_ij(n)>=|n_2|}, we have
\begin{equation}
    \label{d_12(n-m)=>|n_2-m_2|:4.3}
    d_{12}(n-m)\ge\vert{n_2-m_2}\vert,
\end{equation}    
since $n_2-m_2<0$; also,
\begin{equation}
    \label{d(n-m)=>|n_2-m_2|:4.3}
    d(n-m)\ge\vert{n_2-m_2}\vert.
\end{equation}
Then, by using \eqref{asymptotic_of_adjoint_alpha_zeta_-/sqrt_8z^2+1},
\eqref{d_12(n-m)=>|n_2-m_2|:4.3}, and \eqref{d(n-m)=>|n_2-m_2|:4.3}, we have
${\cal{O}}(N^{-2\vert{n_2-m_2}\vert-1})$ as the asymptotic behavior of
\eqref{estimate_for_(<P(n)|R(z)|P(m)>q(m)_eta)_11}. Moreover, by using
\eqref{asymptotic_e^{zeta_+}} and \eqref{asymptotic_e^{zeta_-}}, we have
$I\hspace{-.1em}V_3\sim{\cal{O}}(N^{2(n_2+m_2)-2\vert{n_2-m_2}\vert-1})$. The estimate
\eqref{2(n_2+m_2)-2|n_2-m_2|=<4p} implies
$I\hspace{-.1em}V_3\sim{\cal{O}}(N^{4p-1})$.

About $I\hspace{-.1em}V_2$, by using the resolvent equation
\begin{equation}
    \label{R(z)=R_>p,>p-1(z)-R_>p,>p-1(z)qR(z)}
    \hat{R}(z)=\hat{R}_{(>p,>p-1)}(z)-\hat{R}_{(>p,>p-1)}(z)\hat{q}_{(\le{p},\le{p-1})}\hat{R}(z),
\end{equation}
we split it into five parts:
$I\hspace{-.1em}V_2=I\hspace{-.1em}V_{2,1}+I\hspace{-.1em}V_{2,2,1}+I\hspace{-.1em}V_{2,2,2,1}-I\hspace{-.1em}V_{2,2,2,2,1}-I\hspace{-.1em}V_{2,2,2,2,2}$,
where
\begin{align*}
    I\hspace{-.1em}V_{2,1}=&\sum_{n_2\ge{p+1}}e^{in\zeta_+(z,\theta)}\hat{q}_1(n)\sum_{m_2\le{p-1}}e^{-im\zeta_-(z,\theta^{\prime})}\\
    &\cdot\bigl(\langle\hat{P}(n)\vert\hat{R}(z)\vert\hat{P}(m)\rangle\hat{q}(m)\hat{\eta}^{(0)}(z,\theta^{\prime})\bigr)_{11},
\end{align*}
\begin{align*}
    I\hspace{-.1em}V_{2,2,1}=&\sum_{n_2\ge{p+1}}e^{in\zeta_+(z,\theta)}\hat{q}_1(n)\sum_{m_2=p}e^{-im\zeta_-(z,\theta^{\prime})}\\
    &\cdot\bigl(\langle\hat{P}(n)\vert\hat{R}(z)\vert\hat{P}(m)\rangle
    \begin{pmatrix}
        \hat{q}_1(m)&0\\
        0&0
    \end{pmatrix}
    \hat{\eta}^{(0)}(z,\theta^{\prime})\bigr)_{11},
\end{align*}
\begin{align*}
    I\hspace{-.1em}V_{2,2,2,1}=&\sum_{n_2\ge{p+1}}e^{in\zeta_+(z,\theta)}\hat{q_1}(n)\sum_{m_2=p}e^{-im\zeta_-(z,\theta^{\prime})}\\
    &\cdot\bigl(\langle\hat{P}(n)\vert\hat{R}_{(>p,>p-1)}(z)\vert\hat{P}(m)\rangle
    \begin{pmatrix}
        0&0\\
        0&\hat{q}_2(m)
    \end{pmatrix}
    \hat{\eta}^{(0)}(z,\theta^{\prime})\bigr)_{11},
\end{align*}
\begin{align*}
    I\hspace{-.1em}V_{2,2,2,2,1}=&\sum_{n_2\ge{p+1}}e^{in\zeta_+(z,\theta)}\hat{q}_1(n)\sum_{m_2=p}e^{-im\zeta_-(z,\theta^{\prime})}\\
    &\cdot\bigl(\langle\hat{P}(n)\vert\hat{R}_{(>p,>p-1)}(z)\hat{q}_{\le{p-1}}\hat{R}(z)\vert\hat{P}(m)\rangle
    \begin{pmatrix}
        0&0\\
        0&\hat{q}_2(m)
    \end{pmatrix}
    \hat{\eta}^{(0)}(z,\theta^{\prime})\bigr)_{11},
\end{align*}
\begin{align*}
    I\hspace{-.1em}V_{2,2,2,2,2}=&\sum_{n_2\ge{p+1}}e^{in\zeta_+(z,\theta)}\hat{q}_1(n)\sum_{m_2=p}e^{-im\zeta_-(z,\theta^{\prime})}\\
    &\cdot\sum_{k_2=p}\bigl(\langle\hat{P}(n)\vert\hat{R}_{(>p,>p-1)}(z)\hat{P}(k)
    \begin{pmatrix}
        \hat{q}_1(k)&0\\
        0&0
    \end{pmatrix}
    \hat{P}(k)\hat{R}(z)\vert\hat{P}(m)\rangle\\
    &\quad\quad\cdot
    \begin{pmatrix}
        0&0\\
        0&\hat{q}_2(m)
    \end{pmatrix}
    \hat{\eta}^{(0)}(z,\theta^{\prime})\bigr)_{11}.
\end{align*}

$I\hspace{-.1em}V_{2,1}$ is a sum of
\eqref{e^in_zeta_+q_1(n)e^imzeta_-<P(n)|R(z)|P(m)>q(m)eta^(0)(z,theta)}'s,
where $n_2\ge{p+1}$ and $m_2\le{p-1}$. From Lemma \ref{lemma_d_ij(n)>=|n_2|},
we have
\begin{equation}
    \label{d_12(n-m)=>|n_2-m_2|-1:4.2.1}
    d_{12}(n-m)\ge\vert{n_2-m_2}\vert-1,
\end{equation}
since $n_2-m_2>0$. Then, by using
\eqref{asymptotic_of_adjoint_alpha_zeta_-/sqrt_8z^2+1} and
\eqref{d_12(n-m)=>|n_2-m_2|-1:4.2.1}, we have
$I\hspace{-.1em}V_{2,1}\sim{\cal{O}}(N^{4p-3})$ in the same way as
$I\hspace{-.1em}V_3$.

$I\hspace{-.1em}V_{2,2,1}$ is also a sum of
\eqref{e^in_zeta_+q_1(n)e^imzeta_-<P(n)|R(z)|P(m)>q(m)eta^(0)(z,theta)}'s,
where $n_2\ge{p+1}$ and $m_2=p$. Then, by noticing that $\hat{q}_2(m)=0$ for
$m_2=p$, we have $I\hspace{0.1em}V_{2,2,1}\sim{\cal{O}}(N^{4p-1})$ in the same
way as above.

We know $\hat{R}_{(>p,>p-1)}(z)$, since we have already computed $\hat{q}(n)$
for $n_2\ge{p+1}$ and $\hat{q}_2(n)$ for $n_2=p$, which implies that
$I\hspace{-.1em}V_{2,2,2,1}$ is a known term.

$I\hspace{-.1em}V_{2,2,2,2,1}$ is a linear combination of the
following terms:
\begin{align}
    \label{e_e_(<P(n)|R_{>p,>p-1}(z)|P(k)>q(k)<P(k)|R(z)|P(m)>q(m)_eta)_11}
    \begin{split}
        &e^{in\zeta_+(z,\theta)}e^{-im\zeta_-(z,\theta^{\prime})}\\
        &\cdot\bigl(\langle\hat{P}(n)\vert\hat{R}_{(>p,>p-1)}(z)\vert\hat{P}(k)\rangle\hat{q}(k)\langle\hat{P}(k)\vert\hat{R}(z)\vert\hat{P}(m)\rangle
        \begin{pmatrix}
            0&0\\
            0&1
        \end{pmatrix}
        \hat{\eta}^{(0)}(z,\theta^{\prime})\bigr)_{11},
    \end{split}
\end{align}
where $n_2\ge{p+1}$, $k_2\le{p-1}$, and $m_2=p$. By using
\eqref{estimate_for_<P(n)|R(z)|P(m)>q(m)_eta}, we have
\begin{align}
    \label{<P(n)|R_(>p,>p-1)(z)|P(k)>q(k)<P(k)|R(z)|P(m)>(0,0;0,1)_eta}
    \begin{split}
        &\bigl(\langle\hat{P}(n)\vert\hat{R}_{(>p,>p-1)}(z)\vert\hat{P}(k)\rangle\hat{q}(k)\langle\hat{P}(k)\vert\hat{R}(z)\vert\hat{P}(m)\rangle
        \begin{pmatrix}
            0&0\\
            0&1
        \end{pmatrix}
        \hat{\eta}^{(0)}(z,\theta^{\prime})\bigr)_{11}\\
        &={\cal{O}}(\langle{z}\rangle^{-2d(n-k)-1})\hat{q}_1(k){\cal{O}}(\langle{z}\rangle^{-2d_{12}(k-m)-2})\frac{\overline{\alpha}(\zeta_-(z,\theta^{\prime}))}{\sqrt{8z^2+1}}\\
        &\quad+{\cal{O}}(\langle{z}\rangle^{-2d_{12}(n-k)-2})\hat{q}_2(k){\cal{O}}(\langle{z}\rangle^{-2d(k-m)-1})\frac{\overline{\alpha}(\zeta_-(z,\theta^{\prime}))}{\sqrt{8z^2+1}}.
    \end{split}
\end{align}
From Lemma \ref{lemma_d_ij(n)>=|n_2|}, we have
\begin{equation}
    \label{d_12(k-m)=>|k_2-m_2|:4.2.2.2.2.1}
    d_{12}(k-m)\ge\vert{k_2-m_2}\vert,
\end{equation}
\begin{equation}
    \label{d_12(n-k)=>|n_2-k_2|-1:4.2.2.2.2.1}
    d_{12}(n-k)\ge\vert{n_2-k_2}\vert-1,
\end{equation}
since $k_2-m_2<0$ and $n_2-k_2>0$; also,
\begin{equation}
    \label{d(n-k)=>|n_2-k_2|:4.2.2.2.2.1}
    d(n-k)\ge\vert{n_2-k_2}\vert.
\end{equation}
Then, from Lemmas \ref{lemma_d_triangle_ineq}, \ref{lemma_d_ij=<d_ii+d_ij},
and \ref{lemma_d_ii=<d_ij+d_ji+1}, by using
\eqref{asymptotic_of_adjoint_alpha_zeta_-/sqrt_8z^2+1},
\eqref{d_12(k-m)=>|k_2-m_2|:4.2.2.2.2.1},
\eqref{d_12(n-k)=>|n_2-k_2|-1:4.2.2.2.2.1}, and
\eqref{d(n-k)=>|n_2-k_2|:4.2.2.2.2.1}, we have
${\cal{O}}(N^{-2\vert{n_2-k_2}\vert-2\vert{k_2-m_2}\vert})$ as the asymptotic
behavior of
\eqref{<P(n)|R_(>p,>p-1)(z)|P(k)>q(k)<P(k)|R(z)|P(m)>(0,0;0,1)_eta}. Moreover,
by using \eqref{asymptotic_e^{zeta_+}} and \eqref{asymptotic_e^{zeta_-}}, we
have
$I\hspace{-.1em}V_{2,2,2,2,1}\sim{\cal{O}}(N^{2(n_2+m_2)-2\vert{n_2-k_2}\vert-2\vert{k_2-m_2}\vert})$. The
estimate \eqref{2(n_2-m_2)-2|n_2-k_2|-2|k_2-m_2|=<4p} implies
$I\hspace{-.1em}V_{2,2,2,2,1}\sim{\cal{O}}(N^{4p-4})$.

$I\hspace{-.1em}V_{2,2,2,2,2}$ is a linear combination of
\eqref{e_e_(<P(n)|R_{>p,>p-1}(z)|P(k)>q(k)<P(k)|R(z)|P(m)>q(m)_eta)_11}'s,
where $n_2\ge{p+1}$, $k_2=p$, and $m_2=p$. Then, by noticing $\hat{q}_2(k)=0$
for $k_2=p$, we have $I\hspace{-.1em}V_{2,2,2,2,2}\sim{\cal{O}}(N^{4p-2})$ in
the same way as $I\hspace{-.1em}V_{2,2,2,2,1}$.

About $I\hspace{-.1em}V_1$, by using the resolvent equation
\eqref{R(z)=R_>p,>p-1(z)-R_>p,>p-1(z)qR(z)}, we split it into three parts:
$I\hspace{-.1em}V_1=I\hspace{-.1em}V_{1,1}-I\hspace{-.1em}V_{1,2,1}-I\hspace{-.1em}V_{1,2,2}$,
where
\begin{align*}
    I\hspace{-.1em}V_{1,1}=&\sum_{n_2\ge{p+1}}e^{in\zeta_+(z,\theta)}\hat{q}_1(n)\sum_{m_2\ge{p+1}}e^{-im\zeta_-(z,\theta^{\prime})}\\
    &\cdot\bigl(\langle\hat{P}(n)\vert\hat{R}_{(>p,>p-1)}(z)\vert\hat{P}(m)\rangle\hat{q}(m)\hat{\eta}^{(0)}(z,\theta^{\prime})\bigr)_{11},
\end{align*}
\begin{align*}
    I\hspace{-.1em}V_{1,2,1}=&\sum_{n_2\ge{p+1}}e^{in\zeta_+(z,\theta)}\hat{q}_1(n)\sum_{m_2\ge{p+1}}e^{-im\zeta_-(z,\theta^{\prime})}\\
    &\cdot\bigl(\langle\hat{P}(n)\vert\hat{R}_{(>p,>p-1)}(z)\hat{q}_{\le{p-1}}\hat{R}(z)\vert\hat{P}(m)\rangle\hat{q}(m)\hat{\eta}^{(0)}(z,\theta^{\prime})\bigr)_{11},
\end{align*}
\begin{align*}
    I\hspace{-.1em}V_{1,2,2}=&\sum_{n_2\ge{p+1}}e^{in\zeta_+(z,\theta)}\hat{q}_1(n)\sum_{m_2\ge{p+1}}e^{-im\zeta_-(z,\theta^{\prime})}\\
    &\cdot\sum_{k_2=p}\bigl(\langle\hat{P}(n)\vert\hat{R}_{(>p,>p-1)}(z)\hat{P}(k)
    \begin{pmatrix}
        \hat{q}_1(k)&0\\
        0&0
    \end{pmatrix}
    \hat{P}(k)\hat{R}(z)\vert\hat{P}(m)\rangle\\
    &\quad\quad\cdot\hat{q}(m)\hat{\eta}^{(0)}(z,\theta^{\prime})\bigr)_{11}.
\end{align*}

$I\hspace{-.1em}V_{1,1}$ is a known term as before.

$I\hspace{-.1em}V_{1,2,1}$ is a linear combination of
\eqref{e_e_(<P(n)|R_{>p,>p-1}(z)|P(k)>q(k)<P(k)|R(z)|P(m)>q(m)_eta)_11}'s,
where $n_2\ge{p+1}$, $k_2\le{p-1}$, and $m_2\ge{p+1}$. Then we have
$I\hspace{-.1em}V_{1,2,1}\sim{\cal{O}}(N^{4p-4})$ in the same way as
$I\hspace{-.1em}V_{2,2,2,2,1}$.

$I\hspace{-.1em}V_{1,2,2}$ is also a linear combination of
\eqref{e_e_(<P(n)|R_{>p,>p-1}(z)|P(k)>q(k)<P(k)|R(z)|P(m)>q(m)_eta)_11}'s,
where $n_2\ge{p+1}$, $k_2=p$, and $m_2\ge{p+1}$. Then, by noticing
$\hat{q}_2(k)=0$ for $k_2=p$, we have
$I\hspace{-.1em}V_{1,2,2}\sim{\cal{O}}(N^{4p-2})$ in the same way as above.

\begin{flushleft}
    {\underline{\it{Asymptotic behavior of $V$}}}
\end{flushleft}

We split $V$ into four parts: $V=V_1+V_2+V_3+V_4$, where
\begin{equation*}
    V_1=\sum_{n_2\ge{p}}\sum_{m_2\ge{p+1}},
    V_2=\sum_{n_2\ge{p}}\sum_{m_2\le{p}},
    V_3=\sum_{n_2\le{p-1}}\sum_{m_2\ge{p+1}},
    V_4=\sum_{n_2\le{p-1}}\sum_{m_2\le{p}}.
\end{equation*}

About $V_4$, it is a linear combination of the following terms:
\begin{equation}
    \label{alpha+_e_e_<P(n)|R(z)|P(m)>q_eta_21}
    \frac{\alpha(\zeta_+(z,\theta))}{\sqrt{8z^2+1}}e^{in\zeta_+(z,\theta)}e^{-im\zeta_-(z,\theta^{\prime})}\bigl(\langle\hat{P}(n)\vert\hat{R}(z)\vert\hat{P}(m)\rangle\hat{q}(m)\hat{\eta}^{(0)}(z,\theta^{\prime})\bigr)_{21},
\end{equation}
where $n_2\le{p-1}$ and $m_2\le{p}$. Then, by using
\eqref{asymptotic_of_alpha_zeta_+/sqrt_8z^2+1},
\eqref{asymptotic_e^{zeta_+}}, \eqref{asymptotic_e^{zeta_-}}, and
\eqref{asymptotic_of_<P(n)|R(z)|P(m)>q(m)eta(z,theta)}, we have
$V_4\sim{\cal{O}}(N^{2(n_2+m_2)+1})\sim{\cal{O}}(N^{4p-1})$.

About $V_3$, it is a linear combination of
\eqref{alpha+_e_e_<P(n)|R(z)|P(m)>q_eta_21}'s, where $n_2\le{p-1}$ and
$m_2\ge{p+1}$. By using \eqref{estimate_for_<P(n)|R(z)|P(m)>q(m)_eta}, we have
\begin{align}
    \label{estimate_for_(<P(n)|R(z)|P(m)>q(m)_eta)_21}
    \begin{split}
        &\bigl(\langle\hat{P}(n)\vert\hat{R}(z)\vert\hat{P}(m)\rangle\hat{q}(m)\hat{\eta}^{(0)}(z,\theta^{\prime})\bigr)_{21}\\
        &\sim{\cal{O}}(\langle{z}\rangle^{-2d_{21}(n-m)-2})\hat{q}_1(m)+{\cal{O}}(\langle{z}\rangle^{-2d(n-m)-1})\hat{q}_2(m)\frac{\overline{\alpha}(\zeta_-(z,\theta^{\prime}))}{\sqrt{8z^2+1}}.
    \end{split}
\end{align}
From Lemma \ref{lemma_d_ij(n)>=|n_2|}, we have
\begin{equation}
    \label{d_21(n-m)=>|n_2-m_2|-1:5.3}
    d_{21}(n-m)\ge\vert{n_2-m_2}\vert-1,
\end{equation}
since $n_2-m_2<0$; also,
\begin{equation}
    \label{d(n-m)=>|n_2-m_2|:5.3}
    d(n-m)\ge\vert{n_2-m_2}\vert.
\end{equation}
Then, by using \eqref{asymptotic_of_adjoint_alpha_zeta_-/sqrt_8z^2+1},
\eqref{d_21(n-m)=>|n_2-m_2|-1:5.3}, and \eqref{d(n-m)=>|n_2-m_2|:5.3}, we have
${\cal{O}}(N^{-2\vert{n_2-m_2}\vert})$ as the asymptotic behavior of
\eqref{estimate_for_(<P(n)|R(z)|P(m)>q(m)_eta)_21}. Moreover, by using
\eqref{asymptotic_e^{zeta_+}}, \eqref{asymptotic_e^{zeta_-}}, and
\eqref{asymptotic_of_alpha_zeta_+/sqrt_8z^2+1}, we have
$V_3\sim{\cal{O}}(N^{2(n_2+m_2)-2\vert{n_2-m_2}\vert+1})$. The estimate
\eqref{2(n_2+m_2)-2|n_2-m_2|=<4p} implies $V_3\sim{\cal{O}}(N^{4p-3})$.

About $V_2$, by using the resolvent equation
\eqref{R(z)=R_>p,>p-1(z)-R_>p,>p-1(z)qR(z)}, we split it into five parts:
$V_2=V_{2,1}+V_{2,2,1}+V_{2,2,2,1}-V_{2,2,2,2,1}-V_{2,2,2,2,2}$, where
\begin{align*}
    V_{2,1}=&\frac{\alpha(\zeta_+(z,\theta))}{\sqrt{8z^2+1}}\sum_{n_2\ge{p}}e^{in\zeta_+(z,\theta)}\hat{q}_2(n)\sum_{m_2\le{p-1}}e^{-m\zeta_-(z,\theta^{\prime})}\\
    &\cdot\bigl(\langle\hat{P}(n)\vert\hat{R}(z)\vert\hat{P}(m)\rangle\hat{q}(m)\hat{\eta}^{(0)}(z,\theta^{\prime})\bigr)_{21},
\end{align*}
\begin{align*}
    V_{2,2,1}=&\frac{\alpha(\zeta_+(z,\theta))}{\sqrt{8z^2+1}}\sum_{n_2\ge{p}}e^{in\zeta_+(z,\theta)}\hat{q}_2(n)\sum_{m_2=p}e^{-im\zeta_-(z,\theta^{\prime})}\\
    &\cdot\bigl(\langle\hat{P}(n)\vert\hat{R}(z)\vert\hat{P}(m)\rangle
    \begin{pmatrix}
        \hat{q}_1(m)&0\\
        0&0
    \end{pmatrix}
    \hat{\eta}^{(0)}(z,\theta^{\prime})\bigr)_{21},
\end{align*}
\begin{align*}
    V_{2,2,2,1}=&\frac{\alpha(\zeta_+(z,\theta))}{\sqrt{8z^2+1}}\sum_{n_2\ge{p}}e^{in\zeta_+(z,\theta)}\hat{q}_2(n)\sum_{m_2=p}e^{-im\zeta_-(z,\theta^{\prime})}\\
    &\cdot\bigl(\langle\hat{P}(n)\vert\hat{R}_{(>p,>p-1)}(z)\vert\hat{P}(m)\rangle
    \begin{pmatrix}
        0&0\\
        0&\hat{q}_2(m)
    \end{pmatrix}
    \hat{\eta}^{(0)}(z,\theta^{\prime})\bigr)_{21},
\end{align*}
\begin{align*}
    V_{2,2,2,2,1}=&\frac{\alpha(\zeta_+(z,\theta))}{\sqrt{8z^2+1}}\sum_{n_2\ge{p}}e^{in\zeta_+(z,\theta)}\hat{q}_2(n)\sum_{m_2=p}e^{-im\zeta_-(z,\theta^{\prime})}\\
    &\cdot\bigl(\langle{P}(n)\vert\hat{R}_{(>p,>p-1)}\hat{q}_{\le{}p-1}\hat{R}(z)\vert\hat{P}(m)\rangle
    \begin{pmatrix}
        0&0\\
        0&\hat{q}_2(m)
    \end{pmatrix}
    \hat{\eta}^{(0)}(z,\theta^{\prime})\bigr)_{21},
\end{align*}
\begin{align*}
    V_{2,2,2,2,2}=&\frac{\alpha(\zeta_+(z,\theta))}{\sqrt{8z^2+1}}\sum_{n_2\ge{p}}e^{in\zeta_+(z,\theta)}\hat{q}_2(n)\sum_{m_2=p}e^{-im\zeta_-(z,\theta^{\prime})}\\
    &\cdot\sum_{k_2=p}\bigl(\langle\hat{P}(n)\vert\hat{R}_{(>p,>p-1)}(z)\hat{P}(k)
    \begin{pmatrix}
        \hat{q}_1(k)&0\\
        0&0
    \end{pmatrix}
    \hat{P}(k)\hat{R}(z)\vert\hat{P}(m)\rangle\\
    &\quad\quad\ \cdot
    \begin{pmatrix}
        0&0\\
        0&\hat{q}_2(m)
    \end{pmatrix}
    \hat{\eta}^{(0)}(z,\theta^{\prime})\bigr)_{21}.
\end{align*}

$V_{2,1}$ is a linear combination of
\eqref{alpha+_e_e_<P(n)|R(z)|P(m)>q_eta_21}'s, where $n_2\ge{p}$ and
$m_2\le{p-1}$. From Lemma \ref{lemma_d_ij(n)>=|n_2|}, we have
\begin{equation}
    \label{d_21(n-m)=>|n_2-m_2|:5.2.1}
    d_{21}(n-m)\ge\vert{n_2-m_2}\vert,
\end{equation}
since $n_2-m_2>0$; also,
\begin{equation}
    \label{d(n-m)=>|n_2-m_2|:5.2.1}
    d(n-m)\ge\vert{n_2-m_2}\vert.
\end{equation}
Then, from \eqref{d_21(n-m)=>|n_2-m_2|:5.2.1} and
\eqref{d(n-m)=>|n_2-m_2|:5.2.1}, we have
$V_{2,1}\sim{\cal{O}}(N^{2(n_2+m_2)-2\vert{n_2-m_2}\vert+1})$ in the same say
as $V_3$. The estimate \eqref{2(n_2+m_2)-2|n_2-m_2|=<4p} implies
$V_{2,1}\sim{\cal{O}}(N^{4p-3})$.

$V_{2,2,1}$ is a sum of \eqref{alpha+_e_e_<P(n)|R(z)|P(m)>q_eta_21}'s, where
$n_2\ge{p}$ and $m_2=p$. Then, by noticing that $\hat{q}_2(m)=0$ for $m_2=p$,
we have $V_{2,2,1}\sim{\cal{O}}(N^{4p-1})$ in the same way as above.

$V_{2,2,2,1}$ is a known term as before.

$V_{2,2,2,2,1}$ is a linear combination of the following terms:
\begin{align}
    \label{alpha/z_e_e_(<P(n)|R_{>p,>p-1}(z)|P(k)>q(k)<P(k)|R(z)|P(m)>0,0,0,1_eta)_21}
    \begin{split}
        &\frac{\alpha(\zeta_+(z,\theta))}{\sqrt{8z^2+1}}e^{in\zeta_+(z,\theta)}e^{-im\zeta_-(z,\theta^{\prime})}\bigl(\langle\hat{P}(n)\vert\hat{R}_{(>p,>p-1)}(z)\vert\hat{P}(k)\rangle\\
        &\cdot\hat{q}(k)\langle\hat{P}(k)\vert\hat{R}(z)\vert\hat{P}(m)\rangle
        \begin{pmatrix}
            0&0\\
            0&1
        \end{pmatrix}
        \hat{\eta}^{(0)}(z,\theta^{\prime})\bigr)_{21},
    \end{split}
\end{align}
where $n_2\ge{p}$, $k_2\le{p-1}$, and $m_2=p$. By using
\eqref{resolvent_estimate_for_R(z)}, \eqref{q_n}, and
\eqref{estimate_for_<P(n)|R(z)|P(m)>q(m)_eta}, we have
\begin{align}
    \label{asymptotic_<P(n)|R_>p,>p-1(z)|P(k)>q(k)<P(k)|R(z)|P(m)>0,0,0,1_eta}
    \begin{split}
        &\bigl(\langle\hat{P}(n)\vert\hat{R}_{(>p,>p-1)}(z)\vert\hat{P}(k)\rangle\hat{q}(k)\langle\hat{P}(k)\vert\hat{R}(z)\vert\hat{P}(m)\rangle
        \begin{pmatrix}
            0&0\\
            0&1
        \end{pmatrix}
        \hat{\eta}^{(0)}(z,\theta^{\prime})\bigr)_{21}\\
        &\sim{\cal{O}}(\langle{z}\rangle^{-2d_{21}(n-k)-2})\hat{q}_1(k){\cal{O}}(\langle{z}\rangle^{-2d_{12}(k-m)-2})\frac{\overline{\alpha}(\zeta_-(z,\theta^{\prime}))}{\sqrt{8z^2+1}}\\
        &\quad+{\cal{O}}(\langle{z}\rangle^{-2d(n-k)-1})\hat{q}_2(k){\cal{O}}(\langle{z}\rangle^{-2d(k-m)-1})\frac{\overline{\alpha}(\zeta_-(z,\theta^{\prime}))}{\sqrt{8z^2+1}}.
    \end{split}
\end{align}
From Lemma \ref{lemma_d_ij(n)>=|n_2|}, we have
\begin{equation}
    \label{d_21(n-k)=>+n_2-k_2|:5.2.2.2.2.1}
    d_{21}(n-k)\ge\vert{n_2-k_2}\vert,
\end{equation}
\begin{equation}
    \label{d_12(k-m)=>|k_2-m_2|:5.2.2.2.2.1}
    d_{12}(k-m)\ge\vert{k_2-m_2}\vert,
\end{equation}
since $n_2-k_2>0$ and $k_2-m_2<0$; also,
\begin{equation}
    \label{d(n-k)=>|n_2-k_2|:5.2.2.2.2.1}
    d(n-k)\ge\vert{n_2-k_2}\vert,
\end{equation}
\begin{equation}
    \label{d(k-m)=>|k_2-m_2|:5.2.2.2.2.1}
    d(k-m)\ge\vert{k_2-m_2}\vert.
\end{equation}
Then, by using \eqref{asymptotic_of_adjoint_alpha_zeta_-/sqrt_8z^2+1},
\eqref{d_21(n-k)=>+n_2-k_2|:5.2.2.2.2.1},
\eqref{d_12(k-m)=>|k_2-m_2|:5.2.2.2.2.1},
\eqref{d(n-k)=>|n_2-k_2|:5.2.2.2.2.1}, and
\eqref{d(k-m)=>|k_2-m_2|:5.2.2.2.2.1}, we have
${\cal{O}}(N^{-2\vert{n_2-k_2}\vert-2\vert{k_2-m_2}\vert-1})$ as the asymptotic
behavior of
\eqref{asymptotic_<P(n)|R_>p,>p-1(z)|P(k)>q(k)<P(k)|R(z)|P(m)>0,0,0,1_eta}. Moreover,
by using \eqref{asymptotic_e^{zeta_+}}, \eqref{asymptotic_e^{zeta_-}}, and
\eqref{asymptotic_of_alpha_zeta_+/sqrt_8z^2+1}, we have
$V_{2,2,2,2,1}\sim{\cal{O}}(N^{2(n_2+m_2)-2\vert{n_2-k_2}\vert-2\vert{k_2-m_2}\vert})$.
The estimate \eqref{2(n_2-m_2)-2|n_2-k_2|-2|k_2-m_2|=<4p} implies
$V_{2,2,2,2,1}\sim{\cal{O}}(N^{4p-4})$.

$V_{2,2,2,2,2}$ is a linear combination of 
\eqref{alpha/z_e_e_(<P(n)|R_{>p,>p-1}(z)|P(k)>q(k)<P(k)|R(z)|P(m)>0,0,0,1_eta)_21}'s,
where $n_2\ge{p}$, $k_2=p$, and $m_2=p$. Then, by noticing $\hat{q}_2(k)=0$ for
$k_2=p$, we have $V_{2,2,2,2,2}\sim{\cal{O}}(N^{4p-2})$ in the same as way as
$V_{2,2,2,2,1}$.

About $V_1$, by using the resolvent equation
\eqref{R(z)=R_>p,>p-1(z)-R_>p,>p-1(z)qR(z)}, we split it into three parts:
$V_1=V_{1,1}-V_{1,2,1}-V_{1,2,2}$, where 
\begin{align*}
    V_{1,1}=&\frac{\alpha(\zeta_+(z,\theta))}{\sqrt{8z^2+1}}\sum_{n_2\ge{p}}e^{in\zeta_+(z,\theta)}\hat{q}_2(n)\sum_{m_2\ge{p+1}}e^{-im\zeta_-(z,\theta^{\prime})}\\
    &\cdot\bigl(\langle\hat{P}(n)\vert\hat{R}_{(>p,>p-1)}(z)\vert\hat{P}(m)\rangle\hat{q}(m)\hat{\eta}^{(0)}(z,\theta^{\prime})\bigr)_{21},
\end{align*}
\begin{align*}
    V_{1,2,1}=&\frac{\alpha(\zeta_+(z,\theta))}{\sqrt{8z^2+1}}\sum_{n_2\ge{p}}e^{in\zeta_+(z,\theta)}\hat{q}_2(n)\sum_{m_2\ge{p+1}}e^{-im\zeta_-(z,\theta^{\prime})}\\
    &\cdot\bigl(\langle\hat{P}(n)\vert\hat{R}_{(>p,>p-1)}(z)\hat{q}_{\le{p-1}}\hat{R}(z)\vert\hat{P}(m)\rangle\hat{q}(m)\hat{\eta}^{(0)}(z,\theta^{\prime})\bigr)_{21},
\end{align*}
\begin{align*}
    V_{1,2,2}=&\frac{\alpha(\zeta_+(z,\theta))}{\sqrt{8z^2+1}}\sum_{n_2\ge{p}}e^{in\zeta_+(z,\theta)}\hat{q}_2(n)\sum_{m_2\ge{p+1}}e^{-im\zeta_-(z,\theta^{\prime})}\\
    &\cdot\sum_{k_2=p}\bigl(\langle\hat{P}(n)\vert\hat{R}_{(>p,>p-1)}(z)\hat{P}(k)
    \begin{pmatrix}
        \hat{q}_1(k)&0\\
        0&0
    \end{pmatrix}
    \hat{P}(k)\hat{R}(z)\vert\hat{P}(m)\rangle\\
    &\quad\quad\cdot\hat{q}(m)\hat{\eta}^{(0)}(z,\theta^{\prime})\bigr)_{21}.
\end{align*}

$V_{1,1}$ is a known term as before.

$V_{1,2,1}$ is a sum of the following terms:
\begin{align}
    \label{asymptotic_alpha_e_e(<P(n)|R_>p,>p-1(z)|P(k)>q(k)<P(k)|R(z)|P(m)>q(m)_eta)_21}
    \begin{split}
        &\frac{\alpha(\zeta_+(z,\theta))}{\sqrt{8z^2+1}}e^{in\zeta_+(z,\theta)}e^{-im\zeta_-(z,\theta^{\prime})}\\&
        \cdot\bigl(\langle\hat{P}(n)\vert\hat{R}_{(>p,>p-1)}(z)\vert\hat{P}(k)\rangle\hat{q}(k)\langle\hat{P}(k)\vert\hat{R}(z)\vert\hat{P}(m)\rangle\hat{q}(m)\hat{\eta}^{(0)}(z,\theta^{\prime})\bigr)_{21},
    \end{split}
\end{align}
where $n_2\ge{p}$, $k_2\le{p-1}$, and $m_2\ge{p+1}$. By using
\eqref{resolvent_estimate_for_R(z)}, \eqref{q_n}, and
\eqref{estimate_for_<P(n)|R(z)|P(m)>q(m)_eta}, we have
\begin{align}
    \label{asymptotic_<P(n)|R_>p,>p-1(z)|P(k)>q(k)<P(k)|R(z)|P(m)>q(m)_eta}
    \begin{split}
        &\bigl(\langle\hat{P}(n)\vert\hat{R}_{(>p,>p-1)}(z)\vert\hat{P}(k)\rangle\hat{q}(k)\langle\hat{P}(k)\vert\hat{R}(z)\vert\hat{P}(m)\rangle\hat{q}(m)\hat{\eta}^{(0)}(z,\theta^{\prime})\bigr)_{21}\\
        &\sim{\cal{O}}(\langle{z}\rangle^{-2d_{21}(n-k)-2})\hat{q}_1(k){\cal{O}}(\langle{z}\rangle^{-2d(k-m)-1})\hat{q}_1(m)\\
        &\quad+{\cal{O}}(\langle{z}\rangle^{-2d_{21}(n-k)-2})\hat{q}_1(k){\cal{O}}(\langle{z}\rangle^{-2d_{12}(k-m)-2})\hat{q}_2(m)\frac{\overline{\alpha}(\zeta_-(z,\theta^{\prime}))}{\sqrt{8z^2+1}}\\
        &\quad+{\cal{O}}(\langle{z}\rangle^{-2d(n-k)-1})\hat{q}_2(k){\cal{O}}(\langle{z}\rangle^{-2d_{21}(k-m)-2})\hat{q}_1(m)\\
        &\quad+{\cal{O}}(\langle{z}\rangle^{-2d(n-k)-1})\hat{q}_2(k){\cal{O}}(\langle{z}\rangle^{-2d(k-m)-1})\hat{q}_2(m)\frac{\overline{\alpha}(\zeta_-(z,\theta^{\prime}))}{\sqrt{8z^2+1}}.
    \end{split}
\end{align}
From Lemma \ref{lemma_d_ij(n)>=|n_2|}, we have
\begin{equation}
    \label{d_21(n-k)=>|n_2-k_2|:5.1.2.1}
    d_{21}(n-k)\ge\vert{n_2-k_2}\vert,
\end{equation}
\begin{equation}
    \label{d_12(k-m)=>|k_2-m_2|:5.1.2.1}
    d_{12}(k-m)\ge\vert{k_2-m_2}\vert,
\end{equation}
\begin{equation}
    \label{d_21(k-m)=>|k_2-m_2|-1:5.1.2.1}
    d_{21}(k-m)\ge\vert{k_2-m_2}\vert-1,
\end{equation}
since $n_2-k_2>0$ and $k_2-m_2<0$; also,
\begin{equation}
    \label{d(n-k)=>|n_2-k_2|:5.1.2.1}
    d(n-k)\ge\vert{n_2-k_2}\vert,
\end{equation}
\begin{equation}
    \label{d(k-m)=>|k_2-m_2|:5.1.2.1}
    d(k-m)\ge\vert{k_2-m_2}\vert.
\end{equation}
Then, by using \eqref{asymptotic_of_adjoint_alpha_zeta_-/sqrt_8z^2+1},
\eqref{d_21(n-k)=>|n_2-k_2|:5.1.2.1}, \eqref{d_12(k-m)=>|k_2-m_2|:5.1.2.1},
\eqref{d_21(k-m)=>|k_2-m_2|-1:5.1.2.1}, \eqref{d(n-k)=>|n_2-k_2|:5.1.2.1}, and
\eqref{d(k-m)=>|k_2-m_2|:5.1.2.1}, we have
${\cal{O}}(N^{-2\vert{n_2-k_2}\vert-2\vert{k_2-m_2}\vert-1})$ as the
asymptotic behavior of
\eqref{asymptotic_<P(n)|R_>p,>p-1(z)|P(k)>q(k)<P(k)|R(z)|P(m)>q(m)_eta}.
Moreover, by using \eqref{asymptotic_of_alpha_zeta_+/sqrt_8z^2+1},
\eqref{asymptotic_e^{zeta_+}}, and \eqref{asymptotic_e^{zeta_-}}, we have
$V_{1,2,1}\sim{\cal{O}}(N^{2(n_2+m_2)-2\vert{n_2-k_2}\vert-2\vert{k_2-m_2}\vert})$.
The estimate \eqref{2(n_2-m_2)-2|n_2-k_2|-2|k_2-m_2|=<4p} implies
$V_{1,2,1}\sim{\cal{O}}(N^{4p-4})$.

$V_{1,2,2}$ is a linear combination of
\eqref{asymptotic_alpha_e_e(<P(n)|R_>p,>p-1(z)|P(k)>q(k)<P(k)|R(z)|P(m)>q(m)_eta)_21}'s,
where $n_2\ge{p}$, $k_2=p$, and $m_2\ge{p+1}$. Then, by noticing
$\hat{q}_2(k)=0$ for $k_2=p$, we have $V_{1,2,2}\sim{\cal{O}}(N^{4p-2})$ in
the same way as $V_{1,2,1}$.


\bibliography{paper}
\end{document}